\documentclass[12pt]{amsart}

\usepackage[T1]{fontenc}
\usepackage{imakeidx}
\usepackage{yfonts}
\usepackage{amssymb}
\usepackage{fancyhdr}
\usepackage{amsmath}
\usepackage{amsfonts}
\usepackage{slashed} 
\usepackage[utf8]{inputenc}
\usepackage{mathrsfs}
\usepackage[all]{xy} 
\usepackage{graphicx} 
\usepackage{xcolor}
\usepackage{enumitem}
\usepackage{glossaries}

\newcommand{\disk}{\ensuremath{\mathbb{D}} } 
\newcommand{\sphere}{\bar{\Bbb{C}}} 
\newcommand{\riem}{\Sigma}  
\renewcommand{\Bbb}[1]{\ensuremath{\mathbb{#1}}}

\newcommand{\R}{\mathbb{R}}

\newglossaryentry{bergman}{%
name=\ensuremath{\mathcal{A}},
    description={Bergman space}
}

\newglossaryentry{Abw}{%
name=\ensuremath{\mathcal{A}_\mathrm{bw}},
    description={bridgeworthy harmonic oneforms}
}

\newglossaryentry{Aharm}{%
name=\ensuremath{\mathcal{A}_{\mathrm{harm}}},
    description={Harmonic Bergman space}
}

\newglossaryentry{Ahm}{%
name=\ensuremath{\mathcal{A}_{\mathrm{hm}}},
    description={Complex linear span of harmonic measures}
}

\newglossaryentry{exactA}{%
name=\ensuremath{\mathcal{A}^\mathrm{e}},
    description={Space of exact forms}
}

\newglossaryentry{seform}{%
name=\ensuremath{\mathcal{A}^{{\mathrm{se}}}},
    description={Semi-exact forms}
}

\newglossaryentry{peform}{%
name=\ensuremath{\mathcal{A}_{\mathrm{harm}}^{\mathrm{pe}}},
    description={Piecewise exact harmonic forms}
}

\newglossaryentry{annulus}{%
name=\ensuremath{\mathbb{A}_{a,b}},
    description={Annulus with inner radius $a$ and outer radius $b$}
}

\newglossaryentry{gota}{%
name=\ensuremath{ \mathrm{\textgoth{A}} },
    description={Forms with prescribed periods}
}

\newglossaryentry{Bphi}{%
name=\ensuremath{ \mathbf{B}(\phi)},
    description={Boundary map}
}

\newglossaryentry{cf}{%
name=\ensuremath{\mathbf{C}_{f}},
    description={right-composition with $f$}
}

\newglossaryentry{cl}{%
name=\ensuremath{\text{cl}},
    description={Closure of a set}
}

\newglossaryentry{D}{%
name=\ensuremath{\mathcal{D}},
    description={Dirichlet space}
}

\newglossaryentry{Dbw}{%
name=\ensuremath{\mathcal{D}_\mathrm{bw}},
    description={bridgeworthy harmonic functions}
}

\newglossaryentry{Dir}{%
name=\ensuremath{\mathbf{Dir}},
    description={Solution map to the Dirichlet problem}
}

\newglossaryentry{Dharm}{%
name=\ensuremath{\mathcal{D}_{\mathrm{harm}}},
    description={Harmonic Dirichlet space}
}

\newglossaryentry{Dhom}{%
name=\ensuremath{\dot{\mathcal{D}}},
    description={Dirichlet space modulo constants}
}

\newglossaryentry{harmeasure}{%
name=\ensuremath{d\omega_{k}},
    description={Harmonic measure}
}

\newglossaryentry{E}{%
name=\ensuremath{\mathbf{E}},
    description={Data to solution map}
}

\newglossaryentry{greenb}{%
name=\ensuremath{G_\Sigma},
    description={Green's function of a bordered Riemann surface $\Sigma$}
}
\newglossaryentry{greenc}{%
name=\ensuremath{\mathscr{G}},
    description={Green's function of a compact Riemann surface}
}

\newglossaryentry{bounce}{%
name=\ensuremath{\mathbf{G}_{U,\riem}},
    description={Bounce operator}
}

\newglossaryentry{grunsk}{%
name=\ensuremath{\mathbf{Gr}_{f}},
    description={Grunsky operator}
}

\newglossaryentry{sobolev}{%
name=\ensuremath{H^s},
    description={Sobolev space}
}
\newglossaryentry{homsobolev}{%
name=\ensuremath{\dot{H}^s},
    description={Homogeneous Sobolev space}
}

\newglossaryentry{sobolevconf}{%
name=\ensuremath{{H}^{1}_{\mathrm{conf}}},
    description={Conformal Sobolev space}
}

\newglossaryentry{Dbvaluescomp}{%
name=\ensuremath{\mathcal{H}'(\partial_k \riem)},
    description={Dirichlet boundary values for one forms}
}

\newglossaryentry{Dbvalues}{%
name=\ensuremath{\mathcal{H}'(\partial\riem)},
    description={Dirichlet boundary values for one forms}
}

\newglossaryentry{BVexactcomp} {%
name=\ensuremath{\dot{H}'(\partial_k \riem)},
    description={Boundary values with exact representative}
}

\newglossaryentry{BVexact} {%
name=\ensuremath{\dot{H}'(\partial \riem)},
    description={Boundary values with exact representative}
}

\newglossaryentry{crop}{%
name=\ensuremath{\mathbf{J}_{1}^{q}},
    description={Cauchy-Royden operator}
}
    
\newglossaryentry{rcrop}{%
name=\ensuremath{\mathbf{J}_{1,k}^{q}},
    description={Restricted Cauchy-Royden operator}
}
\newglossaryentry{Jdot}{%
name=\ensuremath{\dot{\mathbf{J}}_1},
    description={Cauchy-Royden operator on $\dot{\mathcal{D}}$}
}

\newglossaryentry{bergmank}{%
name=\ensuremath{K},
    description={Bergman kernel}
}
\newglossaryentry{schifferk}{%
name=\ensuremath{L},
    description={Schiffer kernel}
}
\newglossaryentry{rest}{%
name=\ensuremath{\mathbf{R}},
    description={Restriction operator}
}

\newglossaryentry{harmrest}{%
name=\ensuremath{\mathbf{R}^{\mathrm{h}}},
    description={Harmonic restriction operator}
}

\newglossaryentry{S}{%
name=\ensuremath{\mathbf{S}},
    description={The Schiffer comparison operator}
}
  
\newglossaryentry{Sharm}{%
name=\ensuremath{\mathbf{S}_{k}^{\mathrm{h}}},
    description={Harmonic Schiffer operator}
}
  
\newglossaryentry{theta}{%
name=\ensuremath{\Theta},
    description={Map}
}

\newglossaryentry{Tmix}{%
name=\ensuremath{\mathbf{T}_{\riem_{j},\riem_{k}}},
    description={Schiffer operator}
}
\newglossaryentry{T}{%
name=\ensuremath{\mathbf{T}},
    description={Schiffer comparison operator}
}

\newglossaryentry{overfare}{%
name=\ensuremath{\mathbf{O}},
    description={Overfare operator}
}
\newglossaryentry{doto}{%
name=\ensuremath{\dot{\mathbf{O}}},
    description={Overfare operator on $\dot{\mathcal{D}}$}
}

\newglossaryentry{exacto}{%
name=\ensuremath{\mathbf{O}^{\mathrm{e}}_{2,1}},
    description={Exact overfare operator}
}
    
\newglossaryentry{ohat}{%
name=\ensuremath{\hat{\mathbf{O}}},
    description={Operator}
}

\newglossaryentry{augo}{%
name=\ensuremath{\mathbf{O}^{\mathrm{aug}}},
    description={Augmented overfare operator}
}
    
\newglossaryentry{oprime}{%
name=\ensuremath{\mathbf{O}'},
    description={Operator}
}
    
\newglossaryentry{oprimedot}{%
name=\ensuremath{\dot{\mathbf{O}}'},
    description={Operator}
}
    
\newglossaryentry{pcap}{%
name=\ensuremath{\mathbf{P}_{\mathrm{cap}}}, description={Projection operator}
}

\newglossaryentry{period}{%
name=\ensuremath{\mathbf{\Upsilon}},
    description={Period map}
}

\theoremstyle{plain}
        \newtheorem{theorem}{Theorem}[section]
        \newtheorem{lemma}[theorem]{Lemma}
        \newtheorem{proposition}[theorem]{Proposition}
        \newtheorem{corollary}[theorem]{Corollary}

\theoremstyle{definition}
        \newtheorem{definition}[theorem]{Definition}

\theoremstyle{remark}
    \newtheorem{remark}[theorem]{Remark}

\numberwithin{equation}{section} 
\numberwithin{figure}{section} 

\usepackage{fullpage}
\author[E. Schippers]{Eric Schippers}
\author[W. Staubach]{Wolfgang Staubach}

\address{\newline
       Eric Schippers \newline
       Machray Hall, Dept. of Mathematics,
   University of Manitoba, \newline Winnipeg, MB
   Canada R3T 2N2}
       \email{eric.schippers@umanitoba.ca}

\address{\newline
       Wolfgang Staubach \newline
       Department of  Mathematics, Uppsala University, \newline
       S-751 06 Uppsala, Sweden}
       \email{wulf@math.uu.se}
       

 
 \title{Scattering theory on Riemann surfaces II: the scattering matrix and generalized period mappings}

\makenoidxglossaries{}
\begin{document}
\begin{abstract}
We construct a scattering theory for harmonic one-forms on Riemann surfaces, obtained from boundary value problems {involving} systems of curves and the jump problem. We obtain an explicit expression for the scattering matrix in terms of integral operators which we call Schiffer operators, and show that the matrix is unitary.  We also obtain a general association of positive polarizing Lagrangian spaces to bordered Riemann surfaces, which unifies the classical polarizations for compact surfaces of algebraic geometry with the infinite-dimensional period map of the universal Teichm\"uller space. 
\end{abstract}

\maketitle

\tableofcontents
\begin{section}{Introduction}
\begin{subsection}{Overview}

 This paper is the second of two papers investigating  the global analysis of a conformally invariant scattering process on compact Riemann surfaces $\mathscr{R}$ divided into two pieces $\riem_1$ and $\riem_2$ by a collection of quasicircles. At least one of the two pieces may have finitely many connected components. Given an $L^2$ harmonic one-form $\alpha_2$ on $\riem_2$, we obtain an $L^2$ harmonic one-form $\alpha_1$ on $\riem_1$ with the same boundary values as $\alpha_2$ on the shared boundary. It is necessary to specify cohomological data.  We call $\alpha_1$ the overfare of $\alpha_2$. The analytic theory of overfare for functions was established by the authors in \cite{Schippers_Staubach_scattering_I}, and for one-forms in \cite{Schippers_Staubach_scattering_II}.  

 We obtain a scattering matrix relating the holomorphic and anti-holomorphic parts of $\alpha_1$ and $\alpha_2$,  together with the holomorphic and anti-holomorphic parts of a one-form on $\mathcal{R}$ which specifies the cohomological data. The scattering matrix has a simple and elegant form in terms of integral operators of Schiffer, and this matrix is unitary. The entire formulation is conformally invariant.
 These results rely on adjoint identities for the Schiffer operators, their interaction with cohomology, and characterizations of their kernels and images, obtained in Part I \cite{Schippers_Staubach_scattering_III}. 
 
 This is applied to define and interpret generalizations of the period map to bordered surfaces. 
 These generalizations are operators on function spaces of forms, and we illustrate how it recovers the classical period mapping. We also derive a generalized Grunsky inequality which takes into account cohomological data. This contains the classical Grunsky inequality and generalization to Riemann surfaces with $g$ handles and $n$ borders due to M. Shirazi \cite{Shirazi_thesis,Shirazi_Grunsky}. We also show its relation to a holomorphic boundary value problem for one-forms. 
  
\end{subsection}
\begin{subsection}{Scattering matrix and unitarity}

 We define a scattering process for one-forms in the following way. The overfare process defined above for functions uniquely defines the overfare of exact one-forms from connected surfaces to arbitrary ones, by 
 \begin{align*}
  \mathbf{O}^{\mathrm{e}}_{\riem_1,\riem_2}: \mathcal{A}^{\mathrm{e}}_{\mathrm{harm}}(\riem_1) & \rightarrow \mathcal{A}^{\mathrm{e}}_{\mathrm{harm}}(\riem_2) \\
  \alpha & \mapsto d \mathbf{O}_{\riem_1,\riem_2} d^{-1}
 \end{align*}
 where $\mathbf{O}_{\riem_1,\riem_2}$ is overfare of harmonic functions and $\mathcal{A}_{\mathrm{harm}}(\riem)$ denotes $L^2$ harmonic one-forms on $\riem$.  For arbitrary one-forms on a connected surface, we specify the cohomological data as follows: let $\zeta \in \mathcal{A}_{\mathrm{harm}}(\mathscr{R})$ be a one-form such that $\alpha - \zeta$ is exact on $\riem_1$.  We seek a one-form with the same boundary values as $\alpha$ and in the cohomology class of $\zeta$ on $\riem_2$. This form is
 \[  \mathbf{O}^{\mathrm{e}}_{\riem_1,\riem_2} \left( \alpha - \left. \zeta \right|_{\riem_1} \right) + \left. \zeta \right|_{\riem_2}.   \]
 We call $\zeta$ a ``catalyzing form'', and forms which are related by overfare via $\zeta$ compatible.
 
 From this overfare process we define a scattering operator which takes the holomorphic parts of the compatible forms, together with the anti-holomorphic part of the catalyzing forms, and produces the anti-holomorphic parts of the compatible forms and the holomorphic part of the catalyzing form. The anti-holomorphic parts can be thought of as left moving waves, while the holomorphic parts can be thought of as right moving waves.  
 
 We give an explicit form for the scattering matrix in terms of the Schiffer operators, using the identities and cohomological results obtained in \cite{Schippers_Staubach_scattering_III}. We furthermore show that this scattering matrix is unitary. 

 We remark that this paper is one part of a longer work \cite{Schippers_Staubach_scattering_arxiv} establishing the scattering theory of one-forms on Riemann surfaces, which we have divided into four parts. The first two papers \cite{Schippers_Staubach_scattering_I,Schippers_Staubach_scattering_II} established the analytic basis for the scattering process in terms of boundary value problems for $L^2$ harmonic functions and one-forms, while the third paper \cite{Schippers_Staubach_scattering_III} is part I to the present paper, whose contents are described above. We intend for these four papers to pave the way for the investigation of geometric and representation theoretic aspects of Teichm\"uller theory and conformal field theory. 
\end{subsection}
\begin{subsection}{Polarizations and Grunsky operators}

 For context, we sketch the well-known classical polarization for compact surfaces. Given a compact Riemann surface $\mathscr{R}$, by the Hodge decomposition theorem, every $L^2$ one-form has a harmonic representative. The spaces of harmonic one-forms in turn decompose into the spaces of holomorphic and anti-holomorphic one-forms. Thus the cohomology classes of a Riemann surface are represented by the direct sum of the vector spaces of holomorphic and  anti-holomorphic one-forms. This decomposition depends on the complex structure. 
 
In complex algebraic geometry, this picture is often represented in terms of the so-called period-matrix. Given a basis of the homology, divided into $a$ and $b$ curves satisfying the usual intersection conditions, one normalizes half of the periods of the holomorphic one-forms, and encoding the remaining periods in a $g \times g$ matrix where $g$ is the genus. Most often one normalizes matrix of $a$ periods to be the identity matrix; in that case, by the Riemann bilinear relations, the matrix of $b$ periods lies in the Siegel upper half-space of symmetric matrices with positive definite imaginary part. 
 It is also possible to represent the periods with a matrix of norm less than one (that is, a matrix in the Siegel disk). It was shown by {{L. Ahlfors \cite{Ahlfors}}} that the period matrix can be used to give coordinates on Teichm\"uller space; the idea of using periods as coordinates on the moduli space goes back to B. Riemann \cite{Riemann}.\\
 
 An analogue of the period map exists for the case of the Teichm\"uller space of the disk. Nag and Sullivan \cite{NS}, following earlier work of A. Kirillov and D. Yuri'ev in the smooth case \cite{KY2}, showed that the set of quasisymmetries of the circle acts symplectically on the space of polarizations of the set of Dirichlet-bounded harmonic functions on the disk, and that the space of polarizations can be identified with an infinite-dimensional Siegel disk. They further outlined various analogies with the classical period matrix. L. Takhtajan and L-P. Teo \cite{Takhtajan_Teo_Memoirs} showed that this ``period matrix'' is in fact the Grunsky matrix, and proved that the period map is a holomorphic map of the Teichm\"uller space of the unit disk (which is also the universal Teichm\"uller space). Later, with Radnell, the authors generalized this holomorphicity to genus zero surfaces with $n$ boundary curves. All of these results demonstrate the existence of a powerful analogy with the classical period matrix. Nevertheless they do not indicate the mathematical source of the analogy, nor how to unify the classical case for compact surfaces and the case of surfaces with border. \\

 For genus zero surfaces with $n$ boundary curves, we showed with Radnell that the graph of the Grunsky matrix gives the boundary values 
 of the set of Dirichlet-bounded harmonic functions  curves \cite{RSS_Dirichletspace}, using overfare. This was extended by M. Shirazi  \cite{Shirazi_thesis,Shirazi_Grunsky} to the genus $g$ case.
 In this paper, we show that by treating polarizations as decompositions of boundary values of semi-exact one-forms, all the versions of the polarizations can be viewed as special cases of a single general theorem. The unifying principle is provided by boundary values of harmonic one-forms.
 In particular, 
 we show that the polarizing subspace of holomorphic one-forms on a bordered surface can be viewed as the graph of an operator in an infinite Siegel disk, from which the polarizations in both the compact case and the case of genus zero surfaces with borders can be recovered.  The overfare process is a crucial part of establishing this unified picture. 
 
 The bound on the polarizing operator can be viewed as a far-reaching generalization of the Grunsky inequalities. We also show how special cases of the Grunsky inequalities can be recovered from this one. 
 We also show the connection to solvability of a holomorphic boundary value problem for $L^2$ one-forms, and give a necessary and sufficient condition for its well-posedness which takes into account cohomology. 
\end{subsection}
\begin{subsection}{Outline}
  In Section \ref{se:preliminaries} we collect results about overfare from \cite{Schippers_Staubach_scattering_I,Schippers_Staubach_scattering_II}, and from the first part of this paper \cite{Schippers_Staubach_scattering_III}. This includes some basic definitions about quasicircles and bordered surfaces, spaces of functions, and harmonic measures and Green's function, and the Schiffer operators themselves. It also includes results about conformally invariant formulation of Sobolev spaces on Riemann surfaces, the nature of boundary values of harmonic $L^2$ functions, and boundedness of the overfare process. 

  Section \ref{se:scattering} describes the overfare process for one-forms. In Section \ref{se:subsection_overfare_oneforms}, we recall results on the $H^{-1/2}$  boundary values of $L^2$ harmonic one-forms, and in particular the fact that they are independent of the ``side'' of the quasicircle the boundary values are obtained from.  We then define the overfare process in terms of what we call ``compatible triples'' in \ref{se:overfare_one_forms_second}, specifying the cohomological data with a one-form. It also contains some theorems on uniqueness of compatible triples, as well as a geometric motivation for the definition.  Section \ref{se:decompositions_compatibility} contains lemmas and theorems which relate the elements of the compatible triples with the Schiffer operators and their interaction with decompositions of the one-forms with respect to holomorphicity/anti-holomorphicity and cohomology. Their derivation relies heavily on results obtained in \cite{Schippers_Staubach_scattering_III}. These technical results are the workhorses leading to the main result in Section \ref{se:scattering_subsection}: the simple form for the scattering matrix in terms of Schiffer operators, and its unitarity. Section \ref{se:scattering_analogies} is a heuristic discussion of analogies with scattering theory on the line.

  Section \ref{se:period_mapping} contains the results on the period map, Grunsky inequalities, and holomorphic boundary value problems. Section \ref{se:period_map_upsilon} defines the period mapping for Riemann surfaces with $g$ handles and $n$ borders embedded in a compact Riemann surface of genus $g$. The main result is that the period is an isomorphism onto the space of semi-exact holomorphic one-forms on the bordered surface.  In Section \ref{se:KYNS_period} we show how this period map is associated to a polarization of a fixed space of $L^2$ one-forms, and thus can be seen as a generalization of the classical period mapping.  In Section \ref{se:Grunsky} we derive the generalization of the Grunsky operator and inequalities. Finally, in Section \ref{se:holomorphic_BVP} we give the solution to the holomorphic boundary value problem. \vspace{0.5cm} 
\end{subsection}
\end{section}

{{\bf{Acknowledgements.}} The first author was partially supported by the National Sciences and Engineering Research Council of Canada. The second author is grateful to Andreas Strömbergsson for partial financial support through a grant from Knut and Alice Wallenberg Foundation. 
}
\begin{section}{Preliminaries} \label{se:preliminaries}
\begin{subsection}{About this section}

 This section gathers the definitions and basic results used throughout the paper.  This includes Dirichlet spaces of functions and Bergman spaces of forms; Riemann surfaces, their boundaries and specialized charts called collar charts; sewing; Green's functions on compact surfaces and surfaces with boundary; Sobolev spaces; and harmonic measures and boundary period matrices.\\

 {In order to make the paper self-contained, this preliminary section has to repeat material from \cite{Schippers_Staubach_scattering_III}.  Sections \ref{se:prelim_sobolev} through \ref{se:prelim_overfare_functions} summarizes work of the authors obtained in \cite{Schippers_Staubach_scattering_I,Schippers_Staubach_scattering_II}. A full treatment can be found there.  }
\end{subsection}
\begin{subsection}{Bordered surfaces and quasicircles}

 We first establish some notation and terminology regarding Riemann surfaces and curves within them. Details can be found in \cite{Schippers_Staubach_scattering_I}. 
 
In what follows we let $\mathbb{C}$ denote the complex plane, $\gls{annulus}$ denote the annulus $ \{ z;\, a<|z|<b \}$, and $\disk = \{ z \in \mathbb{C} : |z|<1 \}$ denote the unit disk in the plane.

We consider the following kind of bordered Riemann surface.
\begin{definition}\label{defn:gn-type surface}
We say that $\riem$ is a \emph{bordered Riemann surface of type} $(g,n),$ if it is bordered and the border has $n$ connected components, each of which is homeomorphic to $\mathbb{S}^1$, and its double $\riem^d$ is a compact surface of genus $2g + n-1$.  
 \end{definition}
Visually, a bordered surface of type $(g,n)$ is a $g$-handled surface bounded by $n$ simple closed curves.  The term ``border'', as opposed to the purely topological term ``boundary'', is here reserved for a boundary which is compatible with the complex structure of the interior; see e.g. \cite{Ahlfors_Sario}.  
 We order the borders and label them accordingly, so that $\partial \riem = \partial_1 \riem \cup \cdots \cup \partial_n \riem$.  The borders can be identified with analytic curves in the double $\riem^d$, and we denote the union $\riem\cup \partial \riem$ by $\text{cl}(\riem)$.  

 In the following, a Jordan curve in a Riemann surface is a homeomorphic image of $\mathbb{S}^1$.  
 A quasicircle in the complex plane $\mathbb{C}$ is a Jordan curve which is the image of the unit circle under a quasiconformal map of the plane.
 \begin{definition} 
 We say that Jordan curve $\Gamma$ of $\mathbb{S}^1$ is a \emph{quasicircle} if it is contained in an open set $U$ and there is a biholomorphism $\phi:U \rightarrow V \subseteq \mathbb{C}$ such that $\phi(\Gamma)$ is a quasicircle in the plane. 
 \end{definition}
 \begin{definition} 
 A doubly-connected chart of a quasicircle $\Gamma$ in a Riemann surface is an open neighbourhood $U$ of $\Gamma$ and a map $\phi:U \rightarrow \mathbb{A}_{r,R} $ for some annulus 
 \[  \mathbb{A}_{r,R} \subset \mathbb{C}, \ \ \  r<1<R,  \] so that $\phi(\Gamma)$ is isotopic to the circle $|z|=1.$  We call $U$ a doubly-connected neighbouhood of $\Gamma$.
 \end{definition}
 It can be shown that any quasicircle has a doubly-connected chart. 
 By shrinking $\mathbb{A}_{r,R} $, we can assume  (1) that $\phi$ extends biholomorphically to an open neighourhood of $\text{cl} \,(U)$, (2), that the boundary curves of $U$ are analytic Jordan curves, and (3) that $\Gamma$ is isotopic to each of the boundary curves (using $\phi^{-1}$ to provide the isotopy).   
  
 \end{subsection}

 \begin{subsection}{Function spaces and holomorphic and harmonic forms}
  In this paper, we will denote positive constants in the inequalities by $C$ whose 
value is not crucial to the problem at hand. The value of $C$ may differ
from line to line, but in each instance could be estimated if necessary.  Moreover, when the values of constants in our estimates are of no significance for our main purpose, then we use the notation $a\lesssim b$ as a shorthand for $a\leq Cb$. If $a\lesssim b$ and $b\lesssim a$ then we write $a\approx b.$  \\

 We define the \emph{Wirtinger operators} via their local coordinate expressions
\[   \partial f = \frac{\partial f}{\partial z}\, dz,  \ \ \ 
   \overline{\partial} f =  \frac{\partial f}{\partial \bar{z}}\, d \bar{z}. \]
 On any Riemann surface,  define the dual of the almost 
 complex structure,  $\ast$ in local coordinates $z=x+iy$,  by 
 \[  \ast (a\, dx + b \, dy) = a \,dy - b \,dx. \]
This is independent of the choice of coordinates.
It can also be computed in coordinates that for any complex function $h$ 
\begin{equation} \label{eq:Wirtinger_to_hodge}
    2 \partial_z h = dh + i \ast dh.
\end{equation}

Denote complex conjugation of functions and forms with a bar, e.g. $\overline{\alpha}$.

Denote by $L^2(U)$ the set of one-forms $\omega$ on an open set $U$ which satisfy
\[   \iint_U \omega \wedge \ast \overline{\omega} < \infty  \]
(observe that the integrand is positive at every point, as can be seen by writing the expression in local coordinates).  
This is a Hilbert space with respect to the inner product
\begin{equation} \label{eq:form_inner_product}
 (\omega_1,\omega_2) =  \iint_U \omega_1 \wedge \ast \overline{\omega_2}.
\end{equation}

\begin{definition}\label{defn:bergman spaces}
The \emph{Bergman space of holomorphic one forms} is 
\begin{equation}
    \gls{bergman}(U) = \{ \alpha \in L^2(U) \,:\, \alpha \ \text{holomorphic} \}.
\end{equation} 
 The anti-holomorphic Bergman space is denoted $\overline{\mathcal{A}(U)}$.   We will also denote 
\begin{equation}
    \gls{Aharm}(U) =\{ \alpha \in L^2(U) \,:\, \alpha \ \text{harmonic} \}.
\end{equation}
\end{definition}

Observe that $\mathcal{A}(U)$ and $\overline{\mathcal{A}(U)}$ are orthogonal with respect to the inner product \eqref{eq:form_inner_product}.  In fact we have the direct sum decomposition
\begin{equation}\label{direct sum decomposition}
 \mathcal{A}_{\mathrm{harm}}(U) = \mathcal{A}(U) \oplus \overline{\mathcal{A}(U)}.    
\end{equation}      
If we restrict the inner product to 
 $\alpha, \beta \in \mathcal{A}(U)$ then since $\ast \overline{\beta} = i \overline{\beta}$, we have   
\[  (\alpha,\beta) = i \iint_U \alpha \wedge \overline{\beta}.      \]

Denote the projections induced by this decomposition by 
\begin{align}  \label{eq:hol_antihol_projections_Bergman}
  \mathbf{P}_U: \mathcal{A}_{\mathrm{harm}}(U) & \rightarrow \mathcal{A}(U) \nonumber \\
  \overline{\mathbf{P}}_U : \mathcal{A}_{\mathrm{harm}}(U) & \rightarrow  \overline{\mathcal{A}(U)}.
\end{align}

Let $f: U \rightarrow V$ be a biholomorphism. We denote the pull-back of $\alpha \in \mathcal{A}_{\mathrm{harm}}(V)$
under $f$ by $f^*\alpha.$ 
Explicitly, if $\alpha$ is given in local coordinates $w$ by $a(w)\, dw + \overline{b(w)} \, d\bar{w}$ and $w=f(z),$ then the pull-back is given by 
\[   f^* \left( a(w)\, dw + \overline{b(w)} \,d\bar{w} \right)= a(f(z)) f'(z)\, dz + \overline{b(f(z))} \overline{f'(z)}\, d\bar{z}.   \]
The Bergman spaces are all conformally invariant, in the sense that if $f:U \rightarrow V$ is a biholomorphism, then $f^*\mathcal{A}(V) = \mathcal{A}(U)$ and this preserves the inner product.  Similar statements hold for the anti-holomorphic and harmonic spaces.\\ 

\begin{definition}\label{def: exact holo and harm forms}
 We define the space $\gls{exactA}_{\mathrm{harm}}(U)$ as the subspace of exact elements of $\mathcal{A}_{\mathrm{harm}}(U)$, and similarly for $\mathcal{A}^\mathrm{e}(\riem)$ and $\overline{\mathcal{A}^\mathrm{e}(\riem)}$.  
\end{definition}

{We also consider one-forms which have zero boundary periods, which we call semi-exact.
\begin{definition} \label{de:semi_exact}
 Let $\riem$ be a bordered surface of type $(g,n)$. We say that an $L^2$ one-form $\alpha \in \mathcal{A}_{\mathrm{harm}}(\riem)$ is semi-exact if for any simple closed curve $\gamma$ homotopic to a boundary curve $\partial_k \riem$, 
 \[ \int_{\gamma}  \alpha =0. \]
 The class of semi-exact one-forms on $\riem$ is denoted $\mathcal{A}^{\mathrm{se}}_{\mathrm{harm}}(\riem)$. The holomorphic and anti-holomorphic semi-exact one-forms are denoted by $\gls{seform}(\riem)$  and 
 $\overline{\mathcal{A}^{\mathrm{se}}(\riem)}$ respectively.
\end{definition}}

The following spaces also play significant roles in this paper.
\begin{definition}\label{defn:dirichlet spaces}
The \emph{Dirichlet spaces of functions} are defined by 
\begin{align*}
   \gls{Dharm}(U)& = \{ f:U \rightarrow \mathbb{C}, f \in C^2(U), \,:\, 
   df\in L^2 (U)\,\,\,\mathrm{and}\,\, \, d\ast df =0 \},\\
   \gls{D}(U)& = \{ f:U \rightarrow \mathbb{C} \,:\, 
    df \in \mathcal{A}(U) \}, \ \text{and} \\
    \overline{\mathcal{D}(U)} & = \{ f:U \rightarrow \mathbb{C} \,:\, 
    df \in \overline{\mathcal{A}(U)} \}. \\
\end{align*}
\end{definition}
We also consider homogeneous Dirichlet spaces $\dot{\mathcal{D}}_{\mathrm{harm}}(U)$, $\dot{\mathcal{D}}(U)$, and $\dot{\overline{\mathcal{D}}}(U)$ obtained by quotienting by the following equivalence relation: two functions are equivalent if they differ by a locally constant function. 
We can define a degenerate inner product on $\mathcal{D}_{\mathrm{harm}}(U)$ by 
\[   (f,g)_{\mathcal{D}_{\mathrm{harm}}(U)} = (df,dg)_{\mathcal{A}_{\mathrm{harm}}(U)},   \] 
where the right hand side is the inner product (\ref{eq:form_inner_product}) restricted to elements of $\mathcal{A}_{\mathrm{harm}}(U)$.  The inner product can be used to define a seminorm on $\mathcal{D}_{\mathrm{harm}}(U)$, by letting $$\Vert f\Vert^2_{\mathcal{D}_{\mathrm{harm}}(U)}:=(df,df)_{\mathcal{A}_{\mathrm{harm}}(U)}.$$

Observe that the Dirichlet spaces are conformally invariant in the same sense as the Bergman spaces.  That is, if $f: U \rightarrow V$ is a biholomorphism then 
\begin{equation*}
 \gls{cf} h = h \circ f
\end{equation*} 
satisfies
\[  \mathbf{C}_f :\mathcal{D}(V) \rightarrow \mathcal{D}(U) \]
and this is a seminorm preserving bijection. Similar statements hold for the anti-holomorphic and harmonic spaces.  On the homogeneous spaces it is an isometry. 
\end{subsection}
\begin{subsection}{Harmonic measures} \label{ harmonic measures}

We start with the definition of harmonic measure in the context of bordered Riemann surfaces. See \cite{Schippers_Staubach_scattering_II} for details. 
\begin{definition}\label{def:harmonic measure}
Let $\omega_k$, $k=1,\ldots,n$ be the unique harmonic function which is continuous 
 on the closure of $\riem$ and which satisfies
 \begin{equation*}
  \omega_k = \left\{ \begin{array}{ll} 
    1 & \mathrm{on}\,\,\, \partial_k \riem \\ 0 & \mathrm{on}\,\,\, \partial_j \riem, \ \  j \neq k.  \end{array}  \right.   
 \end{equation*}
 The one-forms $\gls{harmeasure}$ are the {\it harmonic measures}.  We denote the complex linear span of the harmonic measures by $\gls{Ahm}(\riem).$ Moreover we define $\ast \mathcal{A}_{\mathrm{hm}}(\riem) = \{ \ast \alpha : \alpha \in \mathcal{A}_{\mathrm{hm}}(\riem) \}.$
\end{definition}

By definition any element of $\mathcal{A}_{\mathrm{hm}}(\riem)$ is exact, and its anti-derivative $\omega$ is constant on each boundary curve. On the other hand, the elements of $\ast \mathcal{A}_{\mathrm{hm}}(\riem)$ are all closed but not necessarily exact. Elements of $\mathcal{A}_{\mathrm{hm}}(\riem)$ and $\ast \mathcal{A}_{\mathrm{hm}}(\riem)$ extend real analytically to the border, in the sense that they are restrictions to $\riem$ of harmonic one-forms on the double.  In particular they are square-integrable. Summarizing:
 
\begin{proposition} \label{pr:harmonic_measures_L2}
 Let $\riem$ be a bordered surface of type $(g,n)$.  Then $\mathcal{A}_{\mathrm{hm}}(\riem) \subseteq \mathcal{A}_{\mathrm{harm}}^{\mathrm{e}}(\riem)$ and 
 $\ast \mathcal{A}_{\mathrm{hm}}(\riem) \subseteq \mathcal{A}_{\mathrm{harm}}(\riem)$.  
\end{proposition} 

 \begin{definition}\label{periodmatrix}
 The \emph{boundary period matrix} $\Pi_{jk}$ of a non-compact surface $\riem$ of type $(g,n)$ is defined by
 \[  \Pi_{jk} := \int_{\partial \riem} \omega_j  \ast d\omega_k = \int_{\partial_j \riem} \ast d \omega_k.   \]
\end{definition} 
 \begin{theorem} \label{th:period_matrix_invertible} If we let $j,k$ run from $1$ to $n$, omitting one fixed value $m$ say, then the resulting matrix
  $\Pi_{jk}$ is symmetric and positive definite.  
 \end{theorem}
 
 Thus $\Pi_{jk}$, $j,k = 1,\ldots\hat{m},\ldots,n$ is an invertible matrix, and we can specify $n-1$ of the boundary periods of elements of $\ast \mathcal{A}_{\mathrm{hm}}(\riem)$.
 \begin{corollary}  \label{co:boundary_periods_specified_starmeasure}
  Let $\riem$ be of type $(g,n)$ and $\lambda_1,\ldots,\lambda_n \in \mathbb{C}$ be such that $\lambda_1 + \cdots + \lambda_n =0$.  Then there is an $\alpha \in \ast \mathcal{A}_{\mathrm{hm}}(\riem)$ such that 
  \begin{equation}\label{periodjaveln}
    \int_{\partial_k \riem} \alpha = \lambda_k   
  \end{equation}   
  for all $k=1,\ldots,n$.  
 \end{corollary}
\end{subsection}
\begin{subsection}{Green's functions} 
Another basic notion which is of fundamental importance in our investigations is that of Green's functions.
 
 \begin{definition}\label{defn:greens function} Let $\riem$ be a type $(g,n)$ surface.
For fixed $z \in \riem$, we define \emph{Green's function of} $\riem$ to be a function $\gls{greenb}(w;z)$ such that 
 \begin{enumerate}
     \item for a local coordinate $\phi$ vanishing at $z$ the function $w\mapsto G_\Sigma(w;z) + \log|\phi(w)|$ is harmonic in an open neighbourhood
     of $z$;
     \item $\lim_{w \rightarrow \zeta} G_\Sigma(w;z) =0$ for any $\zeta \in \partial \riem$.   
 \end{enumerate}
 \end{definition}
  That such a function exists, follows from \cite[II.3 11H, III.1 4D]{Ahlfors_Sario}, considering $\riem$ to be a subset of its double $\riem^d$.\\  
\begin{definition}
 For compact surfaces $\mathscr{R}$, one defines the \emph{Green's function} $\gls{greenc}$ (see e.g. \cite{Royden}) as the unique function
 $\mathscr{G}(w,w_0;z,q)$ satisfying 
 \begin{enumerate}
  \item $\mathscr{G}$ is harmonic in $w$ on $\mathscr{R} \backslash \{z,q\}$;
  \item for a local coordinate $\phi$ on an open set $U$ containing $z$, $\mathscr{G}(w,w_0;z,q) + \log| \phi(w) -\phi(z) |$ is harmonic 
   for $w \in U$;
  \item for a local coordinate $\phi$ on an open set $U$ containing $q$, $\mathscr{G}(w,w_0;z,q) - \log| \phi(w) -\phi(q) |$ is harmonic 
   for $w \in U$;
  \item $\mathscr{G}(w_0,w_0;z,q)=0$ for all $z,q,w_0$.  
 \end{enumerate}
\end{definition}

 The existence of such a function is a standard fact about Riemann surfaces, see for example \cite{Royden}.  $\mathscr{G}$ is also harmonic in $z$ wherever it is non-singular, and satisfies a number of identities.
 
 \begin{remark}
  The condition (4) involving the point $w_0$ simply determines an arbitrary additive constant, and is not of any interest in the paper. 
 This is because of the property $\mathscr{G}(w,w_1;z,q)  = \mathscr{G}(w,w_0;z,q) - \mathscr{G}(w_1,w_0;z,q)$ (see e.g. \cite{Royden}), which makes $\partial_w \mathscr{G}$ is independent
 of $w_0$, and only such derivatives enter in this paper.  For this reason, we usually leave $w_0$ out of the expression for $\mathscr{G}$.
 \end{remark}

 Green's function is conformally invariant.  That is, if $\riem$ is of type $(g,n)$, and $f:\riem \rightarrow \riem'$ is conformal, then 
 {\begin{equation} \label{eq:Greens_conf_inv_gntype}
  G_{\riem'}(f(w);f(z)) = G_{\riem}(w;z). 
 \end{equation}}
 
 Similarly if $\mathscr{R}$ is compact and $f:\mathscr{R} \rightarrow \mathscr{R}'$ is a biholomorphism, then 
 \begin{equation} \label{eq:Greens_conf_inv_compact}
   \mathscr{G}_{\mathscr{R}'}(f(w),f(w_0);f(z);f(q)) = \mathscr{G}_{\mathscr{R}}(w,w_0;z,q). 
 \end{equation}

 These follow from uniqueness of Green's function; in the case of type $(g,n)$ surfaces, one also needs the fact that a biholomorphism extends to a homeomorphism of the boundary curves. 

 \end{subsection}
\begin{subsection}{Sobolev spaces on Riemann surfaces with boundary} \label{se:prelim_sobolev}

     We shall now use the harmonic measure and Green's function to define a conformally invariant Sobolev space on Riemann surfaces with boundary. For details see \cite{Schippers_Staubach_scattering_I}.\\  
 Let $d\omega_k$ be the harmonic measures given in Definition \ref{def:harmonic measure}. For a collar neighbourhood $U_k$ of $\partial_k \riem$ and $h_k \in \mathcal{D}_{\mathrm{harm}}(U_k)$, 
 we can fix a simple closed analytic curve $\gamma_k$ which is isotopic to $\partial_k \riem$, and define
 \begin{equation}\label{defn:integration over rough}
  \int_{\partial_k \riem}  h_k \ast d \omega_k: =  \iint_{V_k} dh_k \wedge \ast d\omega_k + \int_{\gamma_k} h_k \ast d \omega_k  
 \end{equation}   
 where $V_k$ is the region bounded by $\partial_k \riem$ and $\gamma_k$.  Here $\partial_k \riem$ is oriented positively with respect to $\riem$ and $\gamma_k$ has the same orientation as $\partial_k \riem$ (this is independent of $\gamma_k$).
 Equivalently, for a curves $\Gamma_r = \phi^{-1}(|z|=r)$ defined by a collar chart $\phi$, 
 \[   \int_{\partial_k \riem}  h_k \ast d \omega_k  = \lim_{r \nearrow 1} \int_{\Gamma_r}  h_k \ast d \omega_k.  \]
Now given $h_k \in \mathcal{D}_{\mathrm{harm}}(\riem)$ we set
   \begin{equation}\label{harmonisk matt}
    \mathscr{H}_k  := \int_{\partial_k \riem} h_k \ast d\omega_k.    
   \end{equation} 
\begin{definition}\label{def:norm on H1confU}
Let $\riem$ be a bordered surface of type $(g,n)$ and 
  let $U_k \subseteq \riem$ be collar neighbourhoods of $\partial_k \riem$ for $k=1,\ldots,n$. Set $U = U_1 \cup \cdots \cup U_n$. By ${H}^{1}_{\mathrm{conf}}(U)$ we denote the harmonic Dirichlet space $\mathcal{D}_{\mathrm{harm}}(U)$ endowed with the norm
  \begin{equation}\label{defn:Dconf norm}
 \| h \|_{H^1_{\mathrm{conf}}(U)} := \Big(\| h \|^2_{\mathcal{D}_{\mathrm{harm}}(U)} + 
  \sum_{k=1}^{n}   |\mathscr{H}_k |^2\Big)^{\frac{1}{2}} 
 \end{equation} 
 for $n>1$.  
 In the case that $n=1$,  fix a point $p\in \riem \setminus U_1$ and define instead 
  \begin{equation} \label{eq:constant_hk_definition_onecurve}
    \mathscr{H}_1 := \int_{\partial_1 \riem} h_1 \ast d G_{\riem}(w,p),
  \end{equation}
  where $G_{\riem}(w,p)$ is Green's function of $\riem.$
  
   For the Riemann surface $\riem$, assuming that $\riem$ is connected, we need only one boundary integral to obtain a norm.  If $n >1$, we can choose any fixed boundary curve $\partial_n \riem$ say, and define the norm 
 \begin{equation}\label{equivalent sobolev norm}  
     \| h \|_{{H}^1_{\mathrm{conf}}(\riem)} := \left( \| h \|^2_{\mathcal{D}_{\mathrm{harm}}(\riem)} + |\mathscr{H}_n |^2 \right)^{1/2},
 \end{equation}  
where any of the $\mathscr{H}_k$ could in fact be used in place of $\mathscr{H}_n$. The resulting norms are equivalent. In the case that $n=1$ we use (\ref{eq:constant_hk_definition_onecurve}) to define $\mathscr{H}_1$.
 \end{definition}

\end{subsection}
\begin{subsection}{Boundary values of harmonic forms and functions}  \label{se:prelim_BVs_harm} In this section, we collect some results about boundary values of Dirichlet bounded functions and $L^2$ harmonic one forms, obtained in \cite{Schippers_Staubach_scattering_I,Schippers_Staubach_scattering_II}. 

We also consider a space of boundary values of Dirichet-bounded harmonic functions. In the following, let $\Gamma$ be a border of a Riemann surface $\riem$, which is homeomorphic to $\mathbb{S}^1$. For any such $\Gamma$ there is a biholomorphism 
 \[  \phi:U \rightarrow \mathbb{A}_{r,1} \]
 where $U$ is an open set in $\riem$ bounded by two Jordan curves, one of which is $\Gamma$, and $\mathbb{A}_{r,1}$ is an annulus $r<|z|<1$. We call this a collar chart of $\Gamma$. Such a $\phi$ extends homeomorphically to a map from $\Gamma$ onto $\mathbb{S}^1$; in fact, viewing $\Gamma$ as an analytic curve in the double of $\riem$, $\phi$ has a biholomorphic extension to an open neighbourhood of $\Gamma$. 
 
 Given any $h \in \mathcal{D}_{\mathrm{harm}}(\riem)$, $h$ has ``conformally non-tangential'' or ``CNT'' boundary values on $\Gamma$. That is, given a collar chart $\phi$ of $\Gamma$ , $h \circ \phi^{-1}$ has non-tangential boundary values except possibly on a Borel set of logarithmic capacity zero in $\mathbb{S}^1$; we define the CNT boundary value of $h$ at $p$ to be the non-tangential boundary value of $h \circ \phi^{-1}(p)$.  Define a null set on $\Gamma$ to be the image of a logarithmic set of capacity zero under $\phi^{-1}$. Let $\mathcal{H}(\Gamma)$ denote the set of functions on $\Gamma$ arising as CNT boundary values of elements of $\mathcal{D}_{\mathrm{harm}}(\riem)$, where we say two functions are the same if they agree except possibly on a null set.  
 It can be shown that $\mathcal{H}(\Gamma)$ is independent of the choice of collar chart.  The definition also extends immediately to finite collections of distinct borders $\Gamma = \Gamma_1 \cup \cdots \cup \Gamma_n$.  It can also be shown that $\mathcal{H}(\Gamma)$ is the set of CNT boundary values of elements of $\mathcal{D}_{\mathrm{harm}}(U)$ for any collar neighbourhood $U$ of $\Gamma$. 
 
 Treating $\Gamma$ as an analytic curve in the double, we can define the Sobolev space $H^{1/2}(\Gamma)$. Every element of $H^{1/2}(\Gamma)$ uniquely defines an element of $\mathcal{H}(\Gamma)$. 

 We also define the homogeneous space $\dot{\mathcal{H}}(\Gamma)$ of elements of $\mathcal{H}(\Gamma)$ modulo functions which are constant on each connected component of $\Gamma$. Every element of the homogeneous Sobolev space $\dot{H}^{1/2}(\Gamma)$ defines a unique element of $\dot{\mathcal{H}}(\Gamma)$.  

 The possible boundary values of $L^2$ one-forms on $\Gamma$ can be identified with $H^{-1/2}(\Gamma)$. These boundary values were also described as equivalence classes of one-forms in the following way in \cite{Schippers_Staubach_scattering_II}, where details and proofs can be found.  Given two $L^2$ harmonic one-forms $\alpha_1$ and $\alpha_2$ defined near $\Gamma$, we say that they are equivalent, $\alpha_1 \sim \alpha_2$, if and only if there is a collar neighbourhood $U$ of $\Gamma$ and a harmonic functions $h_k$, $k=1,2$ such that $dh_k = \alpha_k$ and the CNT boundary values of $h_1 - h_2$ are constant on $\Gamma$.  Then $\mathcal{H}'(\Gamma)$ is defined to be the set of equivalence classes of such $L^2$ harmonic one-forms. 
  
 We define a norm on $\mathcal{H}'(\Gamma)$ as follows. First define a norm on $\mathcal{H}'(\mathbb{S}^1)$: given any $[\alpha]$ there is a representative of the form
 \begin{equation} \label{eq:alpha_local_form}
   \alpha = f(z) dz + \overline{g(z)} d \overline{z} + \frac{\lambda}{4 \pi i} \left( \frac{dz}{z} - \frac{d\bar{z}}{\bar{z}} \right).
 \end{equation} 
 We define 
 \[  \| [\alpha] \|^2_{\mathcal{H}'(\mathbb{S}^1)} = \| f(z) dz + \overline{g(z)} d\bar{z} \|^2_{\mathcal{A}_{\mathrm{harm}}(\disk)} + |\lambda|^2. \]
  For general $\Gamma$, given a collar chart $\phi:U \rightarrow \mathbb{A}_{r,1}$, define the norm of an $[\alpha] \in \mathcal{H}'(\Gamma)$ via the pullback $(\phi^{-1})^*$: that is, 
  \[  \| [\alpha] \|^2_{\mathcal{H}'(\Gamma)} = \| [(\phi^{-1})^*\alpha] \|^2_{\mathcal{H}'(\mathbb{S}^1)}.  \]

  This norm depends on the choice of collar chart. However, it was shown that the norms induced by different collar charts are comparable. In fact they are all comparable to the $H^{-1/2}(\Gamma)$ norm under the following identification. Given $[\alpha]$ we define a linear functional 
  \begin{align*}
   L_{[\alpha]}: H^{1/2}(\Gamma) & \longrightarrow \mathbb{C} \\ 
   h & \mapsto \lim_{r \nearrow 1} \int_{\Gamma_r} \alpha H 
  \end{align*}
  where $H$ is an element of $\mathcal{D}_{\mathrm{harm}}(U)$ defined on a collar neighbourhood $U$ of $\Gamma$ with CNT boundary values $h$, $\alpha$ is a representative of $[\alpha]$, and $\Gamma_r$ are curves $\phi^{-1}(|z|=r)$ for a collar chart $\phi:U \rightarrow \mathbb{A}_{R,1}$. It can be shown that the integral is independent of all choices, and thus we have a well-defined map into the dual $H^{-1/2}(\Gamma)$ of $H^{1/2}(\Gamma)$:
  \begin{align*}
    \mathcal{H}'(\Gamma) & \longrightarrow H^{-1/2}(\Gamma) \\
    [\alpha] & \mapsto L_{[\alpha]}. 
  \end{align*} 
   With the notation above, it is possible to show that the integral
 \[  \int_{\Gamma} [\alpha] := \lim_{r \nearrow 1} \int_{\Gamma_r} \alpha \]
 is well-defined. Defining 
 \[ \dot{\mathcal{H}}'(\Gamma) = \left\{ [\alpha] \in \mathcal{H}'(\Gamma) \,:\, \int_{\Gamma} [\alpha] =0 \right\} \]
 we define the norm similarly, but dropping the $\lambda$ in the expression (\ref{eq:alpha_local_form}).  
 We then have that $L_{[\alpha]}$ is a well-defined linear functional on $\dot{H}^{-1/2}(\Gamma)$. 
  
  These identifications are bounded isomorphisms.
 \begin{theorem}  \label{th:Honehalf_reformulation}
  Let $\Gamma$ be a border of a Riemann surface which is homeomorphic to $\mathbb{S}^1$. The map 
  \begin{align*}
      \mathcal{H}'(\partial_k \riem) & \rightarrow H^{-1/2}(\partial_k \riem) \\
      [\alpha] & \mapsto L_{[\alpha]} 
  \end{align*}
  is a bounded isomorphism. Similarly 
  \begin{align*}
      \dot{\mathcal{H}}'(\partial_k \riem) & \rightarrow \dot{H}^{-1/2}(\partial_k \riem) \\
      [\alpha] & \mapsto L_{[\alpha]} 
  \end{align*}
  is a bounded isomorphism.
 \end{theorem} 
 \begin{remark} The theorem was stated for a border of a type $(g,n)$ Riemann surface, but in fact the proof only assumes that $\Gamma$ is a border homeomorphic to $\mathbb{S}^1$. 
 \end{remark} 

 If $\Gamma = \Gamma_1 \cup \cdots \cup \Gamma_n$ is a finite collection of borders, then we can define the spaces 
 \[ {\mathcal{H}}'(\Gamma) = {\mathcal{H}}'(\Gamma_1) \times \cdots \times {\mathcal{H}}'(\Gamma_n) \] 
 and
 \[ \dot{\mathcal{H}}'(\Gamma) = \dot{\mathcal{H}}'(\Gamma_1) \times \cdots \times \dot{\mathcal{H}}'(\Gamma_n) \] 
 and endow them with the direct sum norms.
\end{subsection}

\begin{subsection}{Overfare of functions} \label{se:prelim_overfare_functions}
 Given a compact Riemann surface cut into two pieces $\riem_1$ and $\riem_2$ by a collection of quasicircles, we define the following process which we call ``overfare''. Given a Dirichlet bounded harmonic function $h_1$, its overfare is the Dirichlet bounded harmonic function $h_2$ on $\riem_2$ with the same CNT boundary values as $h_1$ on $\partial \riem_1 = \partial \riem_2$.  This process is well-defined and bounded. Boundedness requires extra assumptions, which depend on whether one wants to control constants or not. 

 We now make this precise. Details and proofs can be found in \cite{Schippers_Staubach_scattering_I}. 
 Consider a compact Riemann surface $\mathscr{R}$ separated into two ``pieces'' $\riem_1$ and $\riem_2$ by a collection of quasicircles. It is possible for one or both of $\riem_1$ and $\riem_2$ to be disconnected. The precise definition of ``separating'' is the following.
  \begin{definition}  \label{de:separating_complex}
 Let $\mathscr{R}$ be a compact Riemann surface, and let $\Gamma_1,\ldots,\Gamma_m$ be a collection of  quasicircles in $\mathscr{R}$.  Denote $\Gamma = \Gamma_1 \cup \cdots \cup \Gamma_m$.  We say that $\Gamma$ \emph{separates} $\mathscr{R}$ into $\riem_1$ and $\riem_2$ if 
 \begin{enumerate}
     \item there are doubly-connected neighbourhoods $U_k$ of $\Gamma_k$ for $k=1,\ldots,n$   such that $U_k \cap U_j$ is empty for all $j \neq k$, 
     \item one of the two connected components of $U_k \backslash \Gamma_k$ is in $\riem_1$, while the other is in $\riem_2$; 
     \item $\mathscr{R} \backslash \Gamma = \riem_1 \cup \riem_2$;
     \item $\mathscr{R} \backslash \Gamma$ consists of finitely many connected components;
     \item $\riem_1$ and $\riem_2$ are disjoint.
 \end{enumerate}
 \end{definition}
 In this case we also call $\Gamma$ a separating complex of quasicircles. 
 
 As was shown in in \cite[Proposition 3.33]{Schippers_Staubach_scattering_I}, it turns out that one can identify the borders $\partial \riem_1$ and $\partial \riem_2$ pointwise with $\Gamma$. 

 Now fix a single quasicircle $\Gamma_k$ in the complex. Let $U_1$ and $U_2$ be collar neighbourhoods of $\Gamma_k$ in $\riem_1$ and $\riem_2$ respectively. In principle the set of CNT boundary values of harmonic functions in $\mathcal{D}_{harm}(U_1)$ might not agree with the set of boundary values of harmonic functions in $\mathcal{D}_{harm}(U_2)$.  In fact, it was shown in \cite{Schippers_Staubach_scattering_I} that they agree, so that we can denote this $\mathcal{H}(\Gamma_k)$.  

 More can be said.
  \begin{theorem}[Existence of overfare]\label{thm:bounded_overfare_existence}  Let $\mathscr{R}$ be a compact Riemann surface and let $\Gamma = \Gamma_1 \cup \cdots \cup \Gamma_m$ be a collection of  quasicircles separating $\mathscr{R}$ into $\riem_1$ and $\riem_2$. 
  Let $h_1 \in \mathcal{D}_{\mathrm{harm}}(\riem_1)$.  There is a $h_2 \in \mathcal{D}_{\mathrm{harm}}(\riem_2)$ whose \emph{CNT} boundary values agree with those of $h_1$ up to a null set, and this $h_2$ is unique.   
  \end{theorem}
  
    This theorem motivates the definition of the following operator which plays an important role in the scattering theory that is developed here.
  \begin{definition}\label{defn:overfare operator} 
Thus we can define the \emph{overfare operator} $\gls{overfare}_{\riem_1,\riem_2}$ by   
  \begin{align*}
   \mathbf{O}_{\riem_1,\riem_2} : \mathcal{D}_{\mathrm{harm}}(\riem_1) & \rightarrow 
   \mathcal{D}_{\mathrm{harm}}(\riem_2) \\
   h_1 & \mapsto h_2    
 \end{align*}
  \end{definition}
 One obviously has that
 \[  \mathbf{O}_{\riem_2,\riem_1} \mathbf{O}_{\riem_1,\riem_2} = \text{Id} \]
 and of course one can switch the roles of $\riem_1$ and $\riem_2$.  
 
 The overfare operator is conformally invariant.  That is, if $f:\mathscr{R} \rightarrow \mathscr{R}'$ is a biholomorphism and we set $f(\riem_k) = \riem_k'$ for $k=1,2$ then 
 \begin{equation} \label{eq:overfare_conformally_invariant}
  \mathbf{O}_{\riem_1,\riem_2} \mathbf{C}_f = \mathbf{C}_f \mathbf{O}_{\riem_1',\riem_2'}.   
 \end{equation} 
 We will sue the abbreviated notation $\mathbf{O}_{1,2}$ and $\mathbf{O}_{2,1}$ when the underlying surfaces are clear from context.

 This operator is bounded under certain conditions \cite{Schippers_Staubach_scattering_I}. We summarize the main results here. One must assume that the originating surface is connected. 
 \begin{theorem}[Bounded overfare theorem for general quasicircles]\label{thm:bounded_overfare_Dirichlet} Let $\mathscr{R}$ be a compact Riemann surface and let $\Gamma = \Gamma_1 \cup \cdots \cup \Gamma_m$ be a collection of quasicircles separating $\mathscr{R}$ into $\riem_1$ and $\riem_2$.  Assume that $\riem_1$ is connected. 
  There is a constant $C$ such that 
     \[  \| \mathbf{O}_{1,2} h \|_{\mathcal{D}_{\mathrm{harm}}(\riem_2)}  \leq C \| h \|_{\mathcal{D}_{\mathrm{harm}}(\riem_1)}  \]
     for all $h \in \mathcal{D}_{\mathrm{harm}}(\riem_1)$.  
  \end{theorem}

 If $\riem_1$ is connected and $c$ is a constant, then $\mathbf{O}_{\riem_1,\riem_2} c$ is also constant on $\riem_2$ so the operator
 \begin{equation} \label{eq:O_dot_definition}
 \gls{doto}_{\riem_1,\riem_2} : \dot{\mathcal{D}}_{\mathrm{harm}}(\riem_1) \rightarrow  \dot{\mathcal{D}}_{\mathrm{harm}}(\riem_2) 
 \end{equation}
 is well-defined.  

 \begin{remark}
 If the originating surface (say $\riem_1$) is not connected, this is not well-defined. For example, if $\riem_1$ consists of $n$ disks and $\riem_2$ is connected, then by choosing $h$ to be constant on the connected components of $\riem_1$, the overfare $\dot{\mathbf{O}}_{1,2} h$ can be an arbitrary element of $\mathcal{A}_{\mathrm{harm}}(\riem_2)$.  
 \end{remark} 
 \begin{corollary}
  Let $\mathscr{R}$ be a compact Riemann surface, separated by quasicircles into $\riem_1$ and $\riem_2$.  Assume that $\riem_1$ is connected.  Then $\dot{\mathbf{O}}_{\riem_1,\riem_2}$ is bounded with respect to the Dirichlet norm.  
 \end{corollary}

 For more regular quasicircles, constants in the overfare process can be controlled.   
 A BZM (bounded zero mode) quasicircle $\gamma$ in $\mathscr{R}$ are quasicircles with the following property. Given an open neighbourhood $U$ of $\gamma$ and a conformal map $\phi:U \rightarrow V$ where $V$ is a doubly-connected domain in $\sphere$, let $\gamma'=\phi(\gamma)$ and denote the connected components of the complement by $\Omega_1$ and $\Omega_2$.  A BZM quasicircle is one such that overfare $\mathbf{O}_{\Omega_1,\Omega_2}$ is bounded from $H^1_{\mathrm{conf}}(\Omega_1)$ to $H^1_{\mathrm{conf}}(\Omega_2)$.  For such quasicircles one can remove the assumption of connectedness. 

 \begin{theorem}[Bounded overfare theorem for BZM quasicircles]\label{thm:bounded_overfare_conf}  Let $\mathscr{R}$ be a compact Riemann surface and let $\Gamma = \Gamma_1 \cup \cdots \cup \Gamma_m$ be a collection of \emph{BZM} quasicircles separating $\mathscr{R}$ into $\riem_1$ and $\riem_2$. 
  There is a constant $C$ such that 
     \[  \| \mathbf{O}_{1,2} h \|_{H^1_{\mathrm{conf}}(\riem_2)}  \leq C \| h \|_{H^1_{\mathrm{conf}}(\riem_1)}  \]
     for all $h \in \mathcal{D}_{\mathrm{harm}}(\riem_1)$.  
  \end{theorem}
\end{subsection}
\begin{subsection}{Definitions of Schiffer operators}\label{subsec:defSchiffer} 
Denote by $\mathscr{G}$ Green's function of $\mathscr{R}$, and let $G_{\Sigma_k}$ be Green's functions of $\riem_k$, $k=1,2$.   Here, if $\riem_k$ has more than one connected component, then $
G_{\Sigma_k}$ stands for the function whose restriction to each connected component is the Green's function of that component.

  First, we define the Schiffer operators.  
 To that end, we need to define certain bi-differentials, which will be the integral kernels of the Schiffer operators. 
 
 \begin{definition}
  For a compact Riemann surface $\mathscr{R}$ with Green's function $\mathscr{G}(w,w_0;z,q),$ the \emph{Schiffer kernel} is defined by
 \[  \gls{schifferk}_{\mathscr{R}}(z,w) =  \frac{1}{\pi i} \partial_z \partial_w \mathscr{G}(w,w_0;z,q),    \]
 and the \emph{Bergman kernel} is given by
 \[  \gls{bergmank}_{\mathscr{R}} (z,w) = - \frac{1}{\pi i} \partial_z \overline{\partial}_{{w}} \mathscr{G}(w,w_0;z,q).    \]

 For a non-compact surface $\riem$ of type $(g,n)$ with Green's function $g_\riem$, we define
 \[  L_\riem(z,w) =   \frac{1}{\pi i} \partial_z \partial_w G_\riem(w,z),  \]
 and 
 \[  K_\riem (z,w) = - \frac{1}{\pi i} \partial_z \overline{\partial}_{{w}}G_\riem(w,z).  \]
 \end{definition}

  The kernel functions satisfy the following:
 \begin{enumerate} 
  \item[$(1)$] $L_{\mathscr{R}}$ and $K_{\mathscr{R}}$ are independent of $q$ and $w_0$.  
  \item[$(2)$] $K_{\mathscr{R}}$ is holomorphic in $z$ for fixed $w$, and anti-holomorphic in $w$ for fixed $z$.
  \item[$(3)$] $L_{\mathscr{R}}$ is holomorphic in $w$ and $z$, except for a pole of order two when $w=z$.
\item[$(4)$] $L_{\mathscr{R}}(z,w)=L_{\mathscr{R}}(w,z)$.  
  \item[$(5)$] $K_{\mathscr{R}}(w,z)= - \overline{K_{\mathscr{R}}(z,w)}$.  
 \end{enumerate}

 Also, the well-known reproducing property of the Bergman kernel holds, i.e.
 \begin{equation}  \label{eq:Bergman_reproducing}
  \iint_\riem K_\mathscr{R}(z,w) \wedge h(w) = h(z),
 \end{equation}
 for $h \in A(\mathscr{R})$ \cite{Royden}.

 \begin{definition}\label{def: restriction ops}
For $k=1,2$ define the \emph{restriction operators}
 \begin{align*}
  \mathbf{R}_{\riem_k}:\mathcal{A}(\mathscr{R}) & \rightarrow \mathcal{A}(\riem_k) \\
  \alpha & \mapsto \left. \alpha \right|_{\riem_k}
 \end{align*}
 and 
 \begin{align*}
  \mathbf{R}^0_{\riem_k}: \mathcal{A}(\riem_1 \cup \riem_2) & \rightarrow \mathcal{A}(\riem_k) \\
    \alpha & \mapsto \left. \alpha \right|_{\riem_k}.
 \end{align*}
 \end{definition}
 
 It is obvious that these are bounded operators. In a similar way, we also define the restriction operator
 \[  \gls{harmrest}_{\riem_k} : \mathcal{A}_{\mathrm{harm}}(\mathscr{R}) \rightarrow \mathcal{A}_{\mathrm{harm}}(\riem_k).   \] 

We can now define the Schiffer operators as follows.
\begin{definition}
For $k=1,2$, we define the {\it Schiffer comparison operators} by
 \begin{align*}
  \gls{T}_{\riem_k}: \overline{\mathcal{A}(\riem_k)} & \rightarrow \mathcal{A}(\riem_1 \cup \riem_2)  \\
  \overline{\alpha} & \mapsto \iint_{\riem_k} L_{\mathscr{R}}(\cdot,w) \wedge \overline{\alpha(w)}.
 \end{align*}
\\
and 
\begin{align*}
 \gls{S}_{\riem_k}: \mathcal{A}(\riem_k) & \rightarrow \mathcal{A}(\mathscr{R}) \\
  \alpha & \mapsto \iint_{\riem_k}  K_{\mathscr{R}}(\cdot,w) \wedge \alpha(w).
\end{align*}
The integral defining $\mathbf{T}_{\riem_k}$ is interpreted as a principal value integral whenever $z \in \riem_k$.  
Also, we define for $j,k \in \{1,2\}$
 \begin{equation}\label{defn: T sigmajsigmak}
     \gls{Tmix} = \mathbf{R}^0_{\riem_k} \mathbf{T}_{\riem_j}: \overline{\mathcal{A}(\riem_j)} \rightarrow \mathcal{A}(\riem_k). 
 \end{equation}   
\end{definition} 
 \begin{theorem} \label{th:Schiffer_operators_bounded}  $\mathbf{T}_{\riem_k}$, $\mathbf{T}_{\riem_j,\riem_k}$, and $\mathbf{S}_{\riem_k}$ are bounded for all $j,k =1,2$.  
 \end{theorem}
  
  We will use the notations 
 \[  \mathbf{S}_k, \ \ \mathbf{T}_{j,k}, \ \ \mathbf{T}_k, \ \ \mathbf{R}_{k}, \ \  \mathbf{P}_k=\mathbf{P}_{\riem_k}, \]
 wherever the choice of surfaces $\riem_1$ and $\riem_2$ is clear from context.\\
 
For any operator $\mathbf{M}$, we define the complex conjugate operator by 
 \[  \overline{\mathbf{M}} \overline{\alpha} = \overline{\mathbf{M} \alpha}  \]
 So for example 
 \[   \overline{\mathbf{T}}_{1,2}:\mathcal{A}(\riem_1) \rightarrow \overline{\mathcal{A}(\riem_2)} \]
 and similarly for $\overline{\mathbf{R}}_{\riem_k}$, etc.  
 
All the operators considered in this section are conformally invariant. Explicitly, if $f:\mathscr{R} \rightarrow \mathscr{R}'$ is a biholomorphism between compact surfaces, and we denote
 $\riem_k'= f(\riem_k)$ for $k=1,2$, then 
 \begin{align}  \label{eq:Schiffer_operators_conformally_invariant}
   f^* \; \mathbf{R}_{\riem_k'} & = \mathbf{R}_{\riem_k} \; f^* \nonumber \\
   f^*  \; \mathbf{R}^0_{\riem_k'} & = \mathbf{R}^0_{\riem_k} \; f^* \nonumber \\
   f^* \; \mathbf{T}_{\riem_k'} & = \mathbf{T}_{\riem_k}\; f^* \\
   f^* \; \mathbf{T}_{\riem_j',\riem_k'} & = \mathbf{T}_{\riem_j,\riem_k} \; f^* \nonumber \\ 
   f^* \; \mathbf{S}_{\riem_k'} & = \mathbf{S}_{\riem_k} \; f^*. \nonumber
 \end{align}

\end{subsection}

\end{section}

\begin{section}{Scattering}   \label{se:scattering}

\begin{subsection}{About this section}

In this section we assume that $\mathscr{R}$ is a compact surface separated into two pieces by a complex $\Gamma = \Gamma_1 \cup \cdots \cup \Gamma_n$ of quasicircles into two pieces  $\riem_1$ and $\riem_2$, in the sense of Definition \ref{de:separating_complex}. Other assumptions on the surface will be explicitly added when necessary. 

 In Sections \ref{se:subsection_overfare_oneforms} and \ref{se:overfare_one_forms_second} we define the overfare process for one-forms and motivate it geometrically.
 The remaining Sections \ref{se:decompositions_compatibility} and \ref{se:scattering_subsection} are devoted to the construction of the scattering matrix associated to the overfare of one-forms defined in Section \ref{se:subsection_overfare_oneforms}, and to the proof of its unitarity.  We give an explicit form in terms of the integral operators $\mathbf{T}_{1,k}$ and $\mathbf{S}_k$ of Schiffer. 

 Since Sections \ref{se:decompositions_compatibility} and \ref{se:scattering_subsection} are somewhat technical, we give an outline of how the proofs proceed.
 Unitarity of the scattering matrix follows fairly directly from the adjoint identities derived in part I \cite{Schippers_Staubach_scattering_III}.  The difficulty is to explicitly find the form of the scattering matrix and prove that it indeed has this form.  
 We do this in three main steps.  The first two steps are completed in Section \ref{se:decompositions_compatibility}, and the final step is completed in Section \ref{se:scattering_subsection}.

 The first step involves decomposing arbitrary harmonic one-forms on $\riem_2$ in terms of (a) restrictions of harmonic forms on $\mathscr{R}$, (b) forms in the image of $\mathbf{T}_{1,2}$ and $\overline{\mathbf{T}}_{1,2}$, and (c) harmonic measures on $\riem_2$.  Similar decompositions are given on $\riem_1$. They are motivated by the results of Part I \cite[Section 4]{Schippers_Staubach_scattering_III}, which showed the interrelation between the cohomology of forms  in the range of the operator $\mathbf{T}_{1,2}$ applied to the restrictions of forms on $\mathscr{R}$, and the cohomology of forms in the range of the adjoints of the restriction operator. The decompositions are by necessity somewhat intricate, but they also have a certain elegance and inevitability.

 Once this decomposition is given, in step two we apply the jump formula and cohomology identities to express a special case of the overfare of one-forms $\mathbf{O}^{\mathrm{e}}_{2,1}$ (where the superscript $e$ stands for ``exact'') in terms of the restriction operators, their adjoints, and the Schiffer operators $\mathbf{T}_{1,1}$.   These formulas show in particular that a natural  general overfare process produces compatible forms. This also proves that this overfare process satisfies one of the nine block-matrix unitarity relations.

 Section \ref{se:scattering_subsection} contains the third step in the proof of the unitarity of the scattering matrix, using the adjoint identities of Part I \cite{Schippers_Staubach_scattering_III} as well as the decompositions for the form and its overfare in \ref{se:decompositions_compatibility}.  Finally, in Section \ref{se:scattering_analogies} we give a heuristic discussion of the interpretation of our matrix as a scattering matrix. 
\end{subsection}
 \begin{subsection}{Boundary values of one-forms from two sides of a quasicircle}  \label{se:subsection_overfare_oneforms}

  In this section and the next we describe the overfare of one-forms. Given a one-form $\alpha_1 \in \mathcal{A}_{\mathrm{harm}}(\riem_1)$, we seek a one-form $\alpha_2 \in \mathcal{A}_{\mathrm{harm}}(\riem_2)$ with the same boundary values as $\alpha_1$. There are two issues to take care of. 

  {\bf (1)}  Given two $L^2$ harmonic one-forms $\alpha_1,\alpha_2$ defined on each side of a quasicircle $\gamma \in \mathscr{R}$, what does it mean to say that $\alpha_1$ and $\alpha_2$ have the same boundary values? 

  {\bf(2)} There are many pairs of one-forms $\alpha_1$ and $\alpha_2$ with the same boundary values on $\Gamma = \Gamma_1 \cup \cdots \cup \Gamma_n$. So we need to provide cohomological data to make this well-defined. 

  Issue (1) is dealt with in this section, and issue (2) in the next section.  Briefly, the problem is as follows. An $L^2$ harmonic one-form defined on a collar on one side (in $\riem_1$ say) {of a boundary curve $\Gamma_k=\partial_k \riem_1 = \partial_k \riem_2$ has boundary values in $H^{-1/2}(\partial_k \riem_1)$. } It must be established that these boundary values are also in $H^{-1/2}(\partial_k \riem_2)$.  
  
 To deal with this, we use the corresponding fact for functions: for any quasicircle $\gamma$ the possible CNT boundary values $\mathcal{H}(\gamma)$ of Dirichlet bounded harmonic functions is the same from either side. Then, we will apply the duality of boundary values of one-forms and functions.  In both cases, we call the process of transfer of boundary values from one side of the curve to the other (without necessarily specifying a particular local extension) ``partial overfare''. 

  We first define this partial overfare on functions, {and then proceed to partial overfare on one-forms. These were first defined in \cite{Schippers_Staubach_scattering_II}, where details and proofs can be found}. {Let $U$ be a doubly-connected neighbourhood of $\partial_k \riem_1$, and let $U_1$ and $U_2$ denote the components of $U \backslash \partial_k \riem_1$ in $\riem_1$ and $\riem_2$ respectively.} For any extension $H_1 \in \mathcal{D}_{\mathrm{harm}}(U_1)$ to an open neighbourhood $U_1$ whose CNT boundary values equal $h$, let $H_2$ {be any element of $\mathcal{D}_{\mathrm{harm}}(U_2)$ whose CNT boundary values agree, and let $h_2$ be those CNT boundary values.  
  } We set
  \begin{align*}
   \mathbf{O}(\partial_k \riem_1,\partial_k \riem_2):H^{1/2}(\partial_k \riem_1) & \rightarrow H^{1/2}(\partial_k \riem_2) \\
   h_1 & \mapsto h_2.
  \end{align*}  
  We define $\mathbf{O}(\partial_k \riem_2,\partial_k \riem_1)$ similarly. {The spaces $H^{1/2}(\partial_k \riem_1)$ is defined as a space on the ideal boundary of $\riem_1$, and similarly $H^{1/2}(\partial_k \riem_2)$ is defined on the ideal boundary of $\riem_2$. Thus these spaces cannot be identified a priori.}
  {\begin{remark} An explicit $H_2$ can be found as follows. Let $\phi:U \rightarrow \mathbb{C}$ be a biholomorphism defined in an open neighbourhood of $\partial_k \riem$ in $\mathscr{R}$. One can show that the boundary values of $H_1$ on $\phi(\partial_k \riem_1$) have an extension to an element of the Dirichlet space of the Jordan domain containing $\phi(U_1)$. This then has a unique overfare, say $\hat{H}_2$, to the complementary Jordan domain in the sphere. Set $H_2 = \hat{H}_2 \circ \phi^{-1}$.
  \end{remark}}
  
  This agrees with overfare $\mathbf{O}_{1,2}$ in the following sense \cite[Proposition 4.1]{Schippers_Staubach_scattering_II}.
  \begin{proposition} \label{pr:local_overfare_functions_well_defined}
   Given $h \in H^{1/2}(\partial_k \riem_1)$, 
   let $H$ be any element of $\mathcal{D}_{\mathrm{harm}}(\riem_1)$ whose \textnormal{CNT} boundary values equal $h$ on $\partial_k \riem_1$. Then the boundary values of $\mathbf{O}_{1,2} H$ equal $\mathbf{O}(\partial_k \riem_1,\partial_k \riem_2)h$.  
  \end{proposition}
    Thus $\mathbf{O}(\partial_k \riem_1,\partial_k \riem_2)$ is independent of the choice of extension $H_1$ and doubly-connected chart. 
   
   The definition of partial overfare can be stated succinctly as follows. We have that $\mathcal{H}(\partial_k \riem_1) = \mathcal{H}(\partial_k \riem_2)$.  Given $h_1 \in H^{1/2}(\partial_k \riem_1)$, there is a unique $\tilde{h} \in \mathcal{H}(\partial_k \riem_1) = \mathcal{H}(\partial_k \riem_2)$ . Then $\tilde{h}$ agrees with a unique element $h_2 \in H^{1/2}(\partial_k \riem_2)$, and we can set
   \[  h_2 = \mathbf{O}(\partial_k \riem_1,\partial_k \riem_2)h_1.  \] 
   \begin{remark}
   The fact that the identification $\mathcal{H}(\partial_k \riem_1) = \mathcal{H}(\partial_k \riem_2)$ can be made requires careful analysis of the null sets, and requires that the curves are quasicircles \cite{Schippers_Staubach_scattering_I}. 
   \end{remark}

   We define the partial overfare of a one-form by dualizing the partial overfare of functions, using the identification of boundary values of $L^2$ harmonic one forms with $H^{-1/2}$. Recall that $H^{-1/2}(\partial_k \riem_m)$ is defined by treating $\partial_k \riem_m$ as an analytic curve in the double of $\riem_m$, and therefore we must distinguish $H^{-1/2}(\partial_k \riem_1)$ from $H^{-1/2}(\partial_k \riem_2)$.
   
   Let $L \in H^{-1/2}(\partial_k \riem_1)$. We define 
   \[  \mathbf{O}'(\partial_k \riem_1,\partial_k \riem_2):H^{-1/2}(\partial_k  \riem_1) \rightarrow H^{-1/2}(\partial_k \riem_2) \]
   by 
   \[   [\gls{oprime}(\partial_k \riem_1,\partial_k \riem_2) L](h) = -L(\mathbf{O}(\partial_k \riem_2,\partial_k \riem_1)h)  \ \ \ \text{for all} \ h \in H^{1/2}(\partial_k \riem_1).  \]
   $\mathbf{O}'(\partial_k \riem_2,\partial_k \riem_1)$ is defined similarly.
   \begin{remark}
    The negative sign is introduced in order to take into account the change of orientation of the boundary.
   \end{remark}
   
   We also define
   \[ \gls{oprimedot}_{1,2}:\dot{H}^{-1/2}(\partial_k \riem_1) \rightarrow \dot{H}^{-1/2}(\partial_k \riem_2)  \]
   and
   \[  \dot{\mathbf{O}}'_{2,1}:\dot{H}^{-1/2}(\partial_k \riem_2) \rightarrow \dot{H}^{-1/2}(\partial_k \riem_1)  \]
   in the same way. It is easily verified that these are well-defined.  
   
   We collect some {further} necessary facts from \cite{Schippers_Staubach_scattering_II}. 
   \begin{proposition} \label{pr:overfare_prime_constant} For any $L \in H^{-1/2}(\partial_k \riem)$, 
     \[ [\mathbf{O}'(\partial_k \riem_1,\partial_k \riem_2)L] (1) = - L(1). \]
   \end{proposition} 
   Under certain conditions, partial overfare is bounded.
   \begin{proposition} \label{pr:Hminus_onehalf_local_over_bounded}
   The following statements are valid:\\
   
    \emph{(1)} The partial overfare $\gls{oprimedot}(\partial_k \riem_1,\partial_k \riem_2)$ is bounded  as a map from $\dot{H}^{-1/2}(\partial_k \riem_1)$ to $\dot{H}^{-1/2}(\partial_k \riem_2)$.\\ 
    
    \emph{(2)} If $\partial_k \riem$ is a \emph{BZM} quasicircle, then $\mathbf{O}'(\partial_k \riem_1,\partial_k \riem_2)$ is bounded  as a map from ${H}^{-1/2}(\partial_k \riem_1)$ to ${H}^{-1/2}(\partial_k \riem_2)$.
   \end{proposition}

 The association between $H^{-1/2}(\partial_k \riem_m)$ and $\mathcal{H}'(\partial_k \riem_m)$ given by Theorem \ref{th:Honehalf_reformulation} immediately defines a bounded overfare
 \[  \mathbf{O}'(\partial_k \riem_1,\partial_k \riem_2) :\mathcal{H}'(\partial_k \riem_1) \rightarrow \mathcal{H}'(\partial_k \riem_2) \]
 and similarly for the homogeneous spaces
 \[  \dot{\mathbf{O}}'(\partial_k \riem_1,\partial_k \riem_2) :\dot{\mathcal{H}}'(\partial_k \riem_1) \rightarrow \dot{\mathcal{H}}'(\partial_k \riem_2) \]
 
 We will use the same notation for the overfares on $H^{-1/2}(\partial_k \riem_m)$ and $\mathcal{H}'(\partial_k \riem_m)$.
 
 The partial overfare preserves periods:
 \begin{proposition} \label{pr:overfare_preserves_periods}
  For any $k =1,\ldots,n$ and $[\alpha] \in \mathcal{H}'(\partial_k \riem_1)$ we have that 
  \[  \int_{\partial_k \riem_2}  \mathbf{O}'(\partial_k \riem_1,\partial_k \riem_2) [\alpha] = - \int_{\partial_k \riem_1} [\alpha].  \]
  The same claim holds with the roles of $1$ and $2$ switched.
 \end{proposition} 
  
  The partial overfare agrees with analytic continuation, when there is one.
   \begin{proposition}  \label{pr:extendible_forms_overfare_themselves}
    Let $U$ be a doubly-connected neighbourhood of $\partial_k \riem_1=\partial_k \riem_2$. 
    
     \emph{(1)} For any  
    $\alpha \in \mathcal{A}^{\mathrm{e}}_{\mathrm{harm}}(U)$ we have
    \[   \mathbf{O}'(\partial_k \riem_1,\partial_k \riem_2) [\alpha]=[\alpha] \]
    { where the equality above is in $\dot{H}^{-1/2}(\partial_k \riem_2)$.}
    
    \emph{(2)} If $\partial_k \riem_1$ is a \emph{BZM} quasicircle, then for any  
    $\alpha \in \mathcal{A}_{\mathrm{harm}}(U)$ we have
    \[   \mathbf{O}'(\partial_k \riem_1,\partial_k \riem_2) [\alpha]=[\alpha]. \]
   \end{proposition}
   In other words, one-forms which extend harmonically across a border are their own partial overfare.  

 \end{subsection}
 \begin{subsection}{Overfare of one-forms} \label{se:overfare_one_forms_second}
  In the previous section we saw how to transfer boundary values of a one-form from one side of a curve to the other.  
  We now return to the second issue: that of specifying the cohomological data in the overfare process. 
  We do this via a harmonic form on the compact surface $\mathscr{R}$, which we call a ``catalyzing form''.   

  As with the case of overfare of functions, overfare of forms involves the solution of a Dirichlet problem on the target surface using data on the originating surface. 
  In \cite{Schippers_Staubach_scattering_II} it was shown that the Dirichlet problem for $L^2$ harmonic one-forms is well-posed for Riemann surfaces of finite genus with a finite number of borders homeomorphic to $\mathbb{S}^1$, with boundary data in $H^{-1/2}$ and sufficient cohomological data.  We recall the result here. 
   Fix a connected Riemann surface $\riem$ of genus $g$ with $n$ borders homeomorphic to $\mathbb{S}^1$. Let
 \[   \{ \gamma_1,\ldots,\gamma_{2g}, \partial_1 \riem,\ldots,\partial_{n-1} \riem \}  \]
 be a set of generators for the fundamental group of $\riem$.    
 
 \begin{definition}[$H^{-1/2}$ Dirichlet problem for one-forms]\label{defn:Hminus12_dirichlet_data}
  We say that a harmonic one-form $\alpha$ on $\riem$  solves the $H^{-1/2}$ Dirichlet problem with $H^{-1/2}$ Dirichlet data $(L,\rho,\sigma)$ if
 \begin{enumerate}
     \item for $k=1,\ldots,n$, for any $h_k \in {H}^{1/2}(\partial_k \riem)$ we have
     \[  L_k(h_k) = L_{[\alpha]} h_k;   \]  
     \item for all $k=1,\ldots,2g$
     \[  \int_{\gamma_k} \alpha : = \sigma_k;  \]
     and
     \item for all $k =1,\ldots,n$  
     \[  \int_{\partial_k \riem} \ast \alpha = \rho_k;   \]
 \end{enumerate} 
 \end{definition}
 By \cite[Theorem 3.25]{Schippers_Staubach_scattering_II} this problem is well-posed. This tell us what cohomological data is necessary to specify the overfare.
  
 In this paper, to specify the extra cohomological data
  we will use a one-form on the surface $\mathscr{R}$. 
  \begin{definition} \label{de:weakly_compatible}
   Let $\alpha_k \in \mathcal{A}_{\mathrm{harm}}(\riem_k)$ for $k=1,2$, and let $\zeta \in \mathcal{A}_{\mathrm{harm}}(\mathscr{R})$.  We say that $\alpha_1$ and $\alpha_2$ are weakly compatible with respect to $\zeta$ if
   \begin{enumerate}
       \item $\mathbf{O}'(\partial_k \riem_2,\partial_k \riem_1) [\alpha_2]=[\alpha_1]$ for  $k=1,\ldots,n$, 
       \item $\alpha_k - \mathbf{R}^{\mathrm{h}}_k \zeta \in \mathcal{A}^{\mathrm{e}}_{\mathrm{harm}}(\riem_k)$ for $k=1,2$.\\
   \end{enumerate}
   
  We call $\zeta$ a \emph{weakly catalyzing one-form} for the pair $\alpha_1,\alpha_2$, and $(\alpha_1,\alpha_2,\zeta)$ a weakly compatible triple.
  \end{definition}

  It follows immediately from well-posedness of the $H^{-1/2}$ Dirichlet problem for one-forms that weakly compatible forms exist.  No connectivity assumptions are necessary.
  \begin{theorem} \label{th:weak_compatible_exists}  Let $\mathscr{R}$ be a compact Riemann surface, separated by a complex of quasicircles $\Gamma$ into Riemann surfaces $\riem_1$ and $\riem_2$. 
   Given $\alpha_2 \in \mathcal{A}_{\mathrm{harm}}(\riem_2)$ and $\zeta \in \mathcal{A}_{\mathrm{harm}}(\mathscr{R})$ such that $\alpha_2 - \mathbf{R}_1^{\textnormal h} \zeta \in \mathcal{A}^e_{\mathrm{harm}}(\riem_2)$, there is an $\alpha_1$ such that $\alpha_1$ and $\alpha_2$ are weakly compatible with respect to $\zeta$.  
  \end{theorem} 
  \begin{proof}  The quasicircles in the separating complex can be identified with the boundary curves of $\riem_1$ and $\riem_2$. We label the boundary curves so that $\Gamma_k=\partial_k \riem_1 = \partial_k \riem_2$.
  
   Fix a connected component of $\riem_1$, say $\riem_1'$,  whose boundary curves are $\partial_{m_1} \riem'_1 = \partial_{m_{n'}} \riem'_2$ say.  Let $g'$ be the genus of $\riem_1'$, and let $\gamma_{j_1},\ldots,\gamma_{j_{2g'}}$ be a collection of closed curves in $\riem_1'$ so that 
   \[  \{ \gamma_{j_1},\ldots,\gamma_{j_{2g'}},\partial_{m_1} \riem_{1}',\ldots \partial_{m_{n'}} \riem_1' \}  \]
   is a basis for the homology of $\riem_1'$. 
   Solve the Dirichlet problem on $\riem_1'$ to obtain $\alpha_1' \in \mathcal{A}_{\mathrm{harm}}(\riem_1')$, with data as follows: 
   \begin{enumerate}
   \item  for $k=1,\ldots,n'$ 
   \[  L_{m_k} = \mathbf{O}'(\partial_{m_k} \riem_2,\partial_{m_k} \riem_1) \alpha_2;    \]
   \item for $k=1,\ldots,2g'$ 
   \[  \int_{\gamma_{j_k}} \alpha'_1 : = \int_{\gamma_{j_k}} \zeta;  \]
   and
   \item the third condition in Definition \ref{defn:Hminus12_dirichlet_data} can be specified arbitrarily. 
   \end{enumerate}
   By well-posedness, a solution exists for every connected component, so we obtain $\alpha_1 \in \mathcal{A}_{\mathrm{harm}}(\riem_1)$. By construction $(\alpha_1,\alpha_2,\zeta)$ are weakly compatible.
  \end{proof}

  Assuming that the surface $\riem_2$ is connected, we can use exact overfare to find such an $\alpha_1$ given $\alpha_2$ and $\zeta$.  
  Of course $\alpha_1$ is not unique, and weak compatibility only controls conditions (1) and (2) of Definition \ref{defn:Hminus12_dirichlet_data}. Ahead we will characterize the lack of uniqueness.
  \begin{theorem} 
   Let $\mathscr{R}$ be a compact Riemann surface, separated by a complex of quasicircles $\Gamma$ into Riemann surfaces $\riem_1$ and $\riem_2$. 
   Assume that $\riem_2$ is connected, and that $\zeta \in \mathcal{A}_{\mathrm{harm}}(\mathscr{R})$ satisfies
   \[ \alpha_2 - \mathbf{R}_2^{\textnormal h} \zeta \in \mathcal{A}^{\mathrm{e}}_{\mathrm{harm}}(\riem_2).  \]
   Then if  
   \[   \alpha_1 = \mathbf{O}^{\mathrm{e}}_{21} \left( \alpha_2 - \mathbf{R}_2^{\textnormal h}\zeta\right) + \mathbf{R}_1^{\textnormal h} \zeta     \]
   then $(\alpha_1,\alpha_2,\zeta)$ is a weakly compatible triple. 
  \end{theorem}
  \begin{proof} By definition of $\mathbf{O}^{\mathrm{e}}_{21}$, condition (1) of Definition \ref{de:weakly_compatible} is satisfied. By Proposition \ref{pr:extendible_forms_overfare_themselves} condition (2) is satisfied. 
  \end{proof}
  
  We add a third condition to deal with the ambiguity. 
  \begin{definition} \label{de:strongly_compatible}
   We say that $\alpha_k \in \mathcal{A}_{\mathrm{harm}}(\riem_k)$, $k=1,2$ are compatible with respect to $\zeta \in \mathcal{A}_{\mathrm{harm}}(\mathscr{R})$ if they are weakly compatible with respect to $\zeta$, and additionally
   \begin{enumerate}\setcounter{enumi}{2}
       \item
        \[ \mathbf{S}_1^{\mathrm{h}} \alpha_1 + \mathbf{S}_2^{\mathrm{h}} \alpha_2 = \zeta.  \]
   \end{enumerate}
   In this case we say that $\zeta$ is a catalyzing form.
  \end{definition}
  \begin{remark}
   We will say that $(\alpha_1,\alpha_2,\zeta)$ is a compatible triple if $\alpha_1$ and $\alpha_2$ are compatible with respect to $\zeta$.  Also,  we will say that $\alpha_1$ is compatible with $\alpha_2$ and $\zeta$ if $(\alpha_1,\alpha_2,\zeta)$ is a compatible triple. 
  \end{remark}
  
  The third compatibility condition has a geometric interpretation. We state this now although it will not be proven until Section \ref{se:decompositions_compatibility}.  Assume that $\riem_2$ is connected.  Given $\alpha_2 \in \mathcal{A}_{\mathrm{harm}}(\riem_2)$ and $\zeta \in \mathcal{A}_{\mathrm{harm}}(\mathscr{R})$,   the data (1) and (2) of the boundary value problem of Definition \ref{defn:Hminus12_dirichlet_data} are specified, so as we have seen this is not unique. The one-form $\alpha_1$ solving the boundary value problem is only determined up to addition of a harmonic measure $d\omega \in \mathcal{A}_{\mathrm{harm}}(\riem_1)$. Ahead we will prove that condition (3) uniquely specifies $\alpha_1$, and that this choice is natural. 
  
  More precisely, assume that $\riem_2$ is connected, $\alpha_2 \in \mathcal{A}_{\mathrm{harm}}(\riem_2)$, and  
  $\zeta \in \mathcal{A}_{\mathrm{harm}}(\mathscr{R})$.  Then 
  \[   \alpha_1 = \mathbf{O}^{\mathrm{e}}_{21} \left( \alpha_2 - \mathbf{R}_2^{\textnormal h} \zeta\right) + \mathbf{R}_1^{\textnormal h} \zeta     \] 
  is the unique element of $\mathcal{A}_{\mathrm{harm}}(\riem_1)$ such that $(\alpha_1,\alpha_2,\zeta)$ is a compatible triple; see Corollary \ref{co:exact_overfare_is_strongly_compatible}.

  We conclude with two observations on perturbations of compatible triples $(\alpha_1,\alpha_2,\zeta)$.   Given two forms $\alpha_k \in \mathcal{A}_{\mathrm{harm}}(\riem_k)$ with the same boundary values, there are many catalyzing one-forms $\zeta \in\mathcal{A}_{\mathrm{harm}}(\mathscr{R})$. 

   We call a form $\alpha \in \mathcal{A}_{\mathrm{harm}}(\mathscr{R})$ ``piecewise exact'' if its restrictions to $\riem_1$ and $\riem_2$ are exact.  Denote the class of $L^2$ harmonic piecewise exact one-forms by $\mathcal{A}^{\mathrm{pe}}_{\mathrm{harm}}(\mathscr{R})$.
  \begin{proposition}  \label{pr:perturb_catalyst_by_pe}
   Let $\alpha_k \in \mathcal{A}_{\mathrm{harm}}(\riem_k)$, $k=1,2$ satisfy
   \[  \mathbf{O}(\partial_m {\riem_1},\partial_m \riem_2) [\alpha_1]=[\alpha_2]  \]
   for $k=1,2$.  There exists a {weakly }catalyzing one-form $\zeta \in \mathcal{A}_{\mathrm{harm}}(\mathscr{R})$ such that $(\alpha_1,\alpha_2,\zeta)$ is a weakly compatible triple. Furthermore, given any pair $\zeta,\zeta'$ of one-forms catalyzing the pair $\alpha_1,\alpha_2$, we have that $\zeta - \zeta'$ is piecewise exact.
  \end{proposition}
  \begin{proof}
   To prove existence, we need only choose any harmonic one-form on $\mathscr{R}$ whose periods agree with those of $\alpha_k$ on
   \[   \{ \gamma_1^k,\ldots,\gamma_{2g_k}^k, \partial_1 \riem_k,\ldots,\partial_{n-1} \riem_k \}  \]
   for $k=1,2$. This is possible because of the fact that
   \[  \int_{\partial_k \riem_1} [\alpha_1] = -\int_{\partial_k \riem_2} [\alpha_2] \]
   which follows from the condition $\mathbf{O}(\partial_m {\riem_1},\partial_m \riem_2) [\alpha_1]=[\alpha_2]$.
   
   Now let $\zeta,\zeta' \in \mathcal{A}_{\mathrm{harm}}(\mathscr{R})$ be catalyzing for the pair $\alpha_1,\alpha_2$.  Then 
   \[  \mathbf{R}^{\mathrm{h}}_k\zeta - \mathbf{R}^{\mathrm{h}}_k \zeta ' = (\mathbf{R}^{\mathrm{h}}_k\zeta- \alpha_k) - (\mathbf{R}^{\mathrm{h}}_k\zeta'- \alpha_k) \in \mathcal{A}_{\mathrm{harm}}^{\mathrm{e}}(\riem_k) \]
   for $k=1,2$, which completes the proof.
  \end{proof}
  
  Furthermore, 
  \begin{proposition}  \label{pr:perturb_both_by_harmonic_measure}
   {Assume that either $\riem_1$ or $\riem_2$ is connected. } 
   Let $\alpha_k \in \mathcal{A}_{\mathrm{harm}}(\riem_k)$ for $k=1,2$ be compatible with respect to $\zeta \in \mathcal{A}_{\mathrm{harm}}(\mathscr{R})$. Let $\omega_1$ and $\omega_2$ be harmonic functions which extend continuously to the boundary and are constant there.  Assume further that $\mathbf{O}_{1,2} \omega_1=\omega_2$. Then $\alpha_1 + d\omega_1$ and $\alpha_2 +d \omega_2$ are compatible with respect to $\zeta$. 
  \end{proposition}
  \begin{proof}
   By \cite[Proposition 3.14]{Schippers_Staubach_scattering_II} condition (1) of compatibility is satisfied by $\alpha_1 + d\omega_1$ and $\alpha_2 + d\omega_2$. The fact that (2) continues to be satisfied follows immediately from the fact that $d\omega_1$ and $d\omega_2$ are exact.
   Finally, observe that by \cite[Theorem 3.30]{Schippers_Staubach_scattering_III} 
   \[ \mathbf{S}^{\mathrm{h}}_1 (\alpha_1 + d\omega_1) + \mathbf{S}^{\mathrm{h}}_2 (\alpha_2 + d\omega_2) = \mathbf{S}^{\mathrm{h}}_1 \alpha_1 + \mathbf{S}^{\mathrm{h}}_2 \alpha_2 = \zeta,  \]
   completing the proof.
  \end{proof}
 \end{subsection}
\begin{subsection}{Decompositions of harmonic forms and compatible triples}  \label{se:decompositions_compatibility}   \ \ \ 
 
  In this section we develop a number of relations between the components of compatible triples via Schiffer operators. This necessitates the derivation of a number of decomposition lemmas. Although these are somewhat technical, their purpose is a simple geometric one: to separate out components of the one-forms according to their cohomology.  To a certain extent, the technical appearance of the lemmas is a consequence of their generality --- we only assume that one of the pieces is connected.  In the end, we obtain a simple formula, given in Corollary \ref{co:exact_overfare_is_strongly_compatible}, relating the elements of a compatible triple using the exact overfare $\mathbf{O}^e$, as promised in Section \ref{se:overfare_one_forms_second}.  This will in turn allow us to derive a very simple expression for the scattering matrix in Section \ref{se:scattering} ahead. 
  We conclude with two Corollaries extracting some of these simple overfare formulas for compatible triples,  in special cases.

  We will need the following lemma.  In its statement and proof, we suppress restriction operators to reduce clutter, since they are clear from context.  Because of the asymmetry in the conditions for $\riem_1$ and $\riem_2$, in the statements and proofs there will be repeated division into the two cases.  

  We recall the following special classes of harmonic measures defined in \cite{Schippers_Staubach_scattering_III}. These are important in characterizing the kernels of the Schiffer operators.  
  \begin{definition}  
  We say that $\omega \in \mathcal{D}_{\text{harm}}(\riem_1)$ is \emph{bridgeworthy} if
  \begin{enumerate}
   \item it is constant on each boundary curve;
   \item on any pair of boundary curves $\partial_k \riem_1$ and $\partial_m \riem_1$ that bound the same connected component of $\riem_2$, the boundary values of $\omega$ are equal.
  \end{enumerate}

  We say that $\alpha \in \mathcal{A}_{\mathrm{hm}}(\riem_1)$ is bridgeworthy if $\alpha = d\omega$ for some bridgeworthy harmonic function $\omega$.  Denote the collection of bridgeworthy harmonic functions by
 $\gls{Dbw}(\riem_1)$,  and the collection of bridgeworthy  harmonic one-forms by 
 $\gls{Abw}(\riem_1)$.  The same definitions apply to $\riem_2$.
 \end{definition}
 The name is meant to invoke the following geometric picture: $\omega$ is bridgeworthy if it has the same constant value on any pair of boundary curves which are connected by a ``bridge'' in $\riem_2$.

  \begin{lemma} \label{le:scattering_preparation}   Assume that $\riem_2$ is connected.\\

 \emph{\bf{(1) (Case of $\riem_1$).}}
   Let $\xi \in \mathcal{A}(\mathscr{R})$, $\overline{\eta} \in \mathcal{A}(\mathscr{R})$, $\alpha_1 \in \mathcal{A}(\riem_1)$, and  $\overline{\beta_1} \in \mathcal{A}(\riem_1)$.  Assume that 
  $(\alpha_1 + \overline{\beta}_1 ) - \mathbf{R}^{\mathrm{h}}_1 (\xi + \overline{\eta}) \in \mathcal{A}^\mathrm{e}_{\mathrm{harm}}(\riem_1)$.
  There are $\overline{m},\overline{s} \in \overline{\mathcal{A}(\mathscr{R})}$,
  $n,t \in \mathcal{A}(\mathscr{R})$, such that 
  \begin{align*}
      \alpha_1 - \overline{m} - t & \in \mathcal{A}^\mathrm{e}_{\mathrm{harm}}(\riem_1) \\
      \overline{\beta}_1  - n - \overline{s} & \in \mathcal{A}^\mathrm{e}_{\mathrm{harm}}(\riem_1) 
  \end{align*}
  and 
  \begin{align*}
      \overline{\eta} & = \overline{m} + \overline{s} \\
      \xi & = n + t. 
  \end{align*} 
  
    \emph{\bf{(2) (Case of $\riem_2$).}}
  Let $\xi \in \mathcal{A}(\mathscr{R})$, $\overline{\eta} \in \mathcal{A}(\mathscr{R})$, $\alpha_2 \in \mathcal{A}(\riem_2)$, and  $\overline{\beta_2} \in \mathcal{A}(\riem_2)$.  Assume that 
  $(\alpha_2 + \overline{\beta}_2 ) - \mathbf{R}^{\mathrm{h}}_2 (\xi + \overline{\eta}) \in \mathcal{A}^\mathrm{e}_{\mathrm{harm}}(\riem_2)$.  
  There are $\overline{m},\overline{s} \in \overline{\mathcal{A}(\mathscr{R})}$,
  $n,t \in \mathcal{A}(\mathscr{R})$, and $d\omega \in \mathcal{A}_{\mathrm{bw}}(\riem_2)$ such that 
  \begin{align*}
      \alpha_2 - \partial \omega - \overline{m} - t & \in \mathcal{A}^\mathrm{e}_{\mathrm{harm}}(\riem_2) \\
      \overline{\beta}_2 - \overline{\partial} \omega - n - \overline{s} & \in \mathcal{A}^\mathrm{e}_{\mathrm{harm}}(\riem_2) 
  \end{align*}
  and 
  \begin{align*}
      \overline{\eta} & = \overline{m} + \overline{s} \\
      \xi & = n + t. 
  \end{align*}
  \end{lemma}
  Note that we are not claiming any relation between the $m,n,s,t$ in parts (1) and (2).  
  \begin{proof} 
   In the proof, we will require a basis for the cohomology of $\mathscr{R}$, which we now describe.  Let $g_k$, $k=1,\, 2,$ be the genus of $\Sigma_k$. Assume that there are $p$ curves in the complex $\Gamma$, and assume that there are $q$ connected components  $\riem_1^j$, $j=1,\ldots,q$ of $\riem_1$. Let $g_1^1,\ldots,g_1^q$ be the genuses of these components, so that $g_1 = g_1^1 + \cdots g_1^q$. 
 We then have that 
 \[ g_1+g_2+(p-q) = g,  \]
 so we can define $g_d= p-q$ to be the number of ``dissected handles''.
 Choose a homology basis for $\mathscr{R}$ consisting of 
 \begin{itemize}
     \item $2g_1$ curves $C^1_1,\ldots,C^1_{2g_1}$ corresponding to the handles in $\riem_1$;
     \item $2g_2$ curves $C^2_1,\ldots,C^2_{2g_2}$ corresponding to the handles in $\riem_2$; 
     \item a collection of boundary curves $\Gamma_1,\ldots,\Gamma_{p-q}$ containing $n_j-1$ boundary curves of $\riem_1^j$ for $j=1,\ldots,q$;
     \item a collection of curves $b_1,\ldots,b_{p-q}$ encircling each dissected handle. 
 \end{itemize} 
 There are $q$ boundary curves, call them $e_1,\ldots,e_q$, which are not in the span of this basis, one for each connected component of $\riem_1$.  
 
 We first claim that 
 \begin{enumerate}[label=(\alph*),font=\upshape]
     \item given any $\gamma_1 \in \mathcal{A}_{\mathrm{harm}}(\riem_1)$, there is a $\zeta_1 \in \mathcal{A}_{\mathrm{harm}}(\mathscr{R})$ such that $\gamma_1  - \zeta_1$ is exact in $\riem_1$; and 
     \item  given any $\gamma_2 \in \mathcal{A}_{\mathrm{harm}}(\riem_2)$, there is a $\zeta_2 \in \mathcal{A}_{\mathrm{harm}}(\riem_2)$ and a $d\omega\in \mathcal{A}_{\mathrm{bw}}(\riem_2)$ such that $\gamma_2 - \ast d\omega - \zeta_2$ is exact in $\riem_2$.
 \end{enumerate}
 To show this, consider a dual basis of harmonic one-forms on $\mathscr{R}$ which we denote by $H = \{ H_{C_j^k}, H_{\Gamma_l}, { H_{b_r}} \}$ with the usual meaning. Claim (a) follows from the fact that there is a unique element 
 \[  \zeta_1 \in \mathrm{span}  
 \{ H_{C_1^1},\ldots,H_{C_{2g_1}^1},H_{\Gamma_1},\ldots,H_{\Gamma_{p-q}}  \}  \]
 such that $\gamma_1 - \zeta_1$ is exact, since the set of curves spans the homology of each connected component of $\riem_1$.  
 
 To prove claim (b), observe that 
 one may remove all the $C^2_j$ periods of $\gamma_2$ using a 
 \[  \zeta_2 \in \mathrm{span} \{ H_{C_1^2},\ldots,H_{C_{2g_2}^2} \}.     \]
 That is, we can arrange that $\gamma_2 - \zeta_2$
 has zero periods over all $\Gamma_1,\ldots,\Gamma_{p-q}$ and $C^2_1,\ldots,C^2_{2g_2}$.
 
 As observed above, there are $q$ boundary curves,  $e_1,\ldots,e_{q}$ say, which are not contained in the collection $\Gamma_1,\ldots,\Gamma_{p-q}$. Any one of these, say $e_q$, is a linear combination of the remaining curves $e_1,\ldots,e_{q-1}$.  {Let $\omega_j \in \mathcal{D}_{\mathrm {bw}}(\riem_2)$ be one on the boundary of the $j$th connected component of $\riem_1$ and $0$ on the others. Since one may specify the period of $\ast d\omega$ on $e_1,\ldots,e_{q-1}$ by \cite[Proposition 3.34]{Schippers_Staubach_scattering_III}, it is enough to show that $\{ \ast d\omega_1,\ldots,\ast d\omega_{q-1} \} \cup \{ H_{C_1^2},\ldots,H_{C_{2g_2}^2} \}$
 is linearly independent.    }
 
 Assume that 
 \[  \sum_{j=1}^{q-1} \mu_j \ast d\omega_j + \sum_{l} \lambda_l H_{C^2_l} =0   \]
 where $H_l$ range over the elements of $H$.  We must have
 \[  0= \int_{e_r}   \Big( \sum_{j=1}^{q-1} \mu_j \ast d\omega_j + \sum_{l} \lambda_l H_{C^2_l}\Big) = \mu_r \ast d \omega_r      \]
 for $r=1,\ldots,q-1$, so $\mu_r = 0$ by Theorem \ref{th:period_matrix_invertible}.  Thus 
 \[ \sum_{l} \lambda_l H_{C^2_l} =0 \]
 and the claim now follows from linear independence of elements of $\{ H_{C_1^2},\ldots,H_{C_{2g_2}^2} \}.$

 We shall first prove claim (2) of the lemma, which has the additional issue of the possibility of bridgeworthy forms to deal with. The proof of claim (1) is similar to that of (2), but without this complication. Apply claim (b) of the proof above to obtain $\zeta_2$, $\omega$, $\zeta_2'$, and $\omega'$ such that 
 \begin{align} \label{eq:remove_bridge_temp}
      \alpha_2 -  \zeta_2 - \partial \omega & \in \mathcal{A}^{\mathrm{e}}_{\mathrm{harm}}(\riem_2) \nonumber \\
     \overline{\beta}_2 - \zeta_2' - \overline{\partial} \omega' & \in \mathcal{A}^{\mathrm{e}}_{\mathrm{harm}}(\riem_2).
 \end{align}
 Here we have used the facts that 
 \begin{equation} \label{eq:partial_decomp_temp} 
   \partial \omega = 1/2 (d\omega +i \ast d \omega)  \ \ \ \text{and} \ \ \  \overline{\partial} \omega = 1/2 (d\omega - i \ast d \omega). 
 \end{equation}
 First, we will show that we may take $\overline{\partial} \omega'= \overline{\partial} \omega$ in \eqref{eq:remove_bridge_temp}.  To see this, observe that  
 \[  \alpha_2 + \overline{\beta}_2 - (\xi + \overline{\eta})   \]
 and 
 \[  \alpha_2 + \overline{\beta}_2 -  \zeta_2 - \zeta_2' - \partial \omega - \overline{\partial} \omega'  \]
 are both exact in $\riem_2$.  Subtracting we see that 
 \[  -\xi - \overline{\eta} + \zeta_2 + \zeta_2' + \partial \omega + \overline{\partial} \omega'   \]
 is exact. By the linear independence of the periods of  $\mathcal{A}_{\mathrm{harm}}(\mathscr{R})$ and $\ast d\omega$ for $\omega$ bridgeworthy established above, we must have that 
 $\partial \omega + \overline{\partial} \omega'$ is exact.  Again using \eqref{eq:partial_decomp_temp}, we see that the periods in $\riem_2$ of $\ast d\omega$ and $\ast d\omega'$ agree, so we may take $\omega'=\omega$ in \eqref{eq:remove_bridge_temp} as claimed. 
 
 Writing $\zeta_2=\overline{M}+T$ and $\zeta_2'=N + \overline{S}$ for 
   $\overline{M},\overline{S} \in \overline{\mathcal{A}(\mathscr{R})}$ and $N,T \in \mathcal{A}(\mathscr{R})$, we have thus shown 
    \begin{align*}
      \alpha_2 - \overline{M} - T - \partial \omega& \in \mathcal{A}^\mathrm{e}_{\mathrm{harm}}(\riem_2) \\
      \beta_2 - N - \overline{S} - \overline{\partial} \omega & \in \mathcal{A}^\mathrm{e}_{\mathrm{harm}}(\riem_2).
  \end{align*}  
  Note that $M$, $T$, $N$, and $S$ are not uniquely determined, and we must adjust them to complete the theorem. 
  Since $\alpha_2 + \overline{\beta}_2 - ( \xi + \overline{\eta})$ and $d\omega$ are in $\mathcal{A}^\mathrm{e}_{\mathrm{harm}}(\riem_2)$ we have 
  \[  \overline{M} + T + N + \overline{S} -(\xi + \overline{\eta}) \in \mathcal{A}^{\mathrm{e}}_{\mathrm{harm}}(\riem_2).       \]
  Define $u \in \mathcal{A}(\mathscr{R})$ and $\overline{v} \in \mathcal{A}(\mathscr{R})$ by
  \begin{align*}
      u & = N + T - \xi \\
      \overline{v} & = \overline{M} + \overline{S} - \overline{\eta};
  \end{align*}
  These satisfy
  \[  \int_{C^2_j} u = - \int_{C^2_j} \overline{v}    \]
  for $j=1,\ldots,2 g_2$. 
  Therefore if we set
  \begin{align*}
      \overline{m} & = \overline{M} - \overline{v}/2\\
      t & = T - u/2 \\
      \overline{s} & = \overline{S} -\overline{v}/2 \\
      n & = N -u/2
  \end{align*}
  it still holds that 
  \begin{align*}
    \alpha_2 - \overline{m} - t & \in \mathcal{A}^\mathrm{e}_{\text{harm}}(\riem_2)  \\
    \overline{\beta}_2 - \overline{s} -n & \in {\mathcal{A}^\mathrm{e}_{\text{harm}}(\riem_2)}.  
  \end{align*}
  Since $\overline{m} + \overline{s} = \overline{\eta}$ and $n + t = \xi$ this completes the proof of part (2) of the lemma.
  
  The proof of part (1) is identical, except that one may start directly with 
  \begin{align*}
      \alpha_1 - \overline{M}- T & \in \mathcal{A}^{\mathrm{e}}_{\mathrm{harm}}(\riem_1) \\
      \overline{\beta}_1 - N - \overline{S} & \in \mathcal{A}^{\mathrm{e}}_{\mathrm{harm}}(\riem_1).
  \end{align*}
  (Here of course the $M,T,N,S$ are not necessarily the same as those in the proof of (2).)
  \end{proof}

  By Theorem \ref{th:weak_compatible_exists}, given $\alpha_2+ \overline{\beta}_2 \in \mathcal{A}_{\mathrm{harm}}(\riem_2)$ and $\zeta = \xi + \overline{\eta} \in \mathcal{A}_{\mathrm{harm}}(\mathscr{R})$, there is a weakly compatible one-form $\alpha_1 + \overline{\beta}_1 \in \mathcal{A}_{\mathrm{harm}}(\riem_1)$ with respect to $\zeta$. As promised, we now show that there is a $\alpha_1 + \overline{\beta}_1$ which is in fact compatible, in the case that $\riem_2$ is connected. Furthermore, this compatible form is given by  
  \[  \alpha_1 + \overline{\beta}_1 = \mathbf{O}^\mathrm{e}\left[ \alpha_2 + \overline{\beta}_2 - \mathbf{R}_2 \xi - \overline{\mathbf{R}}_2 \overline{\eta} \right] +  \mathbf{R}_1 \xi + \overline{\mathbf{R}}_1 \overline{\eta}.  \]
  
  To do this, we require a decomposition for harmonic one-forms which is convenient from the point of view of the action of the Schiffer operators. This decomposition will also play a central role in the proof of the unitarity of the scattering matrix. \\
  
  \begin{lemma}[Decomposition lemma] \label{le:decomposition_lemma}
   Assume that $\riem_2$ is connected.\\ 
   
   \emph{\bf{(1) (Case of $\riem_1$).}}  
   Let $\alpha_1,\beta_1 \in \mathcal{A}(\riem_1)$ and $\xi,\eta \in \mathcal{A}(\mathscr{R})$ be such that $\alpha_1 + \overline{\beta}_1 - (\xi+\overline{\eta})$ is exact in $\riem_1$. Then there are $\tau_2,\sigma_2 \in \mathbf{R}_2 \mathcal{A}(\mathscr{R})$ and $\mu_2,\nu_2 \in [\mathbf{R}_2 \mathcal{A}(\mathscr{R}))]^\perp$  such that 
   \begin{align}  \label{eq:compatible_breakdown_A1}
      \alpha_1 - \overline{\mathbf{R}}_1 \overline{\mathbf{S}}_2 \overline{\mu}_2 - \mathbf{R}_1 \mathbf{S}_2 \tau_2 & \in \mathcal{A}^\mathrm{e}_{\mathrm{harm}}(\riem_1) \nonumber\\ 
      \overline{\beta}_1 - \mathbf{R}_1 \mathbf{S}_2 \nu_2 - \overline{\mathbf{R}}_1 \overline{\mathbf{S}}_2 \overline{\sigma}_2  & \in \mathcal{A}^\mathrm{e}_{\mathrm{harm}}(\riem_1)   
  \end{align}
  and 
 \begin{align} \label{eq:compatible_breakdown_B1}
    \overline{\mathbf{S}}_2 \overline{\mu}_2 +  \overline{\mathbf{S}}_2 \overline{\sigma}_2
   & = \overline{\eta} \nonumber  \\
    \mathbf{S}_2 \nu_2 + \mathbf{S}_2 \tau_2 & = \xi.  
 \end{align} 
 Furthermore, there are $\gamma_2, \rho_2 \in \mathcal{A}(\riem_1)$ such that
 \begin{align} \label{eq:compatible_breakdown_C1}
     \alpha_1 & = \mathbf{T}_{2,1} \overline{\gamma}_2 + \mathbf{R}_1 \mathbf{S}_2 \tau_2 \\ \nonumber
     \overline{\beta}_1 & = \overline{\mathbf{T}}_{2,1} \rho_2 + \overline{\mathbf{R}_1} \overline{\mathbf{S}}_2 \overline{\sigma}_2
 \end{align}
 where $\overline{\gamma}_k$ and $\rho_k$ are decomposed as follows:  
 \begin{align} \label{eq:compatible_breakdown_D1}
       \overline{\gamma}_2 & = - \overline{\mu}_2 + \overline{\delta}_2, \ \ \  \overline{\mu}_2 \in [\overline{\mathbf{R}}_2 \overline{\mathcal{A}(\mathscr{R})}], \ \ \  \overline{\delta}_2 \in [\overline{\mathbf{R}}_2 \overline{\mathcal{A}(\mathscr{R})}]^\perp,  \nonumber \\
       \rho_2 & = - \nu_2 + \varepsilon_2, \ \ \ \nu_2 \in \mathbf{R}_2 \mathcal{A}(\mathscr{R}), \ \ \ \varepsilon_2 \in [\mathbf{R}_2 \mathcal{A}(\mathscr{R}) ]^\perp. 
 \end{align} 
   
   \emph{\bf{(2) (Case of  $\riem_2$)}}.  Let $\alpha_2,\beta_2 \in \mathcal{A}(\riem_2)$ and $\xi,\eta \in \mathcal{A}(\mathscr{R})$ be such that $\alpha_2 + \overline{\beta}_2 - (\xi+\overline{\eta})$ is exact in $\riem_2$. Then there are $\tau_1,\sigma_1 \in \mathbf{R}_1 \mathcal{A}(\mathscr{R})$, $\mu_1,\tau_1 \in [\mathbf{R}_1 \mathcal{A}(\mathscr{R}))]^\perp$, 
   and a $d\omega \in \mathcal{A}_{\mathrm{bw}}(\riem_2)$, such that 
   \begin{align} \label{eq:compatible_breakdown_A2}
          \alpha_2 -\partial \omega - \overline{\mathbf{R}}_2 \overline{\mathbf{S}}_1 \overline{\mu}_1 - \mathbf{R}_2 \mathbf{S}_1 \tau_1 & \in \mathcal{A}^\mathrm{e}_{\mathrm{harm}}(\riem_2) \nonumber \\  
      \overline{\beta}_2 - \overline{\partial} \omega- \mathbf{R}_2 \mathbf{S}_1 \nu_1 - \overline{\mathbf{R}}_2 \overline{\mathbf{S}}_1 \overline{\sigma}_1  & \in \mathcal{A}^\mathrm{e}_{\mathrm{harm}}(\riem_2)  
   \end{align}
   and 
   \begin{align} \label{eq:compatible_breakdown_B2}
       \overline{\mathbf{S}}_1 \overline{\mu}_1 + \overline{\mathbf{S}}_1 \overline{\sigma}_1 & = \overline{\eta} \nonumber \\
       \mathbf{S}_1 \nu_1 + \mathbf{S}_1 \tau_1 & = \zeta. 
   \end{align}
   Furthermore, there are $\gamma_1,\rho_1 \in \mathcal{A}(\riem_1)$ such that 
   \begin{align} \label{eq:compatible_breakdown_C2}
         \alpha_2 & = \partial \omega + \mathbf{T}_{1,2} \overline{\gamma}_1 + \mathbf{R}_2 \mathbf{S}_1 \tau_1  \nonumber \\ 
     \overline{\beta}_2 & = \overline{\partial} \omega + \overline{\mathbf{T}}_{1,2} \rho_1 + \overline{\mathbf{R}}_2 \overline{\mathbf{S}}_1 \overline{\sigma}_1 
   \end{align}  
    where $\overline{\gamma}_1$ and $\rho_1$ are decomposed as follows:
 \begin{align} \label{eq:compatible_breakdown_D2}
       \overline{\gamma}_1 & = - \overline{\mu}_1 + \overline{\delta}_1, \ \ \  \overline{\mu}_1 \in [\overline{\mathbf{R}}_1 \overline{\mathcal{A}(\mathscr{R})}], \ \ \  \overline{\delta}_1 \in [\overline{\mathbf{R}}_1 \overline{\mathcal{A}(\mathscr{R})}]^\perp,  \nonumber \\
       \rho_1 & = - \nu_1 + \varepsilon_1, \ \ \ \nu_1 \in \mathbf{R}_1 \mathcal{A}(\mathscr{R}), \ \ \ \varepsilon_1 \in [\mathbf{R}_1 \mathcal{A}(\mathscr{R}) ]^\perp. 
 \end{align} 
  \end{lemma}
  \begin{proof}
   The claims \eqref{eq:compatible_breakdown_A1} and \eqref{eq:compatible_breakdown_B1} in part (1) follow directly from   Lemma \ref{le:scattering_preparation} part (1), using the fact that $\mathbf{S}_2 \mathbf{R}_2$ is an isomorphism by \cite[Theorem 4.2]{Schippers_Staubach_scattering_III}.  Similarly, the claims \eqref{eq:compatible_breakdown_A2} and  \eqref{eq:compatible_breakdown_B2} in part (2) follow directly from   Lemma \ref{le:scattering_preparation} part (2), using the fact that $\mathbf{S}_1 \mathbf{R}_1$ is an isomorphism by \cite[Theorem 4.2]{Schippers_Staubach_scattering_III}. 
   
   The claims \eqref{eq:compatible_breakdown_C1} and \eqref{eq:compatible_breakdown_C2} follow directly from \cite[Theorem 4.16]{Schippers_Staubach_scattering_III}. The decompositions of $\gamma_k$ and $\rho_k$ in \eqref{eq:compatible_breakdown_D1} and \eqref{eq:compatible_breakdown_D2} follow from \cite[Theorems 4.1, 4.7, 4.12]{Schippers_Staubach_scattering_III}.
  \end{proof}
  
  \begin{theorem} \label{th:compatibility_existence}  
   Assume that $\riem_2$ is connected.\\
   
   \emph{\bf{(1) (Case of $\riem_1$).}}  Let $\alpha_1+ \beta_1 \in \mathcal{A}_{\mathrm{harm}}(\riem_1)$ and $\zeta = \xi + \overline{\eta} \in \mathcal{A}(\mathscr{R})$ be such that $\alpha_1 + \overline{\beta}_1 - \mathbf{R}_1^{\textnormal h} \zeta \in \mathcal{A}^{\mathrm{e}}_{\mathrm{harm}}(\riem_1)$. 
   There is a $\alpha_2 + \overline{\beta}_2 \in \mathcal{A}_{\mathrm{harm}}(\riem_2)$ which is compatible with $\alpha_1 + \overline{\beta}_1$ {with respect to} $\zeta$.  Given any other compatible $\alpha_2'+\overline{\beta}_2'$, the difference $\alpha_2'+ \overline{\beta}_2' - \alpha_2 +\overline{\beta}_2 \in \mathcal{A}_{\mathrm{bw}}(\riem_2)$. 
   
   Furthermore, if $\rho_2$, $\gamma_2$, $\tau_2$, and $\sigma_2$ are given as in \emph{Lemma \ref{le:decomposition_lemma} part (1)}, then there is a $d\omega \in \mathcal{A}_{\mathrm{bw}}(\riem_2)$ such that  
   \begin{align} \label{eq:compatibility_theorem_overfared_1}
    \alpha_2 & = -\rho_2 + \mathbf{T}_{2,2} \overline{\gamma}_2 + \mathbf{R}_2 \mathbf{S}_2 \tau_2 + \partial \omega \nonumber \\
    \overline{\beta}_2 & = - \overline{\gamma}_2 + \overline{\mathbf{T}}_{2,2} \rho_2 + \overline{\mathbf{R}}_2 \overline{\mathbf{S}}_2 \overline{\sigma}_2 + \overline{\partial} \omega. 
   \end{align}

   \emph{\bf{(2) (Case of $\riem_2$).}} Let $\alpha_2+ \beta_2 \in \mathcal{A}_{\mathrm{harm}}(\riem_2)$ and $\zeta = \xi + \overline{\eta} \in \mathcal{A}(\mathscr{R})$ be such that $\alpha_2 + \overline{\beta}_2 - \mathbf{R}_2^{\mathrm{h}} \zeta \in \mathcal{A}^{\mathrm{e}}_{\mathrm{harm}}(\riem_2)$. 
   There is a unique $\alpha_1 + \overline{\beta}_1 \in \mathcal{A}_{\mathrm{harm}}(\riem_1)$ which is compatible with $\alpha_2 + \overline{\beta}_2$ {with respect to} $\zeta$, given by 
   \[  \alpha_1 + \overline{\beta}_1 = \mathbf{O}^{\mathrm{e}}_{2,1} \left( \alpha_2 + \overline{\beta}_2 - \mathbf{R}_2^{\mathrm{h}} \zeta\right) 
   + \mathbf{R}_1^{\mathrm{h}} \zeta.   \]
   
   Furthermore, if $\rho_1$, $\gamma_1$, $\tau_1$, and $\sigma_1$ are given as in Lemma \ref{le:decomposition_lemma} part(2), then 
   \begin{align}  \label{eq:compatibility_theorem_overfared_2}
    \alpha_1 & = -\rho_1 + \mathbf{T}_{1,1} \overline{\gamma}_1 + \mathbf{R}_1 \mathbf{S}_1 \tau_1  \nonumber \\
    \overline{\beta}_1 & = - \overline{\gamma}_1 + \overline{\mathbf{T}}_{1,1} \rho_1 + \overline{\mathbf{R}}_1 \overline{\mathbf{S}}_1 \overline{\sigma}_1. 
   \end{align}
  \end{theorem}
  \begin{proof}
    We first prove (2). By the assumptions we have the decomposition \eqref{eq:compatible_breakdown_A2} - \eqref{eq:compatible_breakdown_D2} of Lemma \ref{le:decomposition_lemma}.  
    
    Using \eqref{eq:compatible_breakdown_C2}, \eqref{eq:compatible_breakdown_B2}, and \cite[Theorem 4.1]{Schippers_Staubach_scattering_III} in that order, we see that
    \begin{align}  \label{eq:tempy_mctempface}
     \alpha_2 + \overline{\beta}_2 - \mathbf{R}_2 \xi - \overline{\mathbf{R}}_2 \overline{\eta} & = d\omega + \mathbf{T}_{1,2} \overline{\gamma}_1 + \mathbf{R}_2 \mathbf{S}_1 \tau_1 + \overline{\mathbf{T}}_{1,2} \rho_1 + \overline{\mathbf{R}}_2 \overline{\mathbf{S}}_1 \overline{\sigma}_1 - \mathbf{R}_2 \xi - \overline{\mathbf{R}}_2 \overline{\eta} \nonumber \\
     & = d\omega + \mathbf{T}_{1,2} \overline{\gamma}_1 - \mathbf{R}_2 \mathbf{S}_1 \nu_1 + \overline{\mathbf{T}}_{1,1} \rho_1 - \overline{\mathbf{R}}_2 \overline{\mathbf{S}}_1 \overline{\mu}_1 \nonumber \\ 
      & = d\omega + \mathbf{T}_{1,2} \overline{\gamma}_1 + \mathbf{R}_2 \mathbf{S}_1 \rho_1 + \overline{\mathbf{T}}_{1,1} \rho_1 + \overline{\mathbf{R}}_2 \overline{\mathbf{S}}_1 \overline{\gamma}_1.
    \end{align}
    Now since $d\omega$ is bridgeworthy, $\mathbf{O}^{\mathrm{e}}_{2,1} d\omega = 0$. Together with\cite[Proposition 4.10]{Schippers_Staubach_scattering_III} we obtain 
    \[  \mathbf{O}_{2,1}^{\mathrm{e}} \left( \alpha_2 + \overline{\beta}_2 - \mathbf{R}_2^{\mathrm{h}} \zeta \right) = -\overline{\gamma}_1 + \mathbf{T}_{1,1} \overline{\gamma}_1 + \overline{\mathbf{R}}_1 \mathbf{S}_1 \overline{\gamma}_1 
    - \rho_1 + \overline{\mathbf{T}}_{1,1} \rho_1 + \mathbf{R}_1 \mathbf{S}_1 \rho_1. \] 
    If we define  
    \[ \alpha_1 + \overline{\beta}_1 =  \mathbf{O}_{2,1}^{\mathrm{e}} \left( \alpha_2 + \overline{\beta}_2 - \mathbf{R}_2^{\mathrm{h}} \zeta \right) + \mathbf{R}_1^{\mathrm{h}} \zeta  \]
    \eqref{eq:compatible_breakdown_C2} again, we obtain that 
    $\alpha_1 + \overline{\beta}_1$ satisfies \eqref{eq:compatibility_theorem_overfared_2}. Furthermore, by construction $\alpha_1 + \overline{\beta}_1$ is weakly compatible with $\alpha_2 + \overline{\beta}_2$ {with respect to} $\zeta$. 
    
    Next we show that $\alpha_1 + \overline{\beta}_1$ is compatible.  Applying $\mathbf{S}_1$ to the expression \eqref{eq:compatibility_theorem_overfared_2} for $\alpha_1$,
    we obtain
   \begin{align} \label{eq:S_compatibility_temp_with_correction}
     \mathbf{S}_1 \alpha_1 & = - \mathbf{S}_1 \rho_1 + \mathbf{S}_1 \mathbf{T}_{1,1} \overline{\gamma}_1 + \mathbf{S}_1 \mathbf{R}_1 \mathbf{S}_1 \tau_1  & \nonumber \\
      & = - \mathbf{S}_1 \rho_1 - \mathbf{S}_2 \mathbf{T}_{1,2} \overline{\gamma}_1 + \mathbf{S}_1 \mathbf{R}_1 \mathbf{S}_1 \tau_1  &  \  \text{{\scriptsize{ \cite[Equation  3.17]{Schippers_Staubach_scattering_III}}}} \nonumber \\
      & = - \mathbf{S}_1 \rho_1 - \mathbf{S}_2 \left(  \alpha_2 - \mathbf{R}_2 \mathbf{S}_1 \tau_1 \right) + \mathbf{S}_1 \mathbf{R}_1 \mathbf{S}_1 \tau_1  &  \text{\scriptsize Equation \eqref{eq:compatible_breakdown_C2}}  \\
      & = \mathbf{S}_1 \nu_1 - \mathbf{S}_2   \alpha_2  + \left( \mathbf{S}_2 \mathbf{R}_2    + \mathbf{S}_1 \mathbf{R}_1  \right) \mathbf{S}_1 \tau_1   &  \text{\scriptsize Equation \eqref{eq:compatible_breakdown_D2}, \cite[Theorem 4.1]{Schippers_Staubach_scattering_III}} \nonumber \\
      & = - \mathbf{S}_2 \alpha_2 + \xi.
      & \text{\scriptsize Equation \eqref{eq:compatible_breakdown_B2}, \cite[Theorem 3.21]{Schippers_Staubach_scattering_III}} \nonumber
   \end{align}  
   This proves the claim.
   
   Finally, if $\alpha_1' + \overline{\beta}_1'$ is another compatible form, it is in particular weakly compatible. Thus $\alpha_1' + \overline{\beta}_1' - (\alpha_1 + \overline{\beta}_1)$ is an exact form $d\omega$ vanishing on the boundary, and hence is in $\mathcal{A}_{\mathrm{hm}}(\riem_1)$. However, since both are compatible, we also have that $\mathbf{S}_1^{\mathrm{h}} d\omega = 0$ so by \cite[Proposition 3.37]{Schippers_Staubach_scattering_III} $d\omega$ is bridgeworthy. Since $\riem_2$ is connected, the only bridgeworthy form on $\riem_1$ is $0$. This completes the proof of (2). 
   
   We now prove (1). By the assumptions we have the decomposition \eqref{eq:compatible_breakdown_A1} - \eqref{eq:compatible_breakdown_D1} of Lemma \ref{le:decomposition_lemma}.  Let $G,H \in \mathcal{D}_{\mathrm{harm}}(\riem_2)$ be such that $\partial G = \rho_2$ and $\overline{\partial} H = \overline{\gamma}_2$.  
   Such a $G$ and $H$ are guaranteed to exist by\cite[Lemma 4.6]{Schippers_Staubach_scattering_III}. Since by \cite[Theorem 3.18]{Schippers_Staubach_scattering_III}
   \[ \dot{\mathbf{J}}_{2,1} H = \dot{\mathbf{O}}_{2,1} \dot{\mathbf{J}}_{2,2} \dot{H} - \dot{\mathbf{O}}_{2,1} \dot{H} \]
   it follows using \cite[Theorem 3.8]{Schippers_Staubach_scattering_III} that
   \begin{equation} \label{eq:exact_same_boundary_values_one}
      d \mathbf{J}^q_{2,1} H = \mathbf{T}_{2,1} \overline{\gamma}_2 + \overline{\mathbf{R}}_1 \overline{\mathbf{S}}_2 \overline{\gamma}_2 \ \ \  \text{and} \ \ \  d(\mathbf{J}^q_{2,2} H - H )= - \overline{\gamma}_2 + \mathbf{T}_{2,2} \overline{\gamma}_2 + \overline{\mathbf{R}}_2 \overline{\mathbf{S}}_2 \overline{\gamma}_2 
   \end{equation}
   are exact and have primitives with the same CNT boundary values. Similarly using $G$ we obtain that
   \begin{equation} \label{eq:exact_same_boundary_values_two}
     \overline{\mathbf{T}}_{2,1} {\rho}_2 + {\mathbf{R}}_1  {\mathbf{S}}_2 {\rho}_2 \ \ \  \text{and} \ \ \   - {\rho}_2 + \overline{\mathbf{T}}_{2,2} {\rho}_2 + {\mathbf{R}}_2 {\mathbf{S}}_2 {\rho}_2 
   \end{equation}
   are exact and have primitives with the same boundary values. 
   
   Set now
   \begin{align*}
       \alpha_2 & = -\rho_2 + \mathbf{T}_{2,2} \overline{\gamma}_2 + \mathbf{R}_2 \mathbf{S}_2 \tau_2 \\
       \overline{\beta}_2 & = -\overline{\gamma}_2 + \overline{\mathbf{T}}_{2,2} \rho_2 + \overline{\mathbf{R}}_2 \overline{\mathbf{S}}_2 \overline{\mathbf{\sigma}}_2. 
   \end{align*}
   We will show that this is compatible with $\alpha_1 + \overline{\beta}_1$ with respect to $\zeta$. 
   
   To see that it is weakly compatible, it is enough to show that 
   \[  \alpha_2 + \overline{\beta}_2 - \mathbf{R}_2 \xi - \overline{\mathbf{R}}_2 \overline{\eta} \ \ \ \text{and} \ \ \ \alpha_1 + \overline{\beta}_1 - \mathbf{R}_1 \xi - \overline{\mathbf{R}}_1 \overline{\eta} \]
   are exact and have the same boundary values. We are given that the left expression is exact; a computation identical to \eqref{eq:tempy_mctempface} in part (2) shows that 
   \[  \alpha_1 + \overline{\beta}_1 - \mathbf{R}_1 \xi - \overline{\mathbf{R}}_1 \overline{\eta} = \mathbf{T}_{2,1} \overline{\gamma}_2 + \mathbf{R}_1 \mathbf{S}_2 \rho_2 + \overline{\mathbf{T}}_{2,1} \rho_2 + \overline{\mathbf{R}}_1 \overline{\mathbf{S}}_2 \overline{\gamma}_2 \]
   and we also have that
   \[  \alpha_2 + \overline{\beta}_2 - \mathbf{R}_2 \xi - \overline{\mathbf{R}}_2 \overline{\eta} = -\rho_2 + \mathbf{T}_{2,2} \overline{\gamma}_2 + \mathbf{R}_2 \mathbf{S}_2 \rho_2 - \overline{\gamma}_2 + \overline{\mathbf{T}}_{2,2} \rho_2 + \overline{\mathbf{R}}_2 \overline{\mathbf{S}}_2 \overline{\gamma}_2. \]
   Weak compatibility now follows from the fact that the left and right sides of \eqref{eq:exact_same_boundary_values_one} and \eqref{eq:exact_same_boundary_values_two} are exact and have primitives with the same boundary values.
   
   To show compatibility, we repeat the computation of \eqref{eq:S_compatibility_temp_with_correction} with the indices switched, and without the $\omega$ term.   
   
   Now any other weakly compatible form is of the form $\alpha_2 + \overline{\beta}_2 + d\omega$ for some $d\omega \in \mathcal{A}_{\mathrm{harm}}(\riem_2)$. If this is compatible, we must have that $\mathbf{S}_1^{\mathrm{h}} d\omega=0$, and thus $d\omega$ is bridgeworthy by \cite[Proposition 3.37]{Schippers_Staubach_scattering_III}. 
  \end{proof}
  
  We have thus proven the characterization of compatibility promised in Section \ref{se:subsection_overfare_oneforms}. Though it is contained in the statement of Theorem \ref{th:compatibility_existence}, it deserves to be singled out.
  \begin{corollary} \label{co:exact_overfare_is_strongly_compatible}
   Assume that $\riem_2$ is connected. Then $\alpha_1 \in \mathcal{A}_{\mathrm{harm}}(\riem_1)$ and $\alpha_2 \in \mathcal{A}_{\mathrm{harm}}(\riem_2)$ are compatible with respect to $\zeta \in \mathcal{A}_{\mathrm{harm}}(\mathscr{R})$ if and only if 
   \[ \alpha_1 = \mathbf{O}^{\mathrm{e}}_{2,1} \left(\alpha_2 - \mathbf{R}_2^{\mathrm{h}} \zeta \right) + \mathbf{R}_1^{\mathrm{h}} \zeta.  \]
  \end{corollary}

  We also have two special cases worthy of attention.
  \begin{corollary}
   Assume that both $\riem_1$ and $\riem_2$ are connected. Given $\alpha_k \in \mathcal{A}_{\mathrm{harm}}(\riem_k)$ and $\zeta \in \mathcal{A}_{\mathrm{harm}}(\mathscr{R})$. The following are equivalent.
   \begin{enumerate}[label=(\arabic*),font=\upshape]
       \item $\alpha_k$ are compatible with respect to $\zeta$;
       \item $\alpha_1 = \mathbf{O}^{\mathrm{e}}_{2,1} \left( \alpha_2 - \mathbf{R}_2^{\mathrm{h}} \zeta \right) + \mathbf{R}_1^{\mathrm{h}} \zeta$;
       \item $\alpha_2 = \mathbf{O}^{\mathrm{e}}_{1,2} \left( \alpha_1 - \mathbf{R}_1^{\mathrm{h}} \zeta \right) + \mathbf{R}_2^{\mathrm{h}} \zeta$.
   \end{enumerate}
   In particular, given $\alpha_1$, there is a unique $\alpha_2$ which is compatible with $\alpha_1$ with respect to $\zeta$.  The same claim holds with the indices $1$ and $2$ interchanged.
  \end{corollary}
  
  \begin{corollary}  Assume that the separating complex of curves consists of a single curve. Given $\alpha_k \in \mathcal{A}_{\mathrm{harm}}(\riem_k)$ for $k=1,2$, they are weakly compatible with respect to $\zeta$ if and only if they are compatible with respect to $\zeta$.  
  \end{corollary}
 \end{subsection}
 \begin{subsection}{Unitarity of the scattering matrix}
 
  \label{se:scattering_subsection}
 
 We now show that the scattering matrix of overfare is a unitary matrix whose blocks are Schiffer operators.  We divide this into cases $g \neq 0$ and $g=0$.  
 \begin{theorem}[Scattering matrix, $g \neq 0$]   \label{th:unitary_scattering_genus_nonzero}  Assume that the genus of $\mathscr{R}$ is non-zero, and that $\riem_2$ is connected.  
 
  Assume that $\alpha_1 + \overline{\beta}_1$ and $\alpha_2 + \overline{\beta}_2$ are compatible with respect to $\zeta = \xi + \overline{\eta} \in \mathcal{A}_{\mathrm{harm}}(\mathscr{R})$.  Then
  \begin{equation}  \label{eq:scattering_matrix}
     \left( \begin{array}{c} \overline{\beta}_1 \\ \overline{\beta}_2 \\ \xi \end{array} \right)
     = \left( \begin{array}{ccc} - \overline{\mathbf{T}}_{1,1} & - \overline{\mathbf{T}}_{2,1} & \overline{\mathbf{R}}_1 \\
      - \overline{\mathbf{T}}_{1,2} & - \overline{\mathbf{T}}_{2,2} & \overline{\mathbf{R}}_2 \\ \mathbf{S}_1 & \mathbf{S}_2 & 0 \end{array} \right) 
      \left( \begin{array}{c} {\alpha}_1 \\ {\alpha}_2 \\ \overline{\eta} \end{array} \right).
  \end{equation}
  This matrix is unitary.
 \end{theorem}
 \begin{proof} 
  Unitarity follows from \cite[Theorems 3.21, 3.23, 3.24]{Schippers_Staubach_scattering_III}.  So it only remains to show that the matrix equation holds.  The bottom entry of the left and right hand side are equal by part (3) of the definition of compatibility, so we need only demonstrate that the other two entries are equal.
  
  We then have that both parts of Lemma \ref{le:decomposition_lemma} hold. For $k=1,2$ let $\gamma_k$, $\rho_k$, $\tau_k$, $\sigma_k$, $\mu_k$, and $\nu_k$ be as in Lemma \ref{le:decomposition_lemma}, so that   \eqref{eq:compatible_breakdown_A1}-\eqref{eq:compatible_breakdown_D1} and \eqref{eq:compatible_breakdown_A2}-\eqref{eq:compatible_breakdown_D2} hold. 
 
 Note that $\mathbf{S}_k \nu_k = \mathbf{S}_k \rho_k$ and $\mathbf{S}_k \mu_k = \mathbf{S}_k \gamma_k$ for $k=1,2$, 
 since $\overline{\delta}_k \in [\overline{\mathbf{R}}_k \overline{\mathcal{A}(\mathscr{R})}]^\perp$ and 
 the integral kernel of $\mathbf{S}_k$ is in $\overline{\mathcal{A}(\mathscr{R})}$. Similarly  $\overline{\mathbf{S}}_k \overline{\mu}_k = \overline{\mathbf{S}}_k \overline{\gamma}_k$.

 Next, applying $\mathbf{T}_{1,1}$ to the first equation of  \eqref{eq:compatibility_theorem_overfared_2}, and inserting the second, we obtain
 \begin{align*}
     \overline{\mathbf{T}}_{1,1} \alpha_1 & = - \overline{\mathbf{T}}_{1,1} \rho_1 + \overline{\mathbf{T}}_{1,1} \mathbf{T}_{1,1} \overline{\gamma_1}  + \overline{\mathbf{T}}_{1,1} \mathbf{R}_1 \mathbf{S}_1 \tau_1 \\ 
      & = - \overline{\beta}_1 - \overline{\gamma_1} + \overline{\mathbf{R}}_1 \overline{\mathbf{S}}_1 \overline{\sigma}_1 + \overline{\mathbf{T}}_{1,1} 
      \mathbf{T}_{1,1} \overline{\gamma_1} + \overline{\mathbf{T}}_{1,1} 
      \mathbf{R}_1 \mathbf{S}_1 \tau_1.
 \end{align*}
 Now applying the first identity of \cite[Theorem 3.23]{Schippers_Staubach_scattering_III} and the first line of \cite[Theorem 3.24]{Schippers_Staubach_scattering_III} in that order, we obtain 
 \begin{align*}
     \overline{\mathbf{T}}_{1,1} \alpha_1 & = - \overline{\beta_1} + \overline{\mathbf{R}}_{1} \overline{\mathbf{S}}_1 \overline{\sigma_1} - \overline{\mathbf{T}}_{2,1} \mathbf{T}_{1,2} \overline{\gamma_1} - \overline{\mathbf{R}}_1 \overline{\mathbf{S}}_1 \overline{\gamma_1} 
     + \overline{\mathbf{T}}_{1,1} \mathbf{R}_1 \mathbf{S}_1 \tau_1 & \\ 
       & = - \overline{\beta}_1 + \overline{\mathbf{R}}_1 \overline{\mathbf{S}}_1 \overline{\sigma_1} - \overline{\mathbf{T}}_{2,1} \mathbf{T}_{1,2} \overline{\gamma_1} - \overline{\mathbf{R}}_{1} \overline{\mathbf{S}}_1 \overline{\gamma_1} - \overline{\mathbf{T}}_{2,1} \mathbf{R}_2 \mathbf{S}_1 \tau_1 & \\
       & =  - \overline{\beta}_1 - \overline{\mathbf{T}}_{2,1} \alpha_2 + \overline{\mathbf{T}}_{2,1} \partial \omega +
       \overline{\mathbf{R}}_1 \overline{\mathbf{S}}_1 \overline{\sigma}_1 + \overline{\mathbf{R}}_1 \overline{\mathbf{S}}_1  \overline{\mu}_1 & {\text{\scriptsize Eqn \eqref{eq:compatible_breakdown_C2},\eqref{eq:compatible_breakdown_D2}, \cite[Theorem 4.1]{Schippers_Staubach_scattering_III}  }} \\
       & =  - \overline{\beta}_1 - \overline{\mathbf{T}}_{2,1} \alpha_2 + \overline{\mathbf{R}}_1 \overline{\eta} + \overline{\mathbf{T}}_{21} \partial \omega & {\text{\scriptsize Eqn \eqref{eq:compatible_breakdown_B2}}}.
 \end{align*}
 Rearranging and using the fact that $\overline{\mathbf{T}}_{2,1} \partial \omega=0$ by \cite[Corollary 3.31]{Schippers_Staubach_scattering_III}, we get
 \begin{equation} \label{eq:full_unitary_good2}
     \overline{\beta}_1 = - \overline{\mathbf{T}}_{1,1} \alpha_1  - \overline{\mathbf{T}}_{2,1} \alpha_2 + \overline{\mathbf{R}}_1 \overline{\eta}  
 \end{equation}
 as desired.
 
 The proof of the remaining equation is similar, but there are small differences arising from the asymmetry of the assumptions. Applying $\mathbf{T}_{2,2}$ to the first equation of  \eqref{eq:compatibility_theorem_overfared_1}, and inserting the second, we obtain
  \begin{align*}
     \overline{\mathbf{T}}_{2,2} \alpha_1 & = - \overline{\mathbf{T}}_{2,2} \rho_2 + \overline{\mathbf{T}}_{2,2} \mathbf{T}_{2,2} \overline{\gamma_2}  + \overline{\mathbf{T}}_{2,2} \mathbf{R}_2 \mathbf{S}_2 \tau_2 + \overline{\mathbf{T}}_{2,2} \partial \omega\\ 
      & = - \overline{\beta}_2 - \overline{\gamma_2} + \overline{\mathbf{R}}_2 \overline{\mathbf{S}}_2 \overline{\sigma}_2 + \overline{\mathbf{T}}_{2,2} 
      \mathbf{T}_{2,2} \overline{\gamma_2} + \overline{\mathbf{T}}_{2,2} 
      \mathbf{R}_2 \mathbf{S}_2 \tau_2
 \end{align*}
 where in the last equality we have used the fact that $\overline{\mathbf{T}}_{2,2} \partial \omega = - \overline{\partial} \omega$ by \cite[Corollary 3.31]{Schippers_Staubach_scattering_III}. 
  Now as above, applying the second identity of \cite[Theorem 3.23]{Schippers_Staubach_scattering_III} and the second line of \cite[Theorem 3.24]{Schippers_Staubach_scattering_III}  we obtain 
 \begin{align*}
     \overline{\mathbf{T}}_{2,2} \alpha_1 
       & = - \overline{\beta}_2 + \overline{\mathbf{R}}_2 \overline{\mathbf{S}}_2 \overline{\sigma_2} - \overline{\mathbf{T}}_{1,2} \mathbf{T}_{2,1} \overline{\gamma_2} - \overline{\mathbf{R}}_{2} \overline{\mathbf{S}}_2 \overline{\gamma_2} - \overline{\mathbf{T}}_{1,2} \mathbf{R}_1 \mathbf{S}_2 \tau_2 & \\
       & =  - \overline{\beta}_2 - \overline{\mathbf{T}}_{1,2} \alpha_1 +
       \overline{\mathbf{R}}_2 \overline{\mathbf{S}}_2 \overline{\sigma}_2 + \overline{\mathbf{R}}_2 \overline{\mathbf{S}}_2  \overline{\mu}_2 & {\text{\scriptsize Eqn \eqref{eq:compatible_breakdown_C1}, \cite[Theorem 4.1]{Schippers_Staubach_scattering_III}  }} \\
       & =  - \overline{\beta}_2 - \overline{\mathbf{T}}_{1,2} \alpha_1 + \overline{\mathbf{R}}_2 \overline{\eta}  & {\text{\scriptsize Eqn \eqref{eq:compatible_breakdown_B1},\eqref{eq:compatible_breakdown_D1}}}.
 \end{align*}
 This completes the proof.
 \end{proof}

 In the genus zero case we have the following.
 \begin{theorem}[Scattering matrix genus zero]  \label{th:unitary_scattering_genus_zero} Assume that $g=0$ and $\riem_2$ is connected. and let $\alpha_k + \overline{\beta}_k \in \mathcal{A}_{\mathrm{harm}}(\riem_k)$.  Assume that $\alpha_1 + \overline{\beta}_1 = \mathbf{O}^{\mathrm e}(\alpha_2 + \overline{\beta_2})$.
  Then $\alpha_k$ and $\beta_k$ satisfy
   \begin{equation}  \label{eq:scattering_matrix_genus_zero}
     \left( \begin{array}{c} \overline{\beta}_1 \\ \overline{\beta}_2   \end{array} \right)
     = \left( \begin{array}{cc} {-} \overline{\mathbf{T}}_{1,1} &  { -}\overline{\mathbf{T}}_{2,1}  \\
       {-}\overline{\mathbf{T}}_{1,2} &  { -}\overline{\mathbf{T}}_{2,2}    \end{array} \right) 
      \left( \begin{array}{c} {\alpha}_1 \\ {\alpha}_2  \end{array} \right).
  \end{equation}
  and the matrix is unitary.
 \end{theorem}
 \begin{proof} 
   One obtains the much simpler proof by setting all elements of $\mathcal{A}_{\text{harm}}(\mathscr{R})$ to zero in the proof of Theorem \ref{th:unitary_scattering_genus_nonzero}. 
 \end{proof}
 
 We conclude this section with some observations on the action of the scattering matrix on harmonic measures. {These can be viewed as symmetries of the scattering process.}
 
 Fix $\alpha_k+ \overline{\beta}_k \in \mathcal{A}_{\mathrm{harm}}(\riem_k)$, $k=1,2$, are compatible with respect to $\zeta = \xi + \overline{\eta} \in \mathcal{A}_{\mathrm{harm}}(\mathscr{R})$. By Theorem  \ref{th:compatibility_existence} part (1), if $\alpha_2'+\overline{\beta}_2' \in \mathcal{A}_{\mathrm{harm}}(\riem_2)$ is another compatible form we have that 
 \begin{align*}
   \alpha_2' & = \alpha_2 + \partial \omega \\
   \overline{\beta}_2' & = \overline{\beta}_2 + \overline{\partial} \omega 
 \end{align*} 
 for some bridgeworthy form $d\omega \in \mathcal{A}_{\mathrm{bw}}(\riem_2)$. This is reflected by the matrix equation 
  \begin{equation}  \label{eq:scattering_matrix_on_bridgeworthy}
     \left( \begin{array}{c} 0  \\ \overline{\partial}  \omega \\ 0 \end{array} \right)
     = \left( \begin{array}{ccc} - \overline{\mathbf{T}}_{1,1} & - \overline{\mathbf{T}}_{2,1} & \overline{\mathbf{R}}_1 \\
      - \overline{\mathbf{T}}_{1,2} & - \overline{\mathbf{T}}_{2,2} & \overline{\mathbf{R}}_2 \\ \mathbf{S}_1 & \mathbf{S}_2 & 0 \end{array} \right) 
      \left( \begin{array}{c} 0  \\ \partial \omega \\ 0 \end{array} \right)
  \end{equation}
  which follows from \cite[Corollary 3.31]{Schippers_Staubach_scattering_III}. 
 
  By Proposition \ref{pr:perturb_both_by_harmonic_measure} if we fix $\zeta$ and simultaneously perturb the other two forms by a harmonic measure, the resulting forms are still compatible.  That is, if $\omega_k \in \mathcal{A}_{\mathrm{hm}}(\riem_k)$ for $k=1,2$ satisfy $\mathbf{O}_{1,2} \omega_1=\omega_2$ then 
  \begin{align*}
      \alpha_1' + \overline{\beta}_1' & = \alpha_1 + \overline{\beta}_1 + d\omega_1 \\
      \alpha_2' + \overline{\beta}_2' & = \alpha_2 + \overline{\beta}_2 + d\omega_2 
  \end{align*}
  are also compatible with respect to $\zeta$. This is in turn reflected in the following matrix equation:
  \begin{equation}  \label{eq:scattering_matrix_on_harmonic_measures}
     \left( \begin{array}{c} \overline{\partial} \omega_1 \\ \overline{\partial} \omega_2 \\ 0 \end{array} \right)
     = \left( \begin{array}{ccc} - \overline{\mathbf{T}}_{1,1} & - \overline{\mathbf{T}}_{2,1} & \overline{\mathbf{R}}_1 \\
      - \overline{\mathbf{T}}_{1,2} & - \overline{\mathbf{T}}_{2,2} & \overline{\mathbf{R}}_2 \\ \mathbf{S}_1 & \mathbf{S}_2 & 0 \end{array} \right) 
      \left( \begin{array}{c} \partial \omega_1 \\ \partial \omega_2 \\ 0 \end{array} \right)
  \end{equation}
  which follows from \cite[Theorems 3.25, 3.30]{Schippers_Staubach_scattering_III}.
 \end{subsection}
 \begin{subsection}{Analogies with classical potential scattering} \label{se:scattering_analogies}
  
  In this section we describe the analogy between the scattering matrix in Theorems \ref{th:unitary_scattering_genus_nonzero} and \ref{th:unitary_scattering_genus_zero} and scattering by a potential. We will restrict to the genus zero case in the former, and compare it to scattering by a potential well in one dimension.
  
  Let \(a(x, D)=\)
\(D^{2}+V(x)\) denote the $1$-dimensional Schr\"odinger operator, with \(D=-i \partial_{x}\), where the
potential \(V(x)\) is smooth and goes to 0 sufficiently fast as \(x\) goes to infinity. Because of this decay assumption at infinity, for $\lambda \in \R$ the solutions of the stationary Schr\"odinger equation

$$
a(x,D) u=\lambda^2 u,
$$
should behave like
$$
a_{+}^{l, r}(\lambda) \mathrm{e}^{i \lambda x }+a_{-}^{l, r}(\lambda) \mathrm{e}^{-i \lambda x },
$$
as $x \rightarrow \pm \infty$, where $l$ and $r$ stand for left and right and correspond respectively to $x \rightarrow-\infty$
and $x \rightarrow+\infty$. The so-called {\it{Jost solutions}} $\mathscr{J}_{\pm}^{l, r}$ are the solutions which behave exactly
as $\mathrm{e}^{i \lambda x}$ or $\mathrm{e}^{-i \lambda x}$ as $x \rightarrow \pm \infty$ ($e^{\pm i\lambda x}$ are obviously solutions to the equation $D^2 u= \lambda^2 u$, i.e. the original equation without any potential).\\
Now the scattering problem amounts to finding the components of a solution $u$ of the Schr\"odinger equation in
the basis $\left(\mathscr{J}_{+}^{r}, \mathscr{J}_{-}^{l}\right)$ of the outgoing Jost solutions, if one knows the components of $u$ in the basis
$\left(\mathscr{J}_{+}^{l}, \mathscr{J}_{-}^{r}\right)$ of the incoming Jost solutions.\\
In scattering theory and quantum mechanics,
the $2\times 2$ matrix 
$$\mathbb{S}(\lambda )=\left(\begin{array}{cc}s_{11} & s_{12} \\ s_{21} & s_{22}\end{array}\right)$$
that relates these components
is called the \emph{scattering matrix}. It also turns out that if the potential $V$ is real on
the real axis, then $\overline{\mathscr{J}_{+}^{l, r}}= \mathscr{J}_{-}^{l, r}$ and the scattering matrix
$\mathbb{S}(\lambda)$ is unitary. 

Turning to scattering in quasicircles, let $\riem_1$ and $\riem_2$ be identified with domains in the Riemann sphere $\sphere$. By M\"obius invariance we can assume without loss of generality that $\riem_2$ contains the point at $\infty$ whilst $\riem_1$ contains $0$.  The punctured plane is conformally a cylinder, with the points at $0$ and $\infty$ infinitely far away, with the quasicircle separating $0$ from $\infty$. 
We then identify left with $0$ and right with $\infty$. The quasicircle can be thought of as a potential well with a possibly highly irregular support set, whose Hausdorff dimension is in $[1,2)$.

The problem then is, given left moving solutions (harmonic forms in $\riem_1$) find right moving solutions (harmonic forms in $\riem_2$) which overfare through the potential well. The function behaves harmonically as $z \rightarrow 0/\infty$, but not across the potential well. The holomorphic forms are identified with solutions to the harmonic scattering problem  with expansions of the form 
\[  \alpha_k = \left( \sum_{k=1}^\infty \alpha_k^n z^n  \right) dz  \]
and the anti-holomorphic forms are identified with solutions to the harmonic scattering problem with  expansions of the form 
\[  \overline{\beta}_k = \left( \sum_{k=1}^\infty \beta^k_n z^{-n}  \right) dz.  \]
Thus we identify $\pm$ with holomorphic/anti-holomorphic respectively, and 
\[ \overline{\mathscr{J}_{+}^{l}} \sim \alpha_1, \ \ \  \overline{\mathscr{J}_{-}^{l}} \sim \overline{\beta}_1, \ \ \  \mathscr{J}_{+}^{r} \sim \alpha_2, \ \ \ \mathscr{J}_{-}^{r} \sim \overline{\beta}_2.  \]

 \end{subsection}
\end{section}
\begin{section}{The period mapping}  \label{se:period_mapping}
\begin{subsection}{Assumptions throughout this section}   
  
  \label{se:period_mapping_assumptions}
  The following assumptions will be in force throughout Section \ref{se:period_mapping}.   
 \begin{enumerate}
     \item $\mathscr{R}$ is a compact Riemann surface of genus $g$, with $n$ punctures $p_1,\ldots,p_n$;
     \item $\Gamma = \Gamma_1 \cup \cdots \cup \Gamma_n$ is a collection of quasicircles;
     \item $\Gamma$ separates $\mathscr{R}$ into $\riem_1$ and $\riem_2$ in the sense of Definition \ref{de:separating_complex};
     \item $\riem_2$ is connected; 
     \item $\riem_1 = \Omega_1 \cup \cdots \cup \Omega_n$ where $\Omega_1,\ldots,\Omega_n$ are simply-connected sets with disjoint closures;
     \item $p_k \in \Omega_k$ for $k=1,\ldots,n$.   
 \end{enumerate} 
  We refer to the domains $\Omega_k$ as ``caps''.
 $\Gamma$ is always given the positive orientation with respect to $\riem_1$.   
 
     In this section we generalize the classical period mapping for compact surfaces to surfaces with border. This new period mapping relates both to the cohomology of the set of holomorphic one-forms on the compact surface $\mathscr{R}$, and to the structure of the set of boundary values of holomorphic one-forms on $\riem_2$. Thus it unifies both the classical polarization induced by the holomorphic one-forms on the compact surface (relating cohomology to complex structure) with the period maps of genus zero surfaces with boundary studied by various authors, including the Kirillov-Yuri'ev-Nag-Sullivan period map of the Teichm\"uller space of the disk \cite{KY2}, \cite{NS}, \cite{Takhtajan_Teo_Memoirs}, \cite{RSS_genus_zero},    
 
 In Section \ref{se:period_map_upsilon} we define a canonical isomorphism parametrizing the set of holomorphic one-forms on a surface $\riem_2$ of genus $g$ with $n$ boundary curves. We then use this to define a natural polarization and a map whose graph is the set of holomorphic one-forms on $\riem_2$.  In Section \ref{se:KYNS_period} we show how this generalizes both the classical and KYNS period mappings, and relate it to the Grunsky inequalities. Finally, in Section \ref{se:holomorphic_BVP} we use the machinery of the previous sections to give a reduction of the boundary value problem for holomorphic one-forms with $H^{-1/2}$ data to a non-singular integral equation on the $n$-fold direct sum of the Bergman space of the disk.
\end{subsection}
\begin{subsection}{The generalized period map}  \label{se:period_map_upsilon}

    Consider the space
    \[   \mathcal{A}_{\mathrm{harm}}(\riem_1) \oplus \mathcal{A}_{\mathrm{harm}}(\mathscr{R}).   \]
   
    We have the two projections  
   \[ \gls{pcap} = \mathbf{P}_{\riem_1} \oplus  {\overline{\mathbf{P}}}_{\mathscr{R}} : \mathcal{A}_{\mathrm{harm}}(\riem_1) \oplus \mathcal{A}_{\mathrm{harm}}(\mathscr{R}) \rightarrow \mathcal{A}(\riem_1) \oplus \overline{\mathcal{A}(\mathscr{R})}    \]
   and 
   \[   \overline{\mathbf{P}}_{\mathrm{cap}} = \overline{\mathbf{P}}_{\riem_1} \oplus  {\mathbf{P}}_{\mathscr{R}} : \mathcal{A}_{\mathrm{harm}}(\riem_1) \oplus \mathcal{A}_{\mathrm{harm}}(\mathscr{R}) \rightarrow \overline{\mathcal{A}(\riem_1)} \oplus {\mathcal{A}(\mathscr{R})}   \]
   where $\mathbf{P}_{\riem_1}$ was defined in \eqref{eq:hol_antihol_projections_Bergman}. Also   
   \[   {\mathbf{P}}_{\mathscr{R}}: \mathcal{A}_{\mathrm{harm}}(\mathscr{R}) \rightarrow \mathcal{A}(\mathscr{R})   \]
   is the projection onto the holomorphic part, and similarly $ {\overline{\mathbf{P}}}_{\mathscr{R}}$ is the projection onto the anti-holomorphic part.  The projections are obviously bounded.

   We define the following operator, which we will shortly show is an isomorphism.    
   \begin{align}\label{defn: captial theta}
       \gls{theta}: \overline{\mathcal{A}(\riem_1)} \oplus \mathcal{A}(\mathscr{R}) & \rightarrow \mathcal{A}^{\mathrm{se}}(\riem_2)  \\
      \nonumber (\overline{\gamma},\tau) & \mapsto - \mathbf{T}_{1,2} \overline{\gamma} +  {\mathbf{R}}_2 \tau.  
   \end{align}
   This is bounded because $\mathbf{T}_{1,2}$ and $\mathbf{R}_2$ are.

  We define an augmented overfare operator which contains the extra data of the cohomology class. This will be a factor of the inverse of $\Theta$. First, observe that given $\beta \in \mathcal{A}^{\mathrm{se}}_{\text{harm}}(\riem_2)$, there is a unique one-form $\sigma \in \mathcal{A}(\mathscr{R})$ whose restriction $\mathbf{R}^\mathrm{h}_2 \sigma$ is in the same cohomology class as $\beta$.  Thus $\alpha \in \mathcal{A}(\riem_1)$ and $\beta \in \mathcal{A}(\riem_2)$ can be compatible only via the form $\sigma$.  Also, since $\riem_2$ is connected, the exact overfare is well-defined.  Thus, given $\beta$, we have the following uniquely determined compatible form $\hat{\mathbf{O}} \beta \in \mathcal{A}_{\text{harm}}(\riem_1)$:
    \begin{align*}
       \gls{ohat}: \mathcal{A}^{\mathrm{se}}_{\mathrm{harm}}(\riem_2) & \rightarrow \mathcal{A}_{\mathrm{harm}}(\riem_1)  \\
      \beta & \mapsto \mathbf{O}^{\mathrm{e}}(\beta -  {\mathbf{R}}_2^{\mathrm{h}} \sigma) +
        {\mathbf{R}}_1^{\mathrm{h}}\sigma  
    \end{align*}
    where $\sigma$ is the unique element of 
    $\mathcal{A}_{\mathrm{harm}}(\mathscr{R})$ such that $\beta-  {\mathbf{R}}_2^{\mathrm{h}} \sigma$ is exact.  

 Using this, we define the augmented overfare map
 \begin{align}   \label{eq:Oaugmented_definition}  
   \gls{augo}: \mathcal{A}^{\mathrm{se}}_{\mathrm{harm}}(\riem_2) & \rightarrow \mathcal{A}_{\mathrm{harm}}(\riem_1) \oplus \mathcal{A}_{\mathrm{harm}}(\mathscr{R})  \nonumber \\
   \beta   & \mapsto (  \hat{\mathbf{O}} \beta, \sigma )
 \end{align}
 where $\sigma$ is the unique element of $\mathcal{A}_{\mathrm{harm}}(\mathscr{R})$ such that $\beta -  {\mathbf{R}}^{\mathrm{h}}_2  \sigma$ is exact.

   \begin{theorem}   \label{th:Theta_isomorphism} $\Theta$ is an isomorphism with inverse $\overline{\mathbf{P}}_{\mathrm{cap}} \mathbf{O}^{\mathrm{aug}}$.  
   \end{theorem}
   \begin{proof} The fact that $\Theta$ is surjective follows directly from
   \cite[Corollary 4.18]{Schippers_Staubach_scattering_III}.  
    
    We show it is injective.  Let $(\overline{\gamma},\tau) \in \overline{\mathcal{A}(\riem_1)} \oplus \mathcal{A}(\mathscr{R})$.  
    We then have that 
    \[    - \mathbf{T}_{1,2} \overline{\gamma} +  {\mathbf{R}}_2 \tau - (\overline{\mathbf{R}}_2 \overline{\mathbf{S}}_1 \overline{\gamma} + \mathbf{R}_2  {\tau} )    
            \]
    is exact by \cite[Theorem 4.7]{Schippers_Staubach_scattering_III}.  So applying \eqref{eq:Oaugmented_definition}, we obtain 
    \begin{align}   \label{eq:augmented_overfare_explicit}
     \mathbf{O}^{\mathrm{aug}}  (- \mathbf{T}_{1,2} \overline{\gamma} +  {\mathbf{R}}_2 \tau ) & = ( \hat{\mathbf{O}}(-\mathbf{T}_{1,2} \overline{\gamma} +  {\mathbf{R}}_2 \tau ),\overline{\mathbf{S}}_1 \overline{\gamma} + \tau   )   \nonumber \\
     & = ( \overline{\gamma} -  \mathbf{T}_{1,1} \overline{\gamma}
       + \mathbf{R}_1 \tau_1, \overline{\mathbf{S}}_1 \overline{\delta} +  \tau  )
    \end{align}
    where we have used \cite[Proposition 4.10]{Schippers_Staubach_scattering_III}. to show that
    \begin{align*}
        \hat{\mathbf{O}}(- \mathbf{T}_{1,2} \overline{\gamma} +  {\mathbf{R}}_2 \tau )  & = \mathbf{O}^{\mathrm{e}} (- \mathbf{T}_{1,2} \overline{\gamma} - {\overline{\mathbf{R}}}_1 \overline{\mathbf{S}}_1 \overline{\gamma})  + {\overline{\mathbf{R}}}_1 \overline{\mathbf{S}}_1 \overline{\gamma}  +  {\mathbf{R}}_1 \tau \\
       & = \overline{\gamma} -  \mathbf{T}_{1,1} \overline{\gamma} - \overline{\mathbf{R}}_1 \overline{\mathbf{S}}_1 \overline{\gamma} 
       + \overline{\mathbf{R}}_1 \overline{\mathbf{S}}_1 \overline{\gamma} + \mathbf{R}_1 \tau \\
       & = \overline{\gamma} -  \mathbf{T}_{1,1} \overline{\gamma}
       + {\mathbf{R}}_1 \tau.
    \end{align*}
    Thus 
    \begin{align*}
     \overline{\mathbf{P}}_{\mathrm{cap}} \mathbf{O}^{\mathrm{aug}} \Theta (\overline{\gamma},\tau) & =  \overline{\mathbf{P}}_{\mathrm{cap}} ( \overline{\gamma} -  \mathbf{T}_{1,1} \overline{\gamma}
       + \mathbf{R}_1 \tau,  \overline{\mathbf{S}}_1\overline{\gamma} +  \tau  ) \\
       & = (\overline{\gamma}, \tau).   
    \end{align*}
    This shows that $\overline{\mathbf{P}}_{\mathrm{cap}} \mathbf{O}^{\mathrm{aug}}$ is  a left inverse
    of $\Theta$, and hence $\Theta$ is one-to-one.   
    
    Since $\Theta$ is bounded and bijective, it is an isomorphism, and the left inverse equals the right inverse.
   \end{proof}
   
   The decomposition
   \begin{equation}  \label{eq:riemtwo_decomposition}
      \mathcal{A}^{\mathrm{se}}_{\mathrm{harm}}(\riem_2) = \mathcal{A}^{\mathrm{se}}(\riem_2) \oplus \overline{\mathcal{A}^{\mathrm{se}}(\riem_2)}    
   \end{equation}
   induces a polarization on 
   \[  \mathcal{A}_{\mathrm{harm}}(\riem_1) \oplus \mathcal{A}_{\mathrm{harm}}(\mathscr{R}) \]
  by 
   \begin{align} \label{eq:polarization_general_non_trivialized}
       \mathcal{A}_{\mathrm{harm}}(\riem_1) \oplus \mathcal{A}_{\mathrm{harm}}(\mathscr{R}) 
       & = \mathbf{O}^{\mathrm{aug}} \mathcal{A}^{\mathrm{se}}(\riem_2) \oplus \mathbf{O}^{\mathrm{aug}} \overline{\mathcal{A}^{\mathrm{se}}(\riem_2)} \nonumber \\
       & =: W \oplus \overline{W}.  
   \end{align}
    We also have the fixed polarization 
   \[    \mathcal{A}_{\mathrm{harm}}(\riem_1) \oplus \mathcal{A}_{\mathrm{harm}}(\mathscr{R}) = W_0 \oplus \overline{W_0}  \]
   where 
   \[   W_0 = \overline{\mathcal{A}(\riem_1)} \oplus \mathcal{A}(\mathscr{R}).   \]
   Observe that by Theorem \ref{th:Theta_isomorphism} the new positive polarization $W$ can be written  
   \begin{equation*} 
     W = \mathbf{O}^{\mathrm{aug}} \, \Theta \, W_0. 
   \end{equation*}
   
   Furthermore we have the following result.  Define 
   \begin{align}\label{defn:upsilon}
      \gls{period} :\overline{\mathcal{A}(\riem_1)} \oplus \mathcal{A}(\mathscr{R})  &  \rightarrow \mathcal{A}(\riem_1) \oplus \overline{\mathcal{A}(\mathscr{R})} \\
      \nonumber(\overline{\gamma},\tau) & \mapsto ( - \mathbf{T}_{1,1} \overline{\gamma} +  {\mathbf{R}}_1 \tau,  \overline{\mathbf{S}}_1 \overline{\gamma} ). 
   \end{align}
   Then 
   \begin{theorem}  \label{th:upsilon_expression} We have that $\mathbf{\Upsilon} =  {\mathbf{P}}_{\mathrm{cap}} \mathbf{O}^{\mathrm{aug}} \Theta$.  In particular, 
     the positive polarization $W$ is the graph of $\Upsilon$.  
   \end{theorem}
   \begin{proof}
    The first claim follows from \eqref{eq:augmented_overfare_explicit}.  The second claim follows from the first claim together with the fact that $\overline{\mathbf{P}}_{\mathrm{cap}} \mathbf{O}^{\mathrm{aug}} \Theta = \mathbf{I}$ by Theorem \ref{th:Theta_isomorphism}.    
   \end{proof}
   
   We also have
   \begin{theorem} \label{th:upsilon_bounded}
    The operator norm $\| \mathbf{\Upsilon} \| <1$.   
   \end{theorem}
   \begin{proof}
    Using the notation of the proof of Theorem \ref{th:Theta_isomorphism}, we have 
    \begin{align}  \label{eq:upsilon_bound_temp_one} 
        \| - \mathbf{T}_{1,1} \overline{\gamma} +  {\mathbf{R}}_1 \tau \|^2 
        & = \| \mathbf{T}_{1,1} \overline{\gamma} \|^2 - 2 \text{Re} \left< \mathbf{T}_{1,1} \overline{\gamma}, \mathbf{R}_1 \tau \right> 
        + \| \mathbf{R}_1 \tau \|^2.
    \end{align}
    Now by \cite[Theorem 3.19]{Schippers_Staubach_scattering_III} and \cite[Equation 3.17]{Schippers_Staubach_scattering_III} we have that
    \begin{align} \label{eq:upsilon_bound_temp_two}
     \left< \mathbf{T}_{1,1} \overline{\gamma}, \mathbf{R}_1 \tau \right> & = \left< \mathbf{S}_1 \mathbf{T}_{1,1} \overline{\gamma} ,\tau \right> 
     = - \left< \mathbf{S}_2 \mathbf{T}_{1,2} \overline{\gamma} ,\tau \right> \nonumber \\
     & = - \left< \mathbf{T}_{1,2} \overline{\gamma}, \mathbf{R}_2 \tau \right>.
    \end{align}
    Furthermore \cite[Theorems 3.19, 3.21]{Schippers_Staubach_scattering_III} yield that
    \begin{equation}  \label{eq:upsilon_bound_temp_three}
        \| \mathbf{R}_1 \tau \|^2 = \| \tau \|^2 - \| \mathbf{R}_2 \tau \|^2.
    \end{equation}
    Inserting \eqref{eq:upsilon_bound_temp_two} and \eqref{eq:upsilon_bound_temp_three} in \eqref{eq:upsilon_bound_temp_one} we obtain 
    \begin{equation} \label{eq:upsilon_bound_temp_four}
        \| - \mathbf{T}_{1,1} \overline{\gamma} +  {\mathbf{R}}_1 \tau \|^2  = 
        \| \mathbf{T}_{1,1} \overline{\gamma} \|^2 + 2 \text{Re} \left<  \mathbf{T}_{1,2} \overline{\gamma} ,\mathbf{R}_2 \tau \right> - \| \mathbf{R}_2 \tau \|^2 + \| \tau \|^2.  
    \end{equation}
    By \cite[Theorem 3.23]{Schippers_Staubach_scattering_III}, 
    \[  \| \mathbf{T}_{1,1} \overline{\gamma} \|^2 + \| \overline{\mathbf{S}}_{1} \overline{\gamma} \|^2 = \| \overline{\gamma} \|^2 - \| \mathbf{T}_{1,2} \overline{\gamma} \|^2  \]
    which when inserted in \eqref{eq:upsilon_bound_temp_four} yields
    \begin{align*}
        \| (-\mathbf{T}_{1,1} \overline{\gamma} +  {\mathbf{R}}_1 \tau,\overline{\mathbf{S}}_1 \overline{\gamma}) \|^2 & = 
        \| \overline{\gamma} \|^2 -\| \mathbf{T}_{1,2} \overline{\gamma} \|^2 + 2 \text{Re} \left<  \mathbf{T}_{1,2} \overline{\gamma} ,\mathbf{R}_2 \tau \right> - \| \mathbf{R}_2 \tau \|^2 + \| \tau \|^2  \\
        & = \| (\overline{\gamma},\tau )\|^2 - \| - \mathbf{T}_{1,2} \overline{\gamma} + \mathbf{R}_2 \tau \|^2. 
    \end{align*}
    Now since $(\overline{\gamma},\tau) \mapsto -\mathbf{T}_{1,2} \overline{\gamma} + \mathbf{R}_2 \tau$ is an isomorphism 
   by Theorem \ref{th:Theta_isomorphism}, there is a $c<1$ (uniform for all $(\overline{\gamma},\tau)$) such that
    \[ \| - \mathbf{T}_{1,2} \overline{\gamma} + \mathbf{R}_2 \tau \| \geq c \| (\overline{\gamma},\tau) \|,   \]
    therefore
    \begin{equation*}
        \| (-\mathbf{T}_{1,1} \overline{\gamma} + {\mathbf{R}}_1 \tau,\overline{\mathbf{S}}_1 \overline{\gamma}) \|^2 \leq (1-c^2) 
        \| (\overline{\gamma},\tau) \|^2,
    \end{equation*}
    and the proof is completed.
   \end{proof}
\end{subsection}
\begin{subsection}{Generalized polarizations}
 \label{se:KYNS_period} 
 Theorems \ref{th:upsilon_expression} and \ref{th:upsilon_bounded} generalize the Kirillov-Yuri'ev-Nag-Sullivan (KYNS) period map to Riemann surfaces of arbitrary genus and number of boundary curves, and unify it with the classical period map of compact surfaces. This fact will be shown below, but prior to that, we shall review some of the literature.  For the sake of consistency, we impose our notation and choice of function spaces in this review of the literature. For example, we freely take advantage of the isomorphism between the homogeneous Sobolev space, Dirichlet space and the Bergman space of one-forms on the disk.
 
 S. Nag and D. Sullivan \cite{NS}, following A. A. Kirillov and D. V. Yuri'ev \cite{KrillovYuriev} in the smooth case, showed how the group of quasisymmetries of the circle $\mathrm{QS}(\mathbb{S}^1)$ acts symplectically on  $\dot{H}^{1/2}(\mathbb{S}^1)$. Setting $W_0 = \mathcal{A}(\disk)$, $\overline{W}_0 = \overline{\mathcal{A}(\disk)}$,
 we have the standard polarization
 \[ \dot{H}^{1/2}(\mathbb{S}^1) = W_0 \oplus \overline{W}_0.  \]
 Each quasisymmetry induces a new positive polarization $W \oplus \overline{W} = \dot{H}^{1/2}(\mathbb{S}^1)$, or equivalently an operator $\mathbf{Gr}:\overline{W}_0 \rightarrow W_0$ of norm strictly less than one whose graph is $W$. Kirillov and Yu'riev \cite{KY2} made the important discovery that this operator can be identified with the classical Grunsky operator. Takhtajan and Teo \cite{Takhtajan_Teo_Memoirs} showed that the resulting period mapping taking an element of the universal Teichm\"uller space to its operator $\mathbf{Gr}$ is holomorphic for both the full universal  Teichm\"uller space and the Weil-Petersson universal Teichm\"uller space.\\

Now fix a Riemann surface $\riem$ of genus $g$ with $N$ boundary curves.  We choose a collection 
 \[ \phi_k:\mathbb{S}^1 \rightarrow \partial_k \riem, \ \ \ k=1,\ldots,n \] 
 of quasisymmetric mappings, whose purpose is to map the boundary values into a fixed space.  Sew on $n$ copies of the disk $\disk$ using the mapping $\phi_k$ to identify points of the boundary of the disk with points of $\partial_k \riem$, for $k=1,\ldots,n$.  By \cite[Theorems 3.2, 3.3]{RadnellSchippers_monster}  the resulting compact topological space has a unique complex structure compatible with $\riem_2$ and the disks $\disk$.  We call this Riemann surface $\mathscr{R}$, and the common boundaries of each $\disk$ and $\partial_k \riem_2$ for each $k$ are a separating complex of quasicircles satisfying the assumptions of Section \ref{se:period_mapping_assumptions}. After uniformizing, the copies of the disk are identified with simply connected domains $\Omega_1,\ldots,\Omega_n$ which are biholomorphic to $\disk$ and bounded by non-intersecting quasicircles.  Set $\riem_1 = \Omega_1 \cup \cdots \cup \Omega_n$ and let $\riem_2$ denote the complement of their closure in $\riem_1$. For $k=1,\ldots,n$, $\Omega_k$ is biholomorphic to the disk under the conformal extensions 
 \[ f_k: \disk \rightarrow  \Omega_k \]
 of the quasisymmetric maps $\phi_k$. Furthermore $\riem$ is biholomorphic to $\riem_2$ under the uniformizing map; we henceforth identify $\riem_2$ with $\riem$. 

 { Recall now from the previous section the decomposition
   \begin{equation*}  
      \mathcal{A}^{\mathrm{se}}_{\mathrm{harm}}(\riem) = \mathcal{A}^{\mathrm{se}}(\riem) \oplus \overline{\mathcal{A}^{\mathrm{se}}(\riem)}    
   \end{equation*} 
   and the induced polarization on 
   \[  \mathcal{A}_{\mathrm{harm}}(\riem_1) \oplus \mathcal{A}_{\mathrm{harm}}(\mathscr{R}) \]
  by 
   \begin{align*}  
       \mathcal{A}_{\mathrm{harm}}(\riem_1) \oplus \mathcal{A}_{\mathrm{harm}}(\mathscr{R}) 
       & = \mathbf{O}^{\mathrm{aug}} \mathcal{A}^{\mathrm{se}}(\riem_2) \oplus \mathbf{O}^{\mathrm{aug}} \overline{\mathcal{A}^{\mathrm{se}}(\riem_2)} \nonumber \\
       & =: W \oplus \overline{W}.  
   \end{align*}
  Now $W$ represents the boundary values of the semi-exact holomorphic one-forms $\mathcal{A}^{\mathrm{se}}(\riem_2)$, together with the necessary cohomological data to recover them. We will pull this back to a fixed space in the spirit of the KYNS and classical period mappings.
 }
 
 Let $\Omega(\mathscr{R})$ denote the cohomology classes of closed $L^2$ one-forms on $\mathscr{R}$.
 The complex structure on $\riem$ determines a complex structure on $\mathscr{R}$, which thus determines the class of harmonic forms on $\mathscr{R}$.  We thus obtain a map 
 \[ \mathbf{P}_{\mathrm{harm},\mathscr{R}}:\Omega(\mathscr{R}) \rightarrow \mathcal{A}_{\mathrm{harm}}(\mathscr{R})  \]
 which depends on the complex structure of $\mathscr{R}$ (and hence of $\riem$).  We then define the projections onto the anti-holomorphic and holomorphic parts
 \[ \mathbf{P}_{\mathrm{harm},\mathscr{R}} = \mathbf{P}_\mathscr{R} \oplus \overline{\mathbf{P}}_\mathscr{R}:\Omega(\mathscr{R}) \rightarrow \mathcal{A}(\mathscr{R}) \oplus \overline{\mathcal{A}(\mathscr{R})}  \] 
 (note that we are expanding the domain of $\mathbf{P}_{\mathscr{R}}$ to $\Omega(\mathscr{R})$ for the sake of the discussion in this section ).
 Thus we obtain a polarization 
 \[ \Omega(\mathscr{R}) = W_{\mathscr{R}} \oplus \overline{W}_{\mathscr{R}}  \]
 via 
 \[  W_{\mathscr{R}} = \mathbf{P}_{\mathscr{R}} \Omega(\mathscr{R}),  \ \ \  \overline{W}_{\mathscr{R}} = \overline{\mathbf{P}}_{\mathscr{R}}  \Omega(\mathscr{R}).    \]
 
 Denoting $f=f_1 \times \cdots \times f_n$ and $\mathcal{A}_{\mathrm{harm}}(\disk)^n = \mathcal{A}_{\mathrm{harm}} \oplus \cdots \oplus \mathcal{A}_{\mathrm{harm}}$ we then have the polarization
 \[  \mathcal{A}_{\mathrm{harm}}(\disk)^n \oplus \Omega(\mathscr{R}) = 
   \mathcal{W}  \oplus \overline{\mathcal{W}}     \]
   where  
   \begin{equation}
   \label{eq:general_polarization}
     \mathcal{W} = (f^*,\mathbf{Id}) W 
   \end{equation}
   and $W$ is given by \eqref{eq:polarization_general_non_trivialized}. 
Thus this combines both the classical polarization associated to the complex structure of the compact surface (in the second entry of \eqref{eq:general_polarization}) and the boundary values of the set of one-forms (in the first element of \eqref{eq:general_polarization}).   

 \begin{remark}
  Observe that both the set of quasisymmetries $\phi_k:\mathbb{S}^1 \rightarrow \partial_k \riem$ and the space of closed $L^2$ one-forms is unchanged under quasiconformal deformations $f:\riem \rightarrow \riem$. Thus the space $\Omega^{\mathrm{se}}(\riem)$
  is invariant under a quasiconformal deformation of the complex structure of $\riem$. This fact is one of the motivations for our analytic choices ($L^2$-boundary values and separating curves being quasicircles) in this paper.

 \end{remark}
 \end{subsection}
 \begin{subsection}{Generalized Grunsky inequalities}
  \label{se:Grunsky}
 In this section, we show how Theorem \ref{th:upsilon_bounded} generalizes various versions of the Grunsky inequalities appearing in the literature to the case of surfaces of type $(g,n)$, after pulling back to $n$ copies of the disk.   In general, it is the overfare results (either in special cases or in general) which makes it possible to interpret the Grunsky portion of the polarization in terms of boundary values. Here we consider two cases:\\
 
 {\bf Case I: $g=0$.} If we assume that the genus of $\riem_1$ is zero and $n=1$, then $\mathscr{R} = \sphere$, $\mathcal{A}(\mathscr{R}) = \{0 \}$ and $\mathcal{A}^{\mathrm{se}}(\riem_2)=\mathcal{A}^\mathrm{e}(\riem_2)$. Also note that $(\overline{\mathbf{R}}_1 \overline{\mathcal{A}(\mathscr{R})})^\perp = \overline{\mathcal{A}(\riem_1)}$. 
 Thus the map $\Theta$ (defined previously by \eqref{defn: captial theta}) takes the form
 \begin{align*}
    \Theta:\overline{\mathcal{A}(\riem_1)} & \rightarrow \mathcal{A}^{\mathrm{e}}(\riem_2) \\    
    \overline{\gamma} & \mapsto \mathbf{T}_{1,2} \overline{\gamma}.
 \end{align*}
 In the case that $n=1$ the fact that $\mathbf{T}_{1,2}$ is an isomorphism was first proved by V. V. Napalkov and R. S. Yulmukhametov \cite{Nap_Yulm}. In the genus zero case for general $n$ this is due to Radnell, Schippers, and Staubach \cite{RSS_Dirichletspace}.
 
 The Grunsky inequalities are obtained as follows. With the observations above, the map $\Upsilon$ (defined previously by \eqref{defn:upsilon}) is seen to take the form 
 \begin{align*}
     \Upsilon: \overline{\mathcal{A}(\riem_1)} & \rightarrow \mathcal{A}(\riem_2) \\
     \overline{\gamma} & \mapsto - \mathbf{T}_{1,1} \overline{\gamma}. 
 \end{align*}
 so that Theorem \ref{th:upsilon_bounded} implies that
 \begin{equation}\label{norm of T11}
      \| \mathbf{T}_{1,1} \| <1.
 \end{equation}  
 This is equivalent to the classical estimate on the classical Schiffer operator in the sphere. This version of the Grunsky inequalities first appeared in Bergman and Schiffer \cite{BergmanSchiffer}, though they assume that the boundary curves are analytic.  
 
 Explicitly, in the case that $n=1$, pulling back this estimate to the disk via the map $f:\disk \rightarrow \riem_1$ we obtain the usual form of the Grunsky inequalities. Following \cite{Schippers_Staubach_CAOT}, we define the Grunsky operator as follows
 \begin{align*}
     \gls{grunsk} : \overline{\mathcal{A}(\disk)} & \mapsto \mathcal{A}(\disk) \\
     \overline{\alpha} & \mapsto \mathbf{P}_{\disk} f^* \mathbf{O}^\mathrm{e}_{2,1} \mathbf{T}_{1,2} (f^{-1})^* \overline{\alpha}.
 \end{align*}
 See \cite{Schippers_Staubach_CAOT} for the relation to the usual Grunsky operator written in terms of Faber polynomials and Grunsky coefficients. \cite[Proposition 4.10]{Schippers_Staubach_scattering_III}. and the obvious fact that $\mathbf{P}_{\disk} f^* = f^* \mathbf{P}_{\riem_1}$ we obtain
 \begin{align} \label{eq:Grunsky_conjugation_expression}
     \mathbf{Gr}_f \overline{\alpha} & = f^* \mathbf{P}_{\riem_1} \mathbf{O}^\mathrm{e} \mathbf{T}_{1,2} (f^{-1})^* \overline{\alpha} \nonumber \\
     & = f^* \mathbf{P}_{\riem_1} \mathbf{O}^\mathrm{e} ( - (f^{-1})^*\overline{\alpha} + \mathbf{T}_{1,1} (f^{-1})^* \overline{\alpha} ) \nonumber \\
     & = - f^* \mathbf{T}_{1,1} (f^{-1})^* \overline{\alpha}.
 \end{align}
 Since $f^*$ and $(f^{-1})^*$ are isometries \eqref{norm of T11} yields that 
 \[ \| \mathbf{Gr}_f \| <1. \]
 
 This is equivalent to an integral form of the Grunsky inequalities due to Bergman-Schiffer \cite{BergmanSchiffer}. To see this, we recall some formulas for the Schiffer operators in the special case of the sphere and domain $\riem_1 \subseteq \sphere$, see \cite[Example 3.1]{Schippers_Staubach_scattering_III}. We have in general that
 \begin{equation}  \label{eq:nonsingular_Schiffer}
  \mathbf{T}_{1,1} \alpha(z)= \iint_{\riem_1,w} \left(L_\mathscr{R}(z,w) - L_{\riem_1}(z,w) \right) 
  \wedge \overline{\alpha(w)}. 
  \end{equation}
 where 
 \[  L_{\riem_1}(z,w) =   \frac{1}{\pi i} \partial_z \partial_w G_{\riem_1}(w,z),  \]
 and recall that $G_{\riem_1}$ is Green's function of $\riem_1$. 

 If $\mathscr{R}$ is the Riemann sphere $\sphere$, we have
 \[  \mathscr{G}(w,\infty;z,q) = - \log{ \frac{|w-z|}{|w-q|}}.     \]
 Thus
 \[  L_{\sphere}(z,w) = - \frac{1}{2 \pi i} \frac{dw \,dz}{(w-z)^2}.  \]
 So the Schiffer operators are given by, for $\overline{\alpha(z)} = \overline{h(z)}d \bar{z}$,
 \[  \ \mathbf{T}_1  \, \overline{\alpha} (z) =
    \frac{1}{\pi} \iint_{\riem_1} \frac{\overline{h(w)}}{(w-z)^2} \frac{d\bar{w} 
    \wedge d w}{2i}  \cdot dz.  \]
 
 Using conformal invariance of Green's function 
 we have that
 \[  L_{\riem_1}(z,w) = - \frac{1}{2 \pi i} \frac{(f^{-1})'(w)  (f^{-1})'(z) \,dw \,dz }{(f^{-1}(w)-f^{-1}(w))^2}    \]
 (see \cite[Example 3.2]{Schippers_Staubach_scattering_III}).
 
 Combining this with equation \eqref{eq:nonsingular_Schiffer} we obtain
 \[  \mathbf{T}_{1,1} \overline{\alpha} = \iint_{\riem_1} \frac{dz}{2 \pi i} \left[ \frac{dw}{(w-z)^2} - \frac{(f^{-1})'(w) (f^{-1})'(z) \,dw}{(f^{-1}(w)-f^{-1}(z))^2}  \right] \wedge_w \overline{\alpha(w)}.    \]
 Now let $\overline{\alpha(w)} = \overline{h'(w)} d\bar{w} \in \overline{\mathcal{A}(\disk)}$ (where $h(w) \in \mathcal{D}(\disk)$). Then using the above together with \eqref{eq:Grunsky_conjugation_expression} we see that (after a change of variables)
 \[  \mathbf{Gr}_f \overline{\alpha}  = 
\frac{1}{\pi} \iint_{\disk} dz \left[ \frac{1}{(w-z)^2} - \frac{f'(w) f'(z)}{(f(w)-f(z))^2}  \right] \overline{h'(w)}\frac{ d\bar{w} \wedge dw}{2i}.  \]
It is a well-known fact, originating with Bergman and Schiffer \cite{BergmanSchiffer}, that the bound of one on the norm of this operator implies the Grunsky inequalities for the function $f$ (see e.g. \cite{Schippers_Staubach_CAOT,Schippers_Staubach_Grunsky_expository}).

 Similarly, in the case that $n>1$, pulling back to $\mathbb{D}^n$ via the maps $f_1,\ldots,f_n$ results in the Grunsky operator for multiply-connected domains (see \cite{RSS_Dirichletspace}). 
 
 For a detailed discussion of the literature surrounding the case $n=1$ see {\cite{Schippers_Staubach_Grunsky_expository}}.\\

 {\bf Case II: $g>0$.}  The Grunsky operator in higher genus was defined, and bounds obtained, by M. Shirazi \cite{Shirazi_thesis,Shirazi_Grunsky}, for the case of Dirichlet bounded functions. Here we formulate this in terms of $\mathcal{A}^\mathrm{e}(\riem_2)$, which is of course equivalent up to constants. First, as in \cite{Shirazi_Grunsky} we restrict our attention to the space $(\overline{\mathbf{R}}_1 \overline{\mathcal{A}(\mathscr{R})})^\perp$ and ignore the second component of $\Theta$; that is, we consider
 \[  \Theta' = \left. \Theta \right|_{((\overline{\mathbf{R}}_1 \overline{\mathcal{A}(\mathscr{R})})^\perp \oplus \{0\})}. \]
 In that case, the operator $\Theta$ takes the form 
 \begin{align*}
     \Theta':(\overline{\mathbf{R}}_1 \overline{\mathcal{A}(\mathscr{R})})^\perp & \rightarrow \mathcal{A}^\mathrm{e}(\riem_2) \\
     \overline{\alpha} & \mapsto - \mathbf{T}_{1,2} \overline{\alpha}.
 \end{align*}
 The fact that $\Theta'$ is an isomorphism was obtained by M. Shirazi \cite{Shirazi_thesis,Schippers_Staubach_Shirazi}. We have that the restriction  
 \[  \Upsilon' = \left. \Upsilon \right|_{((\overline{\mathbf{R}}_1 \overline{\mathcal{A}(\mathscr{R})})^\perp \oplus \{0\})} \]
 takes the form 
 \begin{align*}
     \Upsilon': (\overline{\mathbf{R}}_1 \overline{\mathcal{A}(\mathscr{R})})^\perp& \rightarrow \mathcal{A}^\mathrm{e}(\riem_2) \\
     \overline{\gamma} & \mapsto - \mathbf{T}_{1,1} \overline{\gamma}. 
 \end{align*}
 so that once again Theorem \ref{th:upsilon_bounded} implies that
 \[  \| \mathbf{T}_{11} \| <1. \]
 As in the genus zero case, we can define the Grunsky operator 
 \begin{align*}
     \mathbf{Gr}_f : V & \mapsto \bigoplus^n \mathcal{A}(\disk) \\
     \overline{\alpha} & \mapsto \mathbf{P}_{\disk} f^* \mathbf{O}^\mathrm{e}_{2,1} \mathbf{T}_{1,2} (f^{-1})^* \overline{\alpha}.
 \end{align*}
 where 
 \[  V= f^* (\overline{\mathbf{R}}_1 \overline{\mathcal{A}(\mathscr{R})})^\perp \]
 and $f^* = f_1^* \times \cdots \times f_n^*$.
 The Grunsky inequality obtained by M. Shirazi mentioned above is that the norm of $\mathbf{Gr}_f$ is less than one, which follows from $\| \Upsilon \| <1$. By Section \ref{se:KYNS_period} (restricting to exact one-forms), the graph of this Grunsky operator can be interpreted as the set of boundary values of holomorphic functions. See the work of Shirazi \cite{Shirazi_thesis,Shirazi_Grunsky}, for the details.
 
 Here we have not dealt with the deformation theory of Riemann surfaces, since that would require lengthening the paper impractically. The results of this entire paper, and in particular the above discussion, should be placed in the context of Teichm\"uller theory. This would include for example demonstration of the holomorphicity of this period map as well as holomorphicity of its restriction to the Weil-Petersson Teichm\"uller space. The holomorphicity in genus zero for $n$ boundary curves was proven in \cite{RSS_genus_zero}. We hope to deal with the general case, along with a treatment of the symplectic group actions by quasisymmetric reparameterizations, in future publications. 
 
\end{subsection}
\begin{subsection}{The holomorphic boundary value problem} 
\label{se:holomorphic_BVP}
 We motivate the problem, placing analytic issues aside for the moment.\\
 
 {\bf{Problem.}} Given a one-form $\alpha$ on the boundary of $\riem_2$ and a fixed cohomology class on $\riem_2$, is there a holomorphic one-form on $\riem_2$ with boundary values equal to $\alpha$?\\
 
The cohomology class can be fixed by specifying periods, or equivalently any closed one-form in $L^2(\riem_2)$ in that cohomology class. We express the boundary values of the one-form $\alpha$ by parametrizing the boundary by maps $\phi_k:\mathbb{S}^1 \rightarrow \partial_k {\riem}_2$ from the circle to the boundaries .  That is, we look at the boundary parametrization as a kind of coordinate, and pull back the one-form to the circle, and specify the data on $\mathbb{S}^1$.  This data can be viewed as a one-form.   
 
 Adding analytic issues to the picture, assume now that the one-form is in $\mathcal{H}'(\partial \riem)$ and the boundary parametrization is a quasisymmetry.  If it has zero period around its boundaries, then the anti-derivative is an element of $\mathcal{H}(\partial \riem)$, and its pull-back to the disk is an element of $\mathcal{H}(\mathbb{S}^1)$.  In the general case, the original data can be shown to be an element of $\mathcal{H}'(\mathbb{S}^1)$.  
 
 An equivalent picture is as follows.  We sew copies of the disk $\disk^+$ to each boundary curve via quasisymmetries $\phi_1,\ldots,\phi_n$ as in Section \ref{se:KYNS_period} to obtain the surface $\riem_2$ capped by $\riem_1$, with conformal maps $f_k:\disk \rightarrow \Omega_k$ where $\Omega_k$ are the connected components of $\riem_1$.   The data can now be taken to be elements of $\mathcal{H}'(\partial \riem_1)$, and the cohomology class can be specified by an element of $\mathcal{A}_{\mathrm{harm}}(\mathscr{R})$.    
 
 With this motivation, consider the following boundary value problem for holomorphic one-forms. We treat the case that the periods around boundary curves $\partial_k \riem_2$ are zero.
 From this point forward, we make careful analytic definitions and statements.\\
 
 We first state the problem in terms of $H^{-1/2}$ boundary values. 
 \begin{definition}[Holomorphic boundary value problem for semi-exact one-forms with $H^{-1/2}$ data]
 \[  \lambda = (\lambda_1,\ldots,\lambda_{2g}) \in \mathbb{C}^{2g}, \] and let $L \in H^{-1/2}(\partial \riem_2)$.  We say that $\beta \in \mathcal{A}^{\mathrm{se}}(\riem_2)$ solves the holomorphic boundary value problem if it satisfies  
 \[  L_{[\beta]} = L  \]
 and 
 \[ \int_{c_j} \beta = \lambda_j   \]
 for $j=1,\ldots,2g$.
 \end{definition}

 The problem is not well-posed in general.  We will give precise conditions for the existence of a solution momentarily.  
 
First, we reformulate the problem using the theory of Sections 3.6 and 4 of \cite{Schippers_Staubach_scattering_II}.  
 Assume that $\beta$ solves the boundary value problem with respect to the data $\lambda$ and $L$. Assume also that $\delta \in \mathcal{A}_{\mathrm{harm}}(\riem_1)$ is the solution to the $H^{-1/2}$ boundary value problem on $\riem_1$ with respect to $\mathbf{O}'(\partial \riem_2,\partial \riem_1)L_{[\delta]}$. Such a solution is guaranteed to exist by Theorem 3.25 in \cite{Schippers_Staubach_scattering_II}.
 
 applied separately to each connected component of $\riem_1$. Let $\zeta$ be the unique element of  $\mathcal{A}_{\mathrm{harm}}(\mathscr{R})$ with periods 
 \[ \int_{c_j} \zeta = \lambda_j.  \] 
 Then $\delta$ and $\beta$ are weakly compatible with respect to $\zeta$.  
 
 Conversely, if $\delta \in \mathcal{A}_{\mathrm{harm}}(\riem_1)$ and $\beta \in \mathcal{A}(\riem_2)$ are weakly compatible with respect to $\zeta \in \mathcal{A}_{\mathrm{harm}}(\mathscr{R})$ then $\beta$ solves the boundary value problem with data $L=\mathbf{O}'(\partial \riem_1,\partial \riem_2) L_{[\delta]}$.  
 
 Thus we have the following reformulation of the boundary value problem.
 \begin{definition}[Holomorphic CNT Dirichlet BVP for one-forms, semi-exact case]  Let \[  (\delta,\zeta) \in \mathcal{A}_{\mathrm{harm}}(\riem_1) \oplus \mathcal{A}_{\mathrm{harm}}(\mathscr{R}).  \]  We say that $\beta \in \mathcal{A}^{\mathrm{se}} (\riem_2)$ solves the holomorphic boundary value problem with respect to this data if $\delta$ and $\beta$ are weakly compatible with respect to the one-form $\zeta$.  
 \end{definition}

 This allows us to solve the BVP in the following way.  
 \begin{theorem}[Well-posedness of the semi-exact CNT BVP for holomorphic one-forms] \label{th:HBVP_solution_semi-exact}   Let the data $(\delta,\zeta)$ for the semi-exact holomorphic \emph{BVP} be given as above, and
  assume that $\nu,\tau$ are the unique elements of  $\mathbf{R}_1 \mathcal{A}(\mathscr{R})$ such that   $\overline{\mathbf{S}}_1 \overline{\nu} + \mathbf{S}_1 \tau = \zeta$.  
 The semi-exact holomorphic \emph{CNT} Dirichlet \emph{BVP} for forms has a solution with data $(\delta,\zeta)$ if and only if 
 \begin{equation} \label{eq:one_form_BVP_solvability_condition}
   \left[ \delta - \mathbf{R}_1 \mathbf{S}_1  \tau \right]  \in  \mathrm{Im}  [\mathbf{I} - \mathbf{T}_{1,1}].
 \end{equation}
  If this solution exists, it is unique and equals 
  \[  \beta =  -\mathbf{T}_{1,2} \overline{\gamma} + \mathbf{R}_2  \mathbf{S}_1 \tau \]
  where $\overline{\gamma} \in \overline{\mathcal{A}(\riem_1)}$ is the unique one-form such that 
  \[  \overline{\gamma} - \mathbf{T}_{1,1} \overline{\gamma} = \delta - \mathbf{R}_1  \mathbf{S}_1  \tau.  \]
  The component of this unique $\overline{\gamma}$ in $\overline{\mathbf{R}}_1 \overline{\mathcal{A}(\mathscr{R})}$ is $\overline{\nu}$.   Furthermore the solution depends continuously on the initial data.
 \end{theorem}
 \begin{proof}  
   Assume that there exists a solution $\beta \in \mathcal{A}(\riem_2)$. Then 
   \[    \beta - \mathbf{R}_2 \mathbf{S}_1 \tau - \overline{\mathbf{R}_2} \overline{\mathbf{S}}_1 \overline{\nu} \in \mathcal{A}^\mathrm{e}_{\mathrm{harm}}(\riem_2)  \]
  so by \cite[Corollary 4.7]{Schippers_Staubach_scattering_III}
  \[   \beta  - \mathbf{R}_2 \mathbf{S}_1  \tau + \mathbf{T}_{1,2} \overline{\nu}  \in \mathcal{A}^\mathrm{e}(\riem_2).  \]
  Thus by \cite[Theorem 4.11]{Schippers_Staubach_scattering_III} there is a unique $\overline{\alpha} \in [\mathbf{R}_1 \overline{\mathcal{A}(\mathscr{R})} ]^\perp$ such that 
  \[  - \mathbf{T}_{1,2} \overline{\alpha} = \beta - \mathbf{R}_1 \mathbf{S}_1 \tau + \mathbf{T}_{1,2} \overline{\nu}.       \]
  This implies that 
  \[  \beta- \mathbf{R}_2 \mathbf{S}_1 \tau = - \mathbf{T}_{1,2} [ \overline{\alpha} + \overline{\nu} ].     \]
  Since 
  \[  -\mathbf{O}^\mathrm{e}  \mathbf{T}_{1,2} [ \overline{\alpha} + \overline{\nu} ] =  \overline{\alpha} + \overline{\nu} - \mathbf{T}_{1,1} \left( \overline{\alpha} + \overline{\nu} \right)   \]
  and so
  \[  \mathbf{O}'(\partial \riem_2,\partial \riem_1) \left[ \beta - \mathbf{R}_2 \mathbf{S}_1 \tau \right]  = 
    \left[ \delta - \mathbf{R}_1 \mathbf{S}_1 \tau \right], \]
  this proves that $\delta = \overline{\gamma} - \mathbf{T}_{1,1} \overline{\gamma}$ for $\overline{\gamma} = \overline{\alpha} + \overline{\nu}$ and furthermore establishes that the solution has the claimed form.  Uniqueness follows from \cite[Theorem 3.18]{Schippers_Staubach_scattering_II}, observing that the solution is also the solution to the Dirichlet problem with the specified data. 
  
  Conversely, assume that 
  \begin{equation} \label{eq:HBVP_proof_temp} 
   \delta - \mathbf{R}_1 \mathbf{S}_1 \tau = \overline{\gamma} - \mathbf{T}_{1,1} \overline{\gamma}
  \end{equation} 
   for some $\overline{\gamma} \in \overline{\mathcal{A}(\mathscr{R})}$. Let $\overline{\gamma} = \overline{\alpha} + \overline{\nu}$ be the decomposition of $\overline{\gamma}$ with respect to $\overline{\mathcal{A}(\riem_1)} = \overline{\mathbf{R}_1} \overline{\mathcal{A}(\mathscr{R})} \oplus [\overline{\mathbf{R}_1} \overline{\mathcal{A}(\mathscr{R})}]^\perp$.  Then we claim that 
  $\beta  = \mathbf{R}_2 \mathbf{S}_1 \tau - \mathbf{T}_{1,2} \overline{\gamma}$ satisfies 
  $[ \beta] = [ \delta ]$ and has the correct periods. 
  
  To see that $\beta$ has the correct periods, 
  observe that since $\overline{\alpha} \in [\overline{\mathbf{R}_1} \overline{\mathcal{A}(\mathscr{R})} ]^\perp$, $\overline{\mathbf{S}}_1 \overline{\alpha} = 0$ so
  by \cite[Theorem 4.7]{Schippers_Staubach_scattering_III}
  \[  - \mathbf{T}_{1,2} \overline{\gamma} - \overline{\mathbf{R}}_2 \overline{\mathbf{S}}_1 \overline{\nu} = 
     - \mathbf{T}_{1,2} \overline{\gamma} - \overline{\mathbf{R}}_2 \overline{\mathbf{S}}_1 \overline{\gamma}      \]
  is exact, and therefore $- \mathbf{T}_{1,2} \overline{\gamma} + \mathbf{R}_1 \mathbf{S}_1  \tau$ has the 
  specified periods.  To see that the boundary values of $\beta$ are the right ones, we observe that
  \[  \beta = \overline{\mathbf{R}}_2 \overline{\mathbf{S}}_1 \overline{\gamma}
     + \mathbf{R}_2 \mathbf{S}_1 \tau + [ - \mathbf{T}_{1,2} \overline{\gamma} - \overline{\mathbf{R}}_2 \overline{\mathbf{S}}_1 \overline{\gamma} ]  \]
  and then apply Proposition 4.10 in \cite{Schippers_Staubach_scattering_III}.
  to overfare the quantity in brackets, to show that 
  \begin{equation}  \label{eq:decomposition_foundation}
   [ \beta ] = \left[ \overline{\mathbf{R}}_1 \overline{\mathbf{S}}_1 
   \overline{\gamma} + \mathbf{R}_1 \mathbf{S}_1 \tau \right]  + \left[
     \overline{\gamma} - \mathbf{T}_{1,1} \overline{\gamma} - \overline{\mathbf{R}}_1
      \overline{\mathbf{S}}_1 \overline{\gamma} \right] = [ \delta ]   
  \end{equation}
  where we have used (\ref{eq:HBVP_proof_temp}) in the second equality.\\
  
  Finally we show continuous dependence of the solution on the data.
  Let $ \delta - \mathbf{R}_1 \mathbf{S}_1  \tau  \in  \mathrm{Im} (\mathbf{I} - \mathbf{T}_{1,1}).$ Then $\delta - \mathbf{R}_1 \mathbf{S}_1  \tau = (\mathbf{I} - \mathbf{T}_{1,1})\overline{\gamma}$ and $$\overline{\gamma}= \overline{\mathbf{P}_{\riem_1}}(\delta - \mathbf{R}_1 \mathbf{S}_1  \tau )= \overline{\mathbf{P}_{\riem_1}}\delta.$$ Therefore
  
  \begin{equation}\label{foersta normen}
    \Vert\overline{\gamma} \Vert = \Vert \overline{\mathbf{P}_{\riem_1}}\delta \Vert \lesssim  \Vert \delta\Vert.
  \end{equation}
  Furthermore 
   \begin{equation}\label{andra normen}
    \Vert\tau \Vert \leq  \Vert \tau + \overline{\nu} \Vert \leq  \Vert ({\mathbf{R}}_1
      {\mathbf{S}}_1)^{-1} \zeta\Vert \lesssim  \Vert \zeta\Vert.
  \end{equation} 
  
Thus \eqref{foersta normen} and \eqref{andra normen} and the boundedness of $\mathbf{T}_{1,2} $ and $\mathbf{R}_2  \mathbf{S}_1$ yield that
  \begin{equation*}
    \Vert\beta \Vert = \Vert -\mathbf{T}_{1,2} \overline{\gamma} + \mathbf{R}_2  \mathbf{S}_1 \tau\Vert \leq  \Vert \mathbf{T}_{1,2} \overline{\gamma}\Vert + \Vert  \mathbf{R}_2  \mathbf{S}_1 \tau\Vert \lesssim \Vert \overline{\gamma}\Vert + \Vert \tau\Vert\lesssim \Vert \delta\Vert + \Vert \zeta\Vert,
  \end{equation*}
which shows the continuous dependence of the solution $\beta$ on the initial data $(\delta, \zeta).$ Thus the semi-exact CNT BVP is well-posed in the Bergman space of forms {satisfying condition \ref{eq:one_form_BVP_solvability_condition}. }
 \end{proof}   
\end{subsection}
\end{section}

\clearpage

\printnoidxglossary[sort=def]

\end{document}